\documentclass[11pt,a4paper]{article}
\usepackage{amsmath}
\usepackage[affil-it]{authblk}
\usepackage{amsthm}
\usepackage{float}
\usepackage{soul}
\usepackage{amssymb}
\usepackage{latexsym, url}
\usepackage{epsfig}
\usepackage{hhline}
\usepackage{amscd}
\usepackage{enumerate}
\usepackage{color}
\usepackage{geometry}
\usepackage{tabularx}
\usepackage{multirow}
\newtheorem{MainTheorem}{Theorem}

\newtheorem{Lemma}      {Lemma} [section]
\newtheorem{Theorem}    [Lemma] {Theorem}

\newtheorem{Corollary}  [Lemma] {Corollary}
\newtheorem{Proposition}[Lemma] {Proposition}
\newtheorem*{Conjecture}  {Conjecture}

\theoremstyle{definition}
\newtheorem{Definition} [Lemma] {Definition}

\numberwithin{equation}{section}

\newcommand{\dl}{\mathop{\mathrm{sl}}}

\newcommand{\Sym}{{\mathrm{Sym}}}
\newcommand{\Sn}{{\mathrm{S}}}

\newcommand{\Sub}{\mathrm{Sub}}
\newcommand{\Subdir}{\mathrm{Subdir}}
\newcommand{\NormSub}{\mathrm{NormSub}}
\newcommand{\BBF}{\mathbb{F}} 
\newcommand{\GL}{\mathrm{GL}}
\newcommand{\AGL}{\mathop{\mathrm{AGL}}}

\newcommand{\ma}{\mathop{\mathrm{ma}}}

\newcommand{\Aut}{{\mathrm{Aut}}}

\newcommand{\Soc}{\mathop{\mathrm{soc}}}

\newcommand{\Hom}{\mathrm{Hom}}
\newcommand{\Epi}{\mathrm{Epi}}

\newcommand{\cl}{\mathrm{cl}}

\newcommand{\ol}{\overline}

\newcommand{\Gab}{G^{\mathrm{ab}}}
\newcommand{\GG}{G'}
\newcommand{\supp}{\mathrm{supp}}

\def\cn{\mathord{{\!\:{:}\:\!}}} 

\newcommand{\Magma}{\texttt{Magma} }
\newcommand{\Magman}{\texttt{Magma}}

\newcommand{\cSix}{\kappa_4} 
\newcommand{\cSevenA}{\kappa_3} 
\newcommand{\cSeven}{\kappa_1} 
\newcommand{\cEight}{\kappa_5} 
\newcommand{\cNine}{\kappa_2} 
\newcommand{\cTen}{\kappa_6} 
\newcommand{\cTwelve}{\kappa_{8}}
\newcommand{\cTwenty}{\kappa_7}

\begin{document}
\title{Subgroups of symmetric groups: enumeration and asymptotic properties}
\author{Colva~M.~Roney-Dougal and Gareth~Tracey}
\date{\today}
\maketitle

\begin{abstract}
In this paper, we prove that the symmetric group $\Sn_n$ has $2^{n^2/16+o(n^2)}$ subgroups, settling a conjecture of Pyber from 1993. We also derive asymptotically sharp upper and lower bounds on the number of subgroups of $\Sn_n$ of various kinds, including the number of $p$-subgroups. In addition, we prove a range of theorems about random subgroups of $\Sn_n$. In particular, we prove the surprising result that for infinitely many $n$, the probability that a random subgroup of $\Sn_n$ is nilpotent is bounded away from $1$.

\smallskip

\noindent MSC areas: 20B35, 20F69, 05A16, 20E07, 20E25.
\end{abstract}

\section{Introduction}\label{sec:Intro}

In this paper, we count the subgroups of the symmetric group on $n$ points, and determine many group theoretic properties of a random such subgroup. In doing so, we resolve conjectures of Kantor and Pyber from the early 1990s.

Many enumeration problems in algebra and combinatorics can be reduced to the problem of counting the subgroups of some associated group, either absolutely or up to some form of equivalence.  
A prime example is Galois theory: by enumerating subgroups (or certain types of subgroups) of $\Sn_n$, one gets precise information on the number of intermediate fields in a field extension of degree $n$.

As a second example, the number of labelled vertex transitive graphs on $n$ vertices is at most the product of $2^{n-1}$ and the number of minimal transitive subgroups of the symmetric group $\Sn_n$. This observation led Babai and S\'{o}s to the current best known upper bounds on the number of such graphs \cite{BS}.

Motivated by a stunning application to the enumeration problem for isomorphism classes  of 
groups of a given order, Pyber proved in 1993 \cite{PybAnnals}  that for a positive integer $n$, the number $|\Sub(\Sn_n)|$ of subgroups of $\Sn_n$ is at most $2^{\xi n^2+o(n^2)}$ where $\xi=\frac{1}{6}\log{24}$ (our logarithms  are 
to the base $2$). 
Pyber conjectured, however, that
$|\Sub(\Sn_n)|= 2^{n^2/16+o(n^2)}$.
In this paper, we prove a strong form of Pyber's conjecture.
\begin{MainTheorem}\label{thm:PyberConj}
 There exist absolute constants $\alpha > 0$ and $\beta$ such that for all integers $n > 1$
    \[2^{n^2/16+\alpha n\log{n}}  \le |\Sub(\Sn_n)| \le 2^{n^2/16 + \beta n^{3/2}}.\]
\end{MainTheorem}

As a by-product of our methods, we prove considerably tighter bounds on the number $|\Sub_p(\Sn_n)|$ of $p$-subgroups of $\Sn_n$.
\begin{MainTheorem}\label{thm:Pyberp}
 Let $p$ be  a prime. There exist constants  $\beta_p > \alpha_p > 0$ such that for all integers $n \ge p$, 
  \[p^{n^2/(4p^2)}2^{\alpha_p n \log n } \le |\Sub_p(\Sn_n)|\le p^{n^2/(4p^2)}2^{\beta_p n\log{n}}.\] Furthermore, if $p 
  \ge 11$ then $\beta_p  \le  3/2$. 
\end{MainTheorem}

Notice that since there are most $n! < 2^{n \log n}$ conjugates in $\Sn_n$ of a given subgroup, both of these theorems immediately transfer to  conjugacy classes of subgroups, by making the second coefficients in the lower bounds smaller.

For a property $\mathcal{P}$, we write $\Sub_{\mathcal{P}}(G)$ for the set of subgroups of $G$ satisfying $\mathcal{P}$. For a prime $p$, let $\mathcal{N}_p$ be the property of being nilpotent, and of 
order divisible by no primes less than $p$.

\begin{MainTheorem}\label{thm:Pybernilp}
Let $p$ be a prime. Then there exists an absolute constant $\gamma(p)$ such that
\[|\Sub_{\mathcal{N}_p}(\Sn_n)| \le p^{n^2/(4p^2)}2^{\gamma(p) n \log n}.\] In particular the number of nilpotent subgroups of $\Sn_n$ satisfies
\[|\Sub_{\mathrm{nilp}}(\Sn_n)| \le 2^{n^2/16 + \gamma(2) n \log n}.\]
\end{MainTheorem}

 We remark that it was recently proved by Kassabov, Tyburski and Wilson \cite{KassabovTyburskiWilson} that $\mathrm{A}_n$ (and hence also $\Sn_n$) has at least $2^{n^2/36 - o(n^2)}$ \emph{isomorphism classes} of subgroups.

\bigskip

One of our motivations for enumerating subgroups of $\Sn_n$ is an immediate application to computational group theory. At present, the analysis of algorithms for permutation groups is generally only of worst-case performance, whilst in many other disciplines the study of average case or generic-case performance is widespread. The reason for this deficiency is our lack of understanding up to now of the properties of a random permutation group. The present article is a first step in remedying this  problem.

By saying that a \emph{random subgroup of $\Sn_n$ has property $\mathcal{P}$}, we mean that the probability that a uniform random subgroup of $\Sn_n$ satisfies $\mathcal{P}$ tends to $1$ as $n$  tends to 
$\infty$. For example, Lucchini, Menegazzo and Morigi proved in \cite{LucMenMor00} that there exists an absolute constant $b$ such that the number of transitive subgroup of $\Sn_n$ is at most $2^{bn^2/\sqrt{\log n}}$, and hence deduced that a random subgroup of $\Sn_n$ is intransitive.

Nilpotency is thought to be a prevalent property amongst many natural sequences of finite groups. Indeed, a conjecture attributed to Erd\H{o}s, amongst others, states that the number of isomorphism classes of groups of order $2^x$ is greater than or equal to the number of isomorphism classes of groups of order $n$ for all $n\le 2^x$. In \cite{PybAnnals}, Pyber goes further by conjecturing that a finite group chosen uniformly amongst the groups of order up to $n$ is nilpotent, while in \cite{Kantor}, Kantor states the following permutation group theoretic version:
\begin{Conjecture}[Kantor's conjecture]
A random subgroup of $\Sn_n$ is nilpotent.
\end{Conjecture}
 Kantor remarks in \cite{Kantor} that if this conjecture is true, then the study of average-case complexity for permutation groups is unlikely to prove insightful. In this paper, we disprove Kantor's conjecture.

\begin{MainTheorem}\label{thm:not_Kantor}
  Let $n$ be congruent to $3$ modulo $4$. Then the probability that a random subgroup of $\Sn_n$ is nilpotent is bounded away from $1$, 
  as $n \rightarrow \infty$. 
\end{MainTheorem}

As an example of how the tools we develop can be used to study the distribution of subgroups of $\Sn_n$, we also prove two further results about random subgroups.

\begin{MainTheorem}\label{thm:rand_nilp}
    A random nilpotent subgroup of $\Sn_n$ is a $2$-group.  
\end{MainTheorem}

\begin{MainTheorem}\label{thm:randSyl2}
    Let $\mu$ be a real number in $[0, 1/16)$ and let $\nu$ be a real number in $[0, \frac{1}{2} - \frac{\sqrt{3}}{4})$. Then a random subgroup of $\Sn_n$ has an elementary abelian $2$-section of order at least $2^{\mu n}$, and a  Sylow 2-subgroup of order at least $2^{\nu n}$.
\end{MainTheorem}

\medskip

Aschbacher and Guralnick prove in \cite{AschGur} the highly influential result that any finite group can be generated by a soluble subgroup and one other element. This has 
 been widely applied to reduce
questions concerning all finite groups to questions concerning only finite soluble groups. A natural follow-up question therefore is to ask what other similar reduction results are available. In particular, is it true that a finite group can be generated by a nilpotent subgroup, together with some ``small" subset of other elements? This question was addressed in \cite{AschGur}, where a construction is given of infinite families of finite groups which cannot be generated by two nilpotent subgroups. 

Despite this, we prove that one can control the number of additional elements required to generate an arbitrary finite group with a nilpotent subgroup. To state our result, we need two  
definitions.  
We set $\Soc^0(G):=1$, and for $i\geq 1$, define $\Soc^i(G)$ so that $\Soc^i(G)/\Soc^{i-1}(G)=\Soc(G/\Soc^{i-1}(G))$. The \emph{socle length} $\dl(G)$ of  a finite group $G$ is the minimal $i\in\mathbb{N}$ such that $\Soc^i(G)=G$. Thus, for example, 
if $G$ is a finite $p$-group then the nilpotency class $\mathrm{cl}(G)\le \dl(G)\le \log_p{|G|}$.
For $n\in\mathbb{N}$, we write $\Omega(n)$ for the number of prime divisors of $n$, counted with multiplicities. 
For a finite group $G$ we define
  \[\ma(G) = \max\{\Omega{(|A|)}\text{ : }A\text{ an abelian section of }G\}.\] 
Our result can now be stated as follows.

\begin{MainTheorem}\label{thm:NilpotentRedTheorem}
Let $G$ be a finite soluble group. Then $G$ can be generated by a nilpotent subgroup together with $4 \dl(G) \sqrt{\ma(G)}$ other elements. Therefore, if $H$ is an arbitrary finite group then $H$ can be generated by a nilpotent subgroup and $4 \dl(H) \sqrt{\ma(H)} + 1$ elements. 
\end{MainTheorem}

We finish these introductory remarks with three questions. Firstly, 
does there in fact  exist an absolute constant $\gamma$ such that $|\Sub(\Sn_n)|\le 2^{n^2/16+\gamma n\log{n}}$? Secondly, is it possible that  the probability that a random subgroup of $\Sn_n$ is nilpotent tends to $0$ as $n \rightarrow \infty$? Thirdly, does there exist an absolute constant $C$ such that a random subgroup of $\Sn_n$ has no orbits of length greater than $C$? Whilst this third question at first sounds improbable, we present evidence in Section~\ref{sec:Bounded} that it may have a positive answer.

The structure of the paper is as follows. In Section~\ref{sec:Goursat}, we collect notation that will be used throughout the paper 
and present various elementary results. Then we begin our detailed study of subdirect products:
in Section~\ref{sec:SubdirectCounting} we present some new, general results on enumerating subgroups of subdirect products 
of finite groups, 
in Section~\ref{sec:CountingHoms} we develop new tools for counting homomorphisms between finite $p$ groups, and in Section~\ref{sec:GenSec} we use a mixture of computational and theoretical tools to bound normal generator numbers for  $p$-subgroups of $\Sn_n$. 
In Section~\ref{sec:Bounded} we
prove an intriguing result that to enumerate various classes of subgroups of $\Sn_n$, it suffices to consider those with all orbits of bounded length. 
Then we use this in Section~\ref{sec:PyberProof} to prove Theorem~\ref{thm:Pyberp}, with Theorem~\ref{thm:Pybernilp}  as an easy corollary. In Section~\ref{sec:NilpotentRed} we prove Theorem~\ref{thm:NilpotentRedTheorem}, and then immediately use it to prove Theorem~\ref{thm:PyberConj}, and hence Theorem~\ref{thm:randSyl2}. In our final section we first prove Theorem~\ref{thm:rand_nilp}, then Theorem~\ref{thm:not_Kantor} follows surprisingly easily.

\section{Background: Goursat's lemma, generation and commutators}\label{sec:Goursat}

In this section we shall present various elementary results, and fix much of the notation that will be used throughout the paper. 

\subsection{Subdirect products and Goursat's Lemma}

The bulk of our work will be concerned with counting intransitive subgroups of $\Sn_n$,
 which will essentially amount to counting certain subgroups of direct products. A subgroup $G$ of $D = G_1\times\cdots\times G_t$ is \emph{subdirect} if $G\pi_i=G_i$ for all $i$, where $\pi_i:G\rightarrow G_i$ denotes the standard coordinate projection. When dealing with direct products like this, if  the context leaves no room for confusion then we will identify $G_i$ with its natural copy $\{(g_1,\hdots,g_t)\text{ : }g_j=1\text{ whenever }j\neq i\}$ in $D$. Thus, we will often speak of $G_i$ as a normal subgroup of $D$.

We 
start with a standard lemma on the structure of subdirect products. By a \emph{diagonal} subgroup of a direct product $K=K_1\times\cdots\times K_{\ell}$ of pairwise isomorphic groups $K_i$, we mean a subgroup of the form $J=\{(a,a\theta_2,\hdots,a\theta_{\ell})\text{ : }a\in K_1\}$, where $\theta_i:K_1\rightarrow K_i$ are isomorphisms.

\begin{Lemma}\label{lem:subdir} 
    Let $G$ be a subdirect product of $G_1 \times G_2$. Then there exist normal subgroups $N_1$ and $N_2$ of $G_1$ and $G_2$ such that $G \cap G_i = N_i$. Furthermore, $G_1/N_1 \cong G_2/N_2$, and $G/(N_1 \times N_2)$ is isomorphic to a diagonal subgroup of $G_1/N_1 \times G_2/N_2$.
\end{Lemma}

We now fix some key notation that will be used throughout the paper.

\begin{Definition}\label{defn:subdir_basics}
    For group theoretic properties $\mathcal{P}$ and $\mathcal{Q}$ we write $\Sub_{\mathcal{P}, \mathcal{Q}}(G)$ for the set of subgroups of $G$ satisfying $\mathcal{P}$ and $\mathcal{Q}$. 
    Similarly, 
    $\Subdir_{\mathcal{P}}(G_1\times\cdots\times G_t)$ will denote the set of subdirect products of $G_1\times \cdots\times G_t$ that satisfy $\mathcal{P}$. We write  $\mathrm{nilp}$ and $\mathrm{sol}$ for nilpotency and solubility, respectively.  For $G \in \Subdir(G_1\times\cdots\times G_t)$, we shall write 
    $\pi_{i_1, \ldots, i_k}$ for the projection to $G_{i_1} \times \cdots \times G_{i_k}$. 
\end{Definition}

The primary mechanism for counting subgroups of direct products is Goursat's lemma, a version of which we now state.
For finite groups $G$ and $H$,  we write  $\Hom(G,H)$ to denote the set of homomorphisms from $G$ to $H$, and  $\Epi(G,H)$ for the set of epimorphisms. 

\begin{Lemma}[\textbf{Goursat's lemma}]\label{lem:Goursat}
Let $G_1$ and $G_2$ be finite groups. Then the following hold.
\begin{enumerate}[\upshape(i)]
\item \label{goursat_property} Let $\mathcal{P}$ be a group theoretic property. Then the number of subgroups $G$ of $G_1\times G_2$ in which 
  $G\pi_2$ has property $\mathcal{P}$ is precisely   
$$ 
\sum_{H_1\in\Sub(G_1),H_2\in\Sub_{\mathcal{P}}(G_2)}|\Hom(H_2,N_{G_1}(H_1)/H_1)|.$$ 
Therefore, if $\mathcal{P}$ is such that whenever a subgroup of $G_1 \times G_2$ satisfies $\mathcal{P}$, so does its projection to $G_2$, then 
$|\Sub_{\mathcal{P}}(G_1 \times G_2)|$ is at most 
\[|\Sub(G_1)||\Sub_{\mathcal{P}}(G_2)| \max\{|\Hom(H_2, N_{G_1}(H_1)/H_1)| \text{ : } H_1 \in \Sub(G_1), H_2 \in \Sub_\mathcal{P}(G_2)\}.
\] 
\item \label{goursat_kernel}  Let $N_1$ be a normal subgroup of $G_1$. Then the number of subdirect products $G$ of $G_1\times G_2$ such that $G\cap G_1=N_1$ is precisely 
$|\Epi(G_2,G_1/N_1)|.$
\end{enumerate}
\end{Lemma}

For groups $G$ and $H$, our first upper bound for  $|\Hom(G,H)|$ is $|H|^{d(G)}$. 

\subsection{Basic bounds on subgroup enumeration and generation}

Let $G$ be a finite group.  By a \emph{$G$-group} we mean a group on which $G$ acts via automorphisms. For such a $G$-group $N$ we will write $d_G(N)$ for the minimal number of elements needed to generate $N$ as a $G$-group, and $d(G) = d_G(G)$ for the generator number of $G$.  

If every subgroup of $G$ can be generated by $s$ elements, then  $|\Sub(G)| \le |G|^s$. 
The following lemma gives improved bounds when  $G$ is a $p$-group, together with one related estimate.

\begin{Lemma}\label{lem:AnerCount}
    Let $P$ be a $p$-group of order $p^\ell$ with $d(P) = d$. Then 
    \begin{enumerate}
        \item[\upshape(i)] $P$ has at most $\zeta_p p^{k(\ell-k)}$ subgroups of order $p^k$, where $\zeta_p = \prod_{i \ge 1} (1-p^i)^{-1} < 4$. In particular, $|\Sub(P)| \le 4 \ell p^{\ell^2/4}$.
        \item[\upshape(ii)] If  $b\le d$ then $P$ has at least $p^{b(d-b)}$ subgroups of index $p^b$.
        \item[\upshape(iii)] For all $d \ge 1$ we can bound $|\GL_d(2)| \ge 2^{d^2-d}$. 
    \end{enumerate}
\end{Lemma}

\begin{proof} (i). The first statement is {\cite[Lemma 4.2]{Shalev92}}, and the second then follows easily. 

\smallskip

\noindent (ii). Since $P/\Phi(P)\cong \mathbb{F}_p^d$,  we may assume that 
$P$ is an $\mathbb{F}_p$-vector space of dimension $d$. Then the number of subspaces of $P$ of co-dimension $b$ is
$$\frac{(p^d-1)(p^d-p)\hdots(p^d-p^{b-1})}{(p^b-1)(p^b-p)\cdots(p^b-p^{b-1})}.$$
A routine exercise shows that this is at least $p^{b(d-b)}$.

\smallskip

\noindent (iii) This follows from $|\GL_d(2)| = 2^{d(d-1)/2} \prod_{i = 1}^d(2^i-1) \ge 2^{d(d-1)/2}\prod_{i = 1}^d 2^{i-1}$. 
\end{proof}

The following lemma is elementary, but will be used throughout Sections~\ref{sec:CountingHoms} and \ref{sec:GenSec}.
\begin{Lemma}\label{lem:normal_gen_pgps}
Let $G$ be a finite $p$-group, and let $N$ be a normal subgroup of $G$.
\begin{enumerate}
\item[\upshape(i)]   Suppose that $G$ acts on a finite $p$-group $H$ via automorphisms. The subset $Y$ of $H$ generates $H$ as a G-group if and only if $\langle Y\rangle[G,H]=H$. In particular,   $d_G(H)=\dim_{\mathbb{F}_p}{(H/H^p[G,H])}$ and 
$d_G(N) \le \log_p(|N/[G, N]|)$. 

\item[\upshape(ii)]   Suppose that $G$ is a transitive subgroup of $\Sn_n$ for some $n \neq p$, let $\Sigma$ be a minimal block system  for $G$, let $\pi$ be the homomorphism from $G$ to $\Sym(\Sigma)$, and let $S = G\pi $. 
Then $d_G(N) \le d_S(N \cap \ker(\pi)) + d_S(N\pi)$. 
\end{enumerate}
\end{Lemma}

\subsection{Nilpotent subgroups and abelian sections of $\Sn_n$}\label{subsec:nilp}

For a nilpotent group $G$, we write $G_p$ for the Sylow $p$-subgroup of $G$ and $G_{p'}$ for the Hall $p'$-subgroup of $G$.

\begin{Definition}\label{def:gridded}
    Let $\Sigma$ and $\Gamma$ be partitions of a set $\Omega$. We say that $\Sigma$ and $\Gamma$ are \emph{gridded} if each part of $\Sigma$ meets each part of $\Gamma$ in a unique point. We say that they are \emph{$p$-gridded} for some prime $p$ if each part of $\Sigma$ has size the highest power of $p$ dividing $|\Omega|$ (which may be $p^0 = 1$). 
\end{Definition}

Notice that if $\Sigma$ and $\Gamma$ are gridded then all parts of $\Sigma$ have size the number of parts of $\Gamma$, and vice versa. 
For a partition $\Omega=\Delta_1\sqcup\hdots\sqcup \Delta_k$, the corresponding \emph{Young subgroup} of $\Sym(\Omega)$ is the direct product $\Sym(\Delta_1) \times \cdots \times \Sym(\Delta_k)\le\Sym(\Omega)$.

\begin{Lemma}\label{lem:nilp_degree}
Let $G$ be a transitive nilpotent subgroup of $\Sym(\Omega)$, and write $|\Omega| = n = p^kr$ with $p$ a prime and $p \nmid r$. We assume that $r > 1$.
     Let $\Gamma_p$ be the set of $G_p$-orbits and $\Gamma_r$ be the set of $G_{p'}$-orbits.
      Then 
 \begin{enumerate}
     \item[\upshape(i)] $\Gamma_p$ and $\Gamma_{r}$ are $p$-gridded partitions of $\Omega$, and each part in $\Gamma_p$ has size $p^k$. 
    \item[\upshape(ii)] Let  $Y_p$ and  $Y_r$ be the Young subgroups of $\Sym(\Omega)$ corresponding to $\Gamma_p$ and $\Gamma_r$.  Then $Y_p$ has a unique diagonal subgroup $D_p =D_p(\Gamma_p, \Gamma_r) \cong \Sn_{p^k}$ such that $[D_p,G_{p'}]=1$, and similarly  $Y_r$ has a unique diagonal subgroup $D_r =D_r(\Gamma_p, \Gamma_r) \cong \Sn_r$ such that $[ D_r,G_p] = 1$. In particular,  $G \le D_p \times D_r  \cong \Sn_{p^k} \times \Sn_{r}$, and $D_p$ and $D_r$ depend only on $\Gamma_p$ and $\Gamma_r$. 
    \item[\upshape(iii)] There is a unique conjugacy class of maximal transitive nilpotent subgroups of $\Sn_n$. Writing $n = p_1^{k_1} \cdots p_s^{k_s}$, these subgroups have order the product of the orders of the Sylow $p_i$-subgroups of $\Sn_{p_i^{k_i}}$, which is at most  $2^{n-1}$.  
\end{enumerate}
\end{Lemma}

\begin{proof}
(i) and (ii). This is a straightforward exercise, since all orbits of a $p$-group have $p$-power order, and  if a group $H_1 \le \Sym(\Omega)$ centralises a group $H_2 \le \Sym(\Omega)$, then $H_1$ permutes the orbits of $H_2$. 

\smallskip

\noindent (iii). It follows from Part (ii) that the maximum order of a transitive nilpotent subgroup of $\Sn_n$ is the product of the orders of the Sylow $p_i$-subgroups of  $\Sn_{p_i^{k_i}}$, namely $p_i^{(p_i^{k_i} - 1)/(p-1)}$. Now $p^{(m-1)/(p-1)} < 2^{m-1}$ for $p > 2$, and $2^{m_1-1} \cdot 2^{m_2-1}  < 2^{m_1 m_2  - 1}$  so the maximum is attained when $n$ is a power of $2$. 
\end{proof}

We shall often write $G'$ for the derived group $[G, G]$, and $\Gab$ for the quotient $G/G'$. 
\begin{Theorem}[\cite{KovPrae}]\label{thm:KovPrae}
    Let $G$ be a subgroup of $\Sn_n$ with $G\neq G'$. Then there is a prime $p$ dividing $|\Gab|$ such that $|\Gab| \le p^{n/p}$. In particular, $|\Gab| \le 3^{n/3}$. 
\end{Theorem}

\section{Counting subdirect products of finite $p$-groups}\label{sec:SubdirectCounting}

Throughout this section, let $D = G_1 \times \cdots \times G_t$. 
The bulk of our work  will come down to counting subdirect products of such groups $D$. 
In this section we introduce two related data structures called \emph{coordinate tableaux}, which are a compact way of representing structural information about subdirect products.  After that, we define \emph{efficient} generating sets for subdirect products, and give a lower bound on the number of such sets.  We then use these tools to prove the main result of this section, Theorem~\ref{thm:InitialSubdirect}.

\begin{Definition}\label{defn:N(l,H)}
A normal series $N_1\le\cdots\le N_{\ell} = H$ of a group $H$ is \emph{proper} if $N_i<N_{i+1}$ for all $i$. For us, the \emph{length} of the series is $\ell$ (\emph{not} $\ell-1$).
For all $\ell \in \mathbb{Z}_{>0}$ and all groups $H$, we let $\mathcal{N}(\ell,H)$ be the set of proper normal series of  $H$ of length $\ell$.
\end{Definition}

\begin{Theorem}\label{thm:InitialSubdirect}
Suppose that $D$ is a $p$-group, for some prime $p$. 
Let $c \in (0, 1]$ be a constant, let $n_1,\hdots,n_t$  be real numbers such that $|G_i|\le p^{cn_i}$ for all $i \in \{1, \ldots, t\}$, and let $n = \sum_{i = 1}^t n_i$. 
Let $k$ be an upper bound on the length of a proper normal series in the groups $G_i$, 
let $\lambda$ be the maximum of $|\mathcal{N}(\ell,G_i)|$ over all $\ell$ and $i$,
and let  $d = \max\{d_{G_i}(N)/n_i \text{ : } N \unlhd G_i, 1 \le i \le t\}$. 
Then \[|\Subdir(G_1\times\cdots\times G_t)|\le \lambda^tt!^{k+2}p^{cd n^2/4}.\] 
In particular, for all $C > 1$, there exists a constant $\cSeven = \cSeven(p,C) \ge 1$ such that if all $n_i$ are less than $C$ and the $G_i$ are quotients of $p$-subgroups of $\Sn_{n_i}$, then 
\[|\Subdir(G_1\times\cdots\times G_t)|\le p^{dn^2/(4(p-1))}2^{\cSeven t \log t}.\]  
\end{Theorem}

We now introduce the first of our data structures.

\begin{Definition}\label{def:tnst}
A \emph{lower triangular normal subgroup tableau} for $(G_1,\hdots,G_t)$ is a $t \times t$ array $T = (T_{ij})$ with entries as follows:  
\begin{enumerate}[\upshape(i)]
 \item if $i \ge j$ then $T_{ij}$ is a normal subgroup of $G_i$ such that $T_{i1}\le T_{i2} \le \cdots\le T_{ii}$, and 
 \item if $i < j$ then $T_{ij} = -$.
 \end{enumerate} 
  An \emph{upper triangular normal subgroup tableau} $U = (U_{ij})$ is defined similarly, except that $U_{ij} = -$ for $i > j$ and $U_{ii} \le U_{i(i+1)} \le \ldots \le U_{it}$. 
 
 We denote the set of all upper and lower triangular normal subgroup tableaux for $(G_1, \ldots, G_t)$ by $\mathcal{U}(G_1, \ldots, G_t)$ and  $\mathcal{L}(G_1, \ldots, G_t)$, respectively.
\end{Definition}

Here are natural examples of upper and lower triangular normal subgroup tableaux. 
\begin{Definition}\label{def:coo}
Let $G$ be a subdirect product of $D$.
For $i, j \in \{1, \ldots, t\}$  let 
\[
    G_{ij} = \begin{cases} 
    (G\cap (G_1\times\cdots\times G_j))\pi_i = (G \cap \ker(\pi_{j+1, \ldots, t})) \pi_i \le G_i 
    & \mbox{if $i \le j$,}\\
     G \pi_{j, j+1, \ldots, t} \cap G_i \le G_i 
     & \mbox{if $i \ge j$,}
\end{cases}\]
noting that the two definitions of $G_{ii}$ agree. The \emph{upper coordinate tableau} of $G$ is 
 the $t \times t$ array $\mathbb{U}(G) = (U_{ij})$ with $U_{ij} = G_{ij}$ for $i \le j$, and $U_{ij} = -$ for $i > j$. 
The \emph{lower coordinate tableau} of $G$ is   the $t \times t$ array $\mathbb{L}(G) = (L_{ij})$ with $L_{ij} = G_{ij}$  for $i \ge j$,  and   $L_{ij} = -$  for $i< j$. 
For each upper or lower triangular normal subgroup tableau $T$ for $(G_1, \ldots, G_t)$, we shall write $\Subdir(T)$ for the set of subdirect products of $D$ with coordinate tableau $T$.
\end{Definition}

Thus for example if $t = 3$ then 
\[
\mathbb{U}(G) = \left(
\begin{array}{ccc}
G \cap G_1 & (G \cap \ker(\pi_3))\pi_1 & G_1\\
- & (G \cap \ker(\pi_3))\pi_2  &  G_2\\
- & - & G_3
\end{array}\right), \ 
\mathbb{L}(G) = \left(
\begin{array}{ccc}
G \cap G_1 & - & - \\
G \cap G_2 & G\pi_{2, 3} \cap G_2 & - \\
G \cap G_3  & G\pi_{2, 3} \cap G_3 & G_3 \end{array}
\right). \]

The following is clear. 
\begin{Lemma}\label{lem:tableaux}
Let $G$ be a subdirect product of $D$. Then $\mathbb{U}(G) \in \mathcal{U}(G_1, \ldots, G_t)$ and $G_{it} =  G_i$. 
Furthermore, $\mathbb{L}(G) \in \mathcal{L}(G_1, \ldots, G_t)$, 
the group $N:= G_{11}\times \cdots \times G_{t1}$ is normal in $G$, and $G/N$ is a subdirect product of $G_1/G_{11} \times \cdots\times  G_t/G_{t1}$. 
\end{Lemma}

Our next lemma shows that in order to prove Theorem~\ref{thm:InitialSubdirect}, it will suffice to bound the subdirect products of $D$ with a given lower coordinate tableau.

\begin{Lemma}\label{lem:HowManyTabs}
Let $k$ be an upper bound on the length of a proper normal series in the groups $G_i$, 
let $\lambda$ be the maximum of $|\mathcal{N}(\ell,G_i)|$ over all $\ell$ and $i$,
and let $\mathcal{T}$ be  
$\mathcal{U}(G_1, \ldots, G_t)$ or $\mathcal{L}(G_1, \ldots, G_t)$. 
Then $|\mathcal{T}|\le \lambda^tt!^{k+1}$. In particular, for all $C > 1$ there exists a constant $\cNine = \cNine(C)$ such that
if each $G_i$ is a quotient of a permutation group of degree less than $C$ then $|\mathcal{T}| \le 2^{\cNine t \log t}$. 
\end{Lemma}

\begin{proof} Let $\mathcal{S}: N_1\le\cdots\le N_{m}=G_i$ be a (not necessarily proper) normal series of $G_i$. Suppose that $\ell$ distinct groups occur in the series $\mathcal{S}$,  let $N_{e_1} < N_{e_2} < \cdots < N_{e_{\ell}} \in \mathcal{N}(\ell, G_i)$ be a corresponding proper series and let $m_j$ be the number of times that  $N_{e_j}$ occurs in $\mathcal{S}$. Notice that $\ell\le k$. Write $\mathrm{Part}_{\ell}(m)$ for the set of ordered partitions of the integer $m$ into $\ell$ parts, each of size at least $1$.
Since $m_1+\cdots+m_{\ell}=m$, 
the series $\mathcal{S}$ 
naturally corresponds to 
an element of  $\mathcal{N}(\ell,G_i)\times\mathrm{Part}_{\ell}(m)$. Now $|\mathrm{Part}_{\ell}(m)| \le m^{\ell}$, so  $G_i$ has $\sum_{\ell=1}^m|\mathcal{N}(\ell,G_i)||\mathrm{Part}_{\ell}(m)| \leq \sum_{\ell = 1}^m |\mathcal{N}(\ell, G_i)| m^{\ell}\le \lambda m^{\ell+1}\le \lambda m^{k+1}$ not necessarily proper 
normal series of length $m$. 

For $T \in\mathcal{U}(G_1, \ldots, G_t)$, the subgroups in the $i$th row of $T$ form  a normal series 
of $G_i$ of length $t - i+1$, so there are at most $\lambda(t-i+1)^{k+1}$ choices for this row. Thus \[|\mathcal{U}(G_1, \ldots, G_t)| \leq \lambda^tt!^{k+1} \le  2^{(\log \lambda + k+1)t\log t}.\] The argument for $|\mathcal{L}(G_1, \ldots, G_t)|$ is identical.

The final claim follows by setting $\cNine = \log \lambda + k+1$, where $\lambda$ and $k$ are set to their maximum values over all quotients $G_i$ of permutation groups of degree less than $C$. 
\end{proof}

Next, we introduce some highly structured generating sets for subdirect products. We continue with the definition of $G_{ii}$ from Definition~\ref{def:coo}.

\begin{Definition}\label{defn:efficient} 
Let $G$ be a subdirect product of $D$, 
and let $X_i$ be a $G_i$-generating set for $G_{ii}$. An \emph{efficient generating set} for $G$ with respect to $(X_1,\hdots,X_t)$ is subset $Y= Y_1\cup\cdots\cup Y_t$ of $G$ such that for all $i \in \{1, \ldots, t\}$
\begin{enumerate}[\upshape(i)]
    \item $Y_i\subseteq  \ker(\pi_{i+1, \ldots, t})$,
    \item $Y_i\pi_i=X_i$, and
    \item $\langle Y \rangle = G$.
\end{enumerate}    
\end{Definition}

We assume from now on that $D$ is a $p$-group, and show that efficient generating sets exist: later we will  bound the number of subdirect products of $D$ by carefully counting such sets. 

\begin{Lemma}\label{lem:GenSetDef}
Suppose that $D$ is a $p$-group,  let $G$ be a subdirect product of  $D$, 
and let $X_i$ be a $G_i$-generating set for $G_{ii}$. Then $G$ has at least 
\[\prod_{i = 2}^t|G_{11} \times  \cdots \times G_{(i-1)(i-1)}|^{|X_i|}\]
efficient generating sets with respect to $(X_1, \ldots, X_t)$.
\end{Lemma}
\begin{proof}
We induct on $t$, and prove that there are at least $|G_{11} \times \cdots \times G_{(i-1)(i-1)}|^{|X_i|}$ choices for the set $Y_i$, from which the result follows.

The case $t=1$ is trivial, since then $G_{11}=G_1=G$ (and the product is vacuous). 
So assume that $t>1$ and the result holds for $t-1$. 
Identify $G_{11}$ with $\ker(\pi_{2, \ldots, t})$, so $G/G_{11}\cong G\pi_{2, \ldots, t}$ is subdirect in
$G_2\times\cdots\times G_{t}$. For $i \in \{2, \ldots, t\}$, let $\hat{\pi_i} = \pi_i|_{G_2 \times \cdots \times G_t}$, and notice that
 $G_{ii}= (G\pi_{2, \ldots, t})_{ii}$. 
By induction, for each $i \in \{2, \ldots, t\}$ there are at least  $|G_{22} \times \cdots \times G_{(i-1)(i-1)}|^{|X_i|}$ choices in $G\pi_{2, \ldots, t}$  for a set $\hat{Y_i}$ such that $\hat{Y}_2\cup\cdots\cup\hat{Y}_t$ is an  efficient generating set with respect to $(X_2, \ldots, X_t)$. 

Next, let $Y_1 = X_1 \subset G_{11}$,  
and for $i \in \{2, \ldots, t\}$ let $Y_i$ be an arbitrary preimage of $\hat{Y}_i$ in $G$.
Then by construction 
    $Y_i\subseteq \ker(\pi_{i+1, \ldots, t})$ and 
    $Y_i\pi_i=X_i$ for all $i$. Furthermore,  if $y  = (y_1, \ldots, y_t)$ is one such choice then for all $g \in G_{11}$ the element $(gy_1, \ldots, y_t)$ also lies in $G$ and for $i \ge 2$ has the same image as $y$ under $\pi_{i+1, \ldots, t}$ and $\pi_i$, so there are at least $|G_{11} \times \cdots \times G_{(i-1)(i-1)}|^{|X_i|}$ choices for  $Y_i$ for $i \ge 2$, and a unique choice for $Y_1$.

It remains only to check that the $Y_i$ generate $G$, so let $B=\langle Y_2\cup\cdots\cup Y_t\rangle$. Then by induction $B\pi_{2, \ldots, t}=G\pi_{2, \ldots, t}$,
so $G = \langle G_{11}, B \rangle$.  Applying $\pi_1$ shows that $G_1 = \langle G_{11}, B\pi_1 \rangle = G_{11}B\pi_1$, since $G_{11} \unlhd G_1$. 
Since $G_{11}$ is a $p$-group with $G_1$-normal generating set $Y_1$ we therefore deduce that $G_{11} = \langle Y_1 \rangle [G_{11},G_1] = \langle Y_1 \rangle [G_{11},G_{11}B\pi_1] = \langle Y_1 \rangle[G_{11},B\pi_1][G_{11},G_{11}]$. Thus $\langle Y_1 \rangle [G_{11}, B\pi_1] = G_{11}$, so that $Y_1$ is also a $B{\pi_1}$-normal generating set for $G_{11}$, and hence $G_{11} = \langle Y_1 \rangle^{B}$. 
Finally, from  
$G=\langle G_{11},B\rangle$ we deduce that $\langle Y_1\cup B\rangle=G$, as required.
\end{proof}

The following can easily be proved directly, but is also an immediate corollary of Lemma~\ref{lem:GenSetDef}.

 \begin{Lemma}\label{lem:generate_subdirect}
     Let $G$ be a subdirect product of  the $p$-group $D$, and let $d_i$ be an  upper bound for $d_{G_i}(N_i)$, over all normal subgroups $N_i$ of $G_i$. Then $d(G) \leq d_1 + \cdots + d_t$. 
 \end{Lemma}

Next, we use our  tableaux to prove a result from which Theorem~\ref{thm:InitialSubdirect} will follow easily.

\begin{Theorem}\label{thm:SubdirectCount}
Suppose that $D$ is a $p$-group. Let $n_1,\hdots,n_t$ be such that $|G_i| \le p^{n_i}$ for all $i$, 
and let $n:=\sum_{i=1}^tn_i$. Let $U = (U_{ij})\in \mathcal{U}(G_1, \ldots, G_t)$, 
and let $d:=\max\{d_{G_i}(U_{ii})/n_i\text{ : }1\le i\le t\}$.
If 
the groups $U_{ij}$ appearing in $U$ satisfy 
\begin{align}\label{lab:cs}
 c_i := \dfrac{\log_p{|U_{ii}|}}{n_{i}} \ge \dfrac{\log_p{|U_{ji}|}}{n_j}  
\end{align} 
then $|\Subdir(U)|\le p^{c_t d n^2/4}.$
\end{Theorem}

\begin{proof}
Firstly, for each $i$ let $d_i=d_{G_i}(U_{ii})$, and fix a set $X_i$
of  $d_i$ normal generators for $U_{ii}$. 
By Lemma~\ref{lem:GenSetDef}, each group $G \in \Subdir(U)$  has at least
\begin{align}\label{lab:EGS4}
   \prod_{i = 2}^t (|U_{11} \times \cdots \times U_{(i-1)(i-1)}|)^{|X_i|} =    p^{\sum_{i=2}^t(c_1n_1+c_2n_2+\cdots+c_{i-1}n_{i-1})d_i} 
\end{align}
efficient generating sets with respect to $(X_1,\hdots,X_t)$.

Let $Y = Y_1 \cup \cdots \cup Y_t$ be an efficient generating set with respect to $(X_1, \ldots, X_t)$ for some $G \in \Subdir(U)$. Then  $Y_i$ consists of $d_i$ elements of $\ker(\pi_{(i+1) \ldots, t})$,  whilst by definition $y \pi_i \in X_i$ and $y\pi_j \in  G_{ji} = U_{ji}$ for all $y \in Y_i$ and $j < i$. Hence there are at most 
$(|U_{1i} \times U_{2i} \times \cdots \times U_{(i-1)i}|)^{d_i}$ choices for $Y_i$.   
By \eqref{lab:cs} if  $j \le i$ then  $|U_{ji}|\le p^{c_in_j}$. 
 Hence the number of  such sets $Y$ is at most
\[ \prod_{i = 1}^t (|U_{1i} \times U_{2i} \times \cdots \times U_{(i-1)i}|)^{d_i}
\le p^{\sum_{i = 2}^tc_id_i(n_1 + \cdots + n_{i-1})}. 
\]
Combining this with \eqref{lab:EGS4} yields
\begin{align*}\label{lab:logS}
 \log_p{|\Subdir(U)|}\le F & :=\sum_{i=2}^t\left(c_id_i(n_1 + \cdots + n_{i-1})-(c_1n_1+c_2n_2+\cdots +c_{i-1}n_{i-1})d_i \right)\\
 & = \sum_{i=1}^tc_i\left(d_i(n_1+\cdots+n_{i-1})-\sum_{j>i}n_id_j\right).
\end{align*}
The $i$-th row of $U$ contains subgroups of $G_i$, so 
$c_i = (\log_p{|U_{ii}|})/n_i \le (\log_p{|U_{it}|})/{n_i}\le c_t$.
Let $I$ be the subset of $\{1,\hdots,t\}$ for which the coefficient of $c_i$ in the above expression is positive, and let $F_I$ be the result of 
replacing each $c_i$, for $i\in I$, by $c_t$,  and replacing all remaining $c_i$ by $0$. Then
\[ \log_p{|\Subdir(U)|} \le F_I =c_t \sum_{i\in I}\left(d_i(n_1+\cdots+n_{i-1})-\sum_{j>i}n_id_j \right).\]
If $k \not\in I$ then the coefficient of $d_k$ in $F_I$ is non-positive, whilst if $k \in I$ then the coefficient of $d_k$ is $n_1+\cdots+n_{k-1}-\sum_{i\in I,k>i}n_i=\sum_{i\not\in I, i<k}n_i \leq \sum_{i \not\in I}n_i$. It follows that
\begin{align*}
\log_p{|\Subdir(U)|} \le F_I 
\le c_t\left(\sum_{k\in I}d_k\right)\left(\sum_{i\not\in I}n_i\right) \le c_td \left(\sum_{k\in I}n_k\right)\left(\sum_{i\not\in I}n_i\right),
\end{align*}
since $d_k\le d n_k$ for all $k$.
The expression $(\sum_{i\in I}n_i)(\sum_{i\not\in I}n_i)$ is maximised at $(\sum_{i\in I}n_i)=(\sum_{i\not\in I}n_i)=(\sum_{i = 1}^t n_i)/2=n/2$, so 
$\log_p{|\Subdir(U)|}\le c_t d n^2/4$, as needed.
\end{proof}

\begin{proof}[Proof of Theorem~\ref{thm:InitialSubdirect}]
Consider the set of all upper normal subgroup tableaux of all re-indexings of the $G_i$, and let $\mathcal{U}^\ast(G_1, \ldots, G_t)$ be the subset of all tableaux that satisfy \eqref{lab:cs},  so that 
$|\mathcal{U}^\ast(G_1, \ldots, G_t)| \le  t!|\mathcal{U}(G_1, \ldots, G_t)|$. Each subdirect product $G$ of $D$ has
at least one indexing of the $G_i$ such that $\mathcal{U}(G) \in \mathcal{U}^\ast(G_1, \ldots, G_t)$.
We therefore deduce from Lemma~\ref{lem:HowManyTabs} that 
   \[ |\Subdir(D)|  \leq \sum_{U \in \mathcal{U}^\ast(G_1, \ldots, G_t)}|\Subdir(U)|
     \leq t! \lambda^t t!^{k+1} \max\{|\Subdir(U)| \text{ : } U \in \mathcal{U}^\ast(G_1, \ldots, G_t)\}.\]
For any $U \in \mathcal{U}^\ast(G_1, \ldots, G_t)$, our assumption that $|G_i| \le p^{cn_i}$ 
shows that $c \ge c_t$, where $c_t$ is as in Theorem~\ref{thm:SubdirectCount}. Therefore  
\[ |\Subdir(D)|  \le \lambda^t t!^{k+2}   p^{c_td n^2/4} \le \lambda^t t!^{k+2}  p^{cdn^2/4},\] as required.  

The final claim follows by setting $c = 1/(p-1)$ and $\cSeven = \cNine + 1$, where $\cNine=\cNine(C)$ is the constant from Lemma~\ref{lem:HowManyTabs}.  
\end{proof}

\section{Counting homomorphisms between finite $p$-groups}\label{sec:CountingHoms}

In this section we bound the number of epimorphisms from one finite $p$-group to another. Our main result is Theorem~\ref{thm:HomTheorem}.

\begin{Definition}\label{def:standardised}
Let $R$ and $P$ be finite groups, and let $\pi: R \times P \rightarrow P$ be the natural projection. 
Let $G$ be a subgroup of  $R\times P$ 
such that $G\pi= P$, and suppose that $P = P_1 \times \cdots \times P_m$.
Then a generating set $Z$ for $G$ is \emph{standardised with respect to $P_1, \ldots, P_m$} if exactly $d(P_i)$ elements of $Z$ project nontrivially to $P_i$ for each $i$, and no element projects non-trivially to more than one $P_i$. 
\end{Definition}

\begin{Lemma}\label{lem:tProjLemma}
Let $R$ and $P$ be finite $p$-groups, let $G$ be a subgroup of $R \times P$, let $\pi:G\rightarrow P$ be the natural projection, and suppose $P$ has a direct product decomposition $P = P_1 \times \cdots \times P_m$.
If $G\pi=P$ 
then $G$ has a  standardised generating set $Z$  with respect to $P_1, \ldots, P_m$, of size $d(G)$. 

Let $e$ be an upper bound on $d(P_i)$ over all $i \in \{1, \ldots, m\}$, and  
let $S$ be a subgroup of $\langle V \rangle$ for some subset $V$ of $Z$.  Then all but at most 
$e|V|$ elements $z$ of $Z$ satisfy $[z,S^G]\le \ker(\pi)$.
\end{Lemma}
\begin{proof}
It is clear that $G$ has a standardised generating set $Z$ with respect to $P_1, \ldots, P_m$, and since the removal of any element with non-trivial projection to $P_i$ for some $i$ means that the resulting set no longer generates $G$,  by Burnside's Basis Theorem we can assume that $|Z| = d(G)$. 

For each $i$, write $\pi_i:G \rightarrow P_i$ for the natural coordinate projection, and for a subset $U$ of $G$, let the \emph{$P$-support} of $U$ be 
$$\supp_P(U):=\{i\in\{1,\hdots,m\}\text{ : }g\pi_i\neq 1\text{ for some }g\in U\}.$$
Notice that $\supp_P(S) \subseteq  \supp_P(V)$.
The definition of $Z$ bounds
$|\supp_P(\{z\})|\le 1$ for all $z\in Z$, and so $|\supp_P(S)|\le |V|$.  Since $P$-support is invariant under conjugation, $|\supp_P(S^G)|\le |V|$.

By assumption, for each $i \in \{1, \ldots, m\}$ there
are precisely $d(P_i)$ elements $z$ of $Z$ such that $z\pi_i \neq 1$, so 
there are at most $e|\supp_P(S^G)| \le  e|V|$ elements $z$ of $Z$ such that $\supp_P({z}) \subseteq \supp_P(S)$. Elements $g, h \in G$ with disjoint support satisfy $[g, h]\pi = 1$, so  only these at most $e|V|$ elements $z$ of $Z$ satisfy $[z, S^G]\pi \neq 1$. 
\end{proof}

Recall that  $\mathrm{cl}(G)$ denotes the nilpotency class of $G$.
We now remind the reader of some elementary results about commutators. Commutators will be left-normed, so that  $[g_1,\hdots,g_t]=[\hdots[[[g_1,g_2],g_3],\hdots],g_t]$, and 
$\gamma_i(G)$ will denote the  $i$th term of the lower central series of $G$, so that $\gamma_1(G) = G$. We will use the standard identities $[ab, c] = [a, c]^b[b, c]$ and $[a,bc]=[a,c][a,b]^c$ frequently.

\begin{Lemma}\label{lem:commutators}
Let $G$ be arbitrary, and let $H = \langle X \rangle$ be a subgroup of $G$  of nilpotency class $m$. Let $y \in \gamma_i(H)$, let $g \in C_G(\gamma_{i+1}(H))$, and let $h \in H$ be arbitrary.
\begin{enumerate}[\upshape(i)]
    \item Then $[y^h, g] = [y, g]^{[y, h]}$, $[y, hg] = [y, g][y, h]$, and $[y, g]^{-1} = [y, g^{-1}]$. 
\item For all $i$, the subgroup $\gamma_i(H) = \langle [x_1, \ldots, x_i] \text{ : } x_j \in X\rangle^H$. 
\end{enumerate}
\end{Lemma}

\begin{proof}
(i). Here  $[y^h, g]  = [y[y, h], g] = [y, g]^{[y, h]}[[y, h], g] = [y, g]^{[y, h]}$ since $g$ centralises $[y, h]$,  and 
$[y, hg] = [y, g][y, h]^g  = [y, g][y, h]$ for the same reason. The final claim follows similarly.

\smallskip

\noindent (ii). We induct on $i$, and the case $i = 1$ is trivial, so assume that $i\geq 2$ and  let $V_i = \langle [x_1,\hdots,x_i] \text{ : } x_j\in X\rangle$. 
By Lemma~\ref{lem:normal_gen_pgps}(i), the claim  is equivalent to proving that $\gamma_i(H) = V_i\gamma_{i+1}(H)$. A short calculation using induction and the standard commutator identities shows that $\gamma_i(H)$ is generated by $\{[[x_1,\hdots,x_{i-1}]^h,x_i] \text{ : } x_j\in X, h\in H\}$. Since $H$ is finite, the standard commutator identities imply that $[a^{-1},b]$ is a product of conjugates of $[a,b]$ by powers of $a$, for all $a, b \in H$. The identity $[a^b,c]=[a^{-1},c]^{-a^b}[a,b,c]$ therefore implies
that $[[x_1,\hdots,x_{i-1}]^h,x_i]\equiv [x_1,\hdots,x_i]^k$ modulo $\gamma_{i+1}(H)$, for some $k\in\mathbb{Z}$. The result follows.
\end{proof}

For a group $G$, a subgroup $S$ of $G$, and an integer $i\geq 2$, we write $\gamma_i(S,G)$ for the subgroup $[\gamma_{i-1}(S),G]$, so that for example, $\gamma_i(G,G)=\gamma_i(G)$. 
We now prove a rather technical lemma.

\begin{Lemma}\label{lem:tech}
Let $G = \langle Z \rangle$ be a finite $p$-group, 
and let $S$ and $K$ be subgroups of $G$, with $K$ normal.  If there exists a subset $Z_1$ of $Z$ such that $S \le \langle Z_1 \rangle$ and 
  $[z, S] \le K$ for all $z \in Z \setminus Z_1$
then $G$ has a generating set 
$Z_1 \cup X \cup Y$ 
of size $|Z|$ such that 
$|X| \le \cl(G)|[S,G]\cap K|d(S)^{\cl(G)}$ and $Y \subseteq C_G(S)$. 
\end{Lemma}

\begin{proof}
Let $m = \cl(G)$ and $c_j = |(\gamma_{m+1-j}(S, G) \cap K)\setminus \{1\}|$ for $j \in \{0, \ldots, m-1\}$.  We shall prove by induction the stronger result that if $[z, S] \le K$ for all $z \in Z \setminus Z_1$ then for all $i \in \{0, \ldots, m -1\}$ the group $G$ has a generating set $Z_1 \cup X \cup Y$ of size $d:= |Z|$ such that
\begin{equation*}%
|X| \le \sum_{j = 0}^i c_j \  \text{ and } \  Y \subseteq  C_G(\gamma_{m-i}(S)).
d(S)^{m-j}
\end{equation*}
The result will then follow by setting $i = m-1$, since $\gamma_k(S, G) \leq [S, G]$ for all $k \ge 2$.

If $i = 0$, and in particular if $G$ is abelian,  then we may let $X = \emptyset$ and 
$Y = Z \setminus Z_1$.
Assume inductively that $m > i >0$, and there exist sets $X^\prime$ and $Y^\prime$ such that $Z_1 \cup  X^{\prime} 
\cup Y^{\prime}$ is a generating set for $G$ of size $d$, with $|X^{\prime}| \le \sum_{j = 0}^{i-1} 
c_jd(S)^{m-j}$ and $Y^{\prime} \subseteq C_G(\gamma_{m+1-i}(S))$.

Let $W$ 
be a generating set for $S$ of size $d(S)$. Let $\Lambda = \{\lambda_1, \ldots, \lambda_{|\Lambda|}\}$ 
be the set of all $(m-i)$-fold commutators $[w_{k_1}, \ldots, w_{k_{m-i}}]$ of 
elements of $W$, so that $|\Lambda| \leq d(S)^{m-i}$ 
and $[\lambda, g] \in \gamma_{m+1-i}(S, G)$ for all  $\lambda \in \Lambda$ and $g \in G$. Index the elements of $\Lambda$ arbitrarily as $\lambda_1, \ldots, \lambda_{|\Lambda|}$, and notice that
 if $g \in G$ satisfies $[S, g] \le K$ then $[\lambda_k, g] \in \gamma_{m+1-i}(S, G) \cap K$ for all $k \in \{1, \ldots, |\Lambda|\}$. 

We will now show by induction on $t$ that for each $t \in \{0, \ldots, |\Lambda|\}$ there exists a generating set $Z_1 \cup X_t \cup Y_t$ for $G$ of size $d$ such that $Y_t$ centralises $\langle \gamma_{m+1-i}(S), \lambda_1, \ldots, \lambda_t\rangle$ and 
$|X_t|\le |X^{\prime}| + c_it$.  
Setting $X_0 = X^{\prime}$ and $Y_0 = Y^{\prime}$ gives the result for $t = 0$, so assume that $t \ge 1$ and that $X_{t-1}$ and $Y_{t-1}$ have been successfully defined. 

Initially, we let $X_t = X_{t-1}$,  $Y_t = \{y \in Y_{t-1} \text{ : } [\lambda_t, y] = 1\}$ and $U_t = Y_{t-1} \setminus Y_t$. 
Suppose that there exist distinct $u$ and $v$ in $U_t$ such that $[\lambda_t, u] = [\lambda_t, v]$. Then $\lambda_t \in \gamma_{m-i}(S)$ and $u, v \in Y_{t-1} \subseteq C_G(\gamma_{m+1-i}(S))$, so $[\lambda_t, uv^{-1}] = 1$ by Lemma~\ref{lem:commutators}(i). By assumption, $uv^{-1}$ also centralises $\lambda_1, \ldots, \lambda_{t-1}$, so we may remove $v$ from $U_t$ and put $vu^{-1}$ into $Y_t$ whilst maintaining the fact that $Z_1 \cup X_t \cup Y_t \cup U_t$ is a generating set for $G$ of size $d$ such that $Y_t$ centralises $\langle \gamma_{m+1-i}(S), \lambda_1, \ldots, \lambda_t\rangle$. 
Once the values of $[\lambda_t, u]$ are distinct for all $u \in U_t$, the set $U_t$ has size at most $c_i$, so 
we set $X_t = X_t \cup U_t$.
This completes this induction on $t$.

Now, setting $X =X_{|\Lambda|}$ and $Y=Y_{|\Lambda|}$ 
yields a generating set of size $d$, with $|X|$ suitably bounded, so it remains only show that $Y$ centralises $\gamma_{m-i}(S)$. By Lemma~\ref{lem:commutators}(i), for all $s\in S$, all $y\in Y \subseteq   C_G(\gamma_{m+1-i}(S))$, and all $\lambda \in \Lambda \subseteq \gamma_{m-i}(S)$, the commutator $[\lambda^s,y]= [\lambda, y]^{[\lambda, s]}$. Since $Y$ centralises  $\Lambda$, it follows that $Y$ centralises all $S$-conjugates of $\Lambda$. Now Lemma~\ref{lem:commutators}(ii) shows that  $\gamma_{m-i}(S)$ is generated by the set of $S$-conjugates of $\Lambda$, so $Y$ centralises $\gamma_{m-i}(S)$, as required.  
This completes the induction on $i$, and hence the proof.
\end{proof}

We now prove an easier technical result. Recall that if $K$ is a $p$-group, by $K^p$ we mean the subgroup generated by all $p$th powers of elements of $K$. For a subgroup $L$ of a group $G$ we shall write $\NormSub(G, L)$ for the set of $G$-normal subgroups of $L$.

\begin{Lemma}\label{lem:tech2}
Let $G$ be a finite $p$-group,  let $L$ be a normal subgroup of $G$, and let $d$ be an upper bound on 
$d_G(H)$ over all $H \in \NormSub(G, L)$. Then for all $k \ge 1$ there are at most $k^d$ subgroups $H \in \NormSub(G, L)$ with $|L:H|\le k$.
\end{Lemma}
\begin{proof}
We will induct on $k = p^a$. The case $k = p^0$ is trivial, so
assume  $k = p^a$ with $a \ge 1$. 

 For each $H \in \NormSub(G, L)$ with $|L:H| = p^a$, since $L/H \unlhd G/H$ the group $L/H$ intersects $Z(G/H)$ non-trivially, so there exists at least one $J \in \NormSub(G, L)$ with $|L:J| = p^{a-1}$.
Write $\mathrm{Max}(J)$ for the set of maximal subgroups of $J$. Then \[|\{H \in \NormSub(G, L) \text{ : } |L:H| = p^a\}| \leq \sum_{J \in \NormSub(G, L), |L:J| = p^{a-1}}|\mathrm{Max}(J)|.\]  
For all $H \in \mathrm{Max}(J)$ the group $G$ acts trivially on $J/H$, and $J^p\le H$, so  $[J,G]J^p \leq H$. It follows that $|\mathrm{Max}(J)|$ is precisely the number of co-dimension $1$ subspaces of the trivial $G$-module $J/([J,G]J^p)$, namely $(p^{\dim{(J/([J,G]J^p))}}-1)/(p-1) \le p^{\dim{(J/([J,G]J^p))}}-1$. By Lemma~\ref{lem:normal_gen_pgps}(i),  $\dim{(J/([J,G]J^p))}=d_G(J )\le d$, and so $|\mathrm{Max}(J)| \leq p^d-1$. 
Therefore by induction there are at most $p^{(a-1)d} (p^d-1) + p^{(a-1)d} =  p^{ad} = k^d$
subgroups $H \in \NormSub(G, L)$ with $|L:H| \leq k$. 
\end{proof}

The following  is the main result of this section, and can in fact be strengthened slightly: if $F$ is contained in a proper verbal subgroup $V(Q)$ of $Q$, then it suffices to let $d_N$ bound  the normal generator number of subgroups in $\NormSub(G, [V(G), G] \cap \ker(\pi))$, but we omit the details. 

\begin{Theorem}\label{thm:HomTheorem}
Let $G$ and $Q$ be finite $p$-groups, and let $F$ be a normal subgroup of $Q$.
Write $G$ as a subgroup of $R \times P$ for some $R$ and $P$ such that $G \pi = P$, where $\pi: R\times P \rightarrow P$ is the natural projection. Suppose that $P$ factorises as a direct product of $e$-generated subgroups, and let $d_N$ be an upper bound on $d_G(H)$ over all $H \in \NormSub(G, G' \cap \ker(\pi))$. 
Define
\[ \begin{array}{rl}
 h(Q,F,G) & =|\Aut(Q)||Q|^{d(Q)(1 + e) + \mathrm{cl}(G)|[F, Q]|d(F)^{\mathrm{cl}(G)}}. \end{array}\] 
Then
$$|\Epi(G,Q)|\le h(Q,F,G)d(G)^{d(Q)}|C_Q(F)|^{d(G)}|[F, Q]|^{d_N}.
$$
In particular, if $Q$ and each of the given direct factors of $P$ have order at most $p^k$ for some $k$, and $R$ is a subdirect product of groups of order at most $p^k$, then
$$|\Epi(G,Q)|\le \cSevenA d(G)^{\cSix}|C_Q(F)|^{d(G)}|[F, Q]|^{d_N}.
$$
for some constants $\cSevenA =\cSevenA(p,k)$ and $\cSix = \cSix(p,k)$.
\end{Theorem}

\begin{proof}
First, by Lemma~\ref{lem:tProjLemma} there exists a standardised generating set $Z$ for $G$ of size $d(G)$ with respect to the given factorisation of $P$. If $\Epi(G,Q)= \emptyset$ then the result is trivial, otherwise we may choose a subset $V$ of $Z$ of size $d(Q)$ such that $\Epi(\langle V \rangle, Q) \neq \emptyset$, and we may fix a $\rho \in \Epi(\langle V \rangle, Q)$. Write $\Epi(G, Q, \rho)$ for the set of $\theta \in \Epi(G, Q)$ such that $\theta|_{\langle V \rangle} = \rho$.  
Let $S$ be a $d(F)$-generated subgroup of $\langle V \rangle$ such that $\rho(S) = F \unlhd Q$: notice that every epimorphism $\theta \in \Epi(G, Q, \rho)$  maps $S^G$ to $F$. Finally, let $K = \ker(\pi)$. 

Let $H \in \NormSub(G, [S^G,G]\cap K)$ have index at most $|[F,Q]|$ in $[S^G,G]\cap K$,
and define 
\begin{align*}
\mathcal{N}(V,\rho,H) = \{N \unlhd G \text{ : } &  N = \ker(\theta) \mbox{ for some } \theta \in \Epi(G, Q, \rho) 
 \mbox{ and } N \cap ([S^G,G]\cap K)=H\}.
\end{align*}
We first bound $|\mathcal{N}(V,\rho,H)|$: notice in particular that if it is non-zero, then $H \le \ker(\theta)$ for some $\theta \in \Epi(G, Q, \rho)$.
By Lemma~\ref{lem:tProjLemma}, all but at most $d(Q) e$ elements $z$ of $Z$ satisfy $[z,S^G]\le K$, so we can decompose $Z$ as $V \cup W \cup W_1$ with $|W| \leq d(Q) e$ and $[z,S^G]\le K$ for all $z \in W_1$. 

Writing bars to denote reduction modulo $H$, we
apply Lemma~\ref{lem:tech} to $\ol{G} = \langle \ol{Z} \rangle$, with $\ol{S} \leq \langle \ol{V} \cup \ol{W} \rangle$ (so that $Z_1$ becomes $\ol{V} \cup \ol{W}$) 
to see that  $\ol{G}$ has a generating set $$\ol{T}:=  (\ol{V} \cup \ol{W}) \cup \ol{X} \cup \ol{Y}$$ of size $|\ol{Z}|$ such that  $\ol{Y} \subseteq  C_{\ol{G}}(\ol{S})$ and $|\ol{X}| \le \mathrm{cl}(\ol{G})|[\ol{S}, \ol{G}] \cap \ol{K}|d(\ol{S})^{\cl(\ol{G})}$. Furthermore, $d(\ol{S}) \le d(S) \le d(F)$ and by our assumption on $|H|$, 
$$|[F, Q]] \ge \left|\frac{[S^G,G]\cap K}{H}\right| = \left|\frac{[S^G,G]}{H}\cap \frac{K}{H}\right|= |[\ol{S}^{\ol{G}},\ol{G}]\cap\ol{K}|,$$ so $|\ol{X}|\leq \mathrm{cl}(G)|[F, Q]|d(F)^{\mathrm{cl}(G)}$.

Let $\theta \in \Epi(G, Q,\rho)$ satisfy  $\theta(H) = 1$. Then $\theta$ is completely determined by the reduced map $\ol{\theta}:\ol{G}\rightarrow Q$, which in turn is determined by 
the image of each $t \in \ol{T}$. Since $\ol{\theta}|_{\langle\ol{V}\rangle} = \ol{\rho}$, the map 
 $\ol{\theta}$ is determined by $\ol{\theta}|_{\ol{W} \cup \ol{X} \cup \ol{Y}}$. Clearly, there are at most $|Q|^{|\ol{W}| + |\ol{X}|} \le  |Q|^{d(Q) e + |\ol{X}|}$ possibilities for the restriction of $\ol{\theta}$ to $\ol{W} \cup \ol{X}$. Also, since each such $\ol{\theta}$ satisfies $\ol{\theta}(\ol{S})=F$, each element of $\ol{Y}$ is mapped to $C_Q(F)$. 
 Each $N \in \mathcal{N}(V, \rho, H)$ is the kernel of at least one such epimorphism $\theta$, so 
\begin{align*}
  |\mathcal{N}(V,\rho,H)|& 
  \le |Q|^{d(Q) e+|\ol{X}|}|C_Q(F)|^{|\ol{Y}|} 
  \le |Q|^{d(Q) e+\cl(G)|[F, Q]|d(F)^{\mathrm{cl}(G)}}|C_Q(F)|^{d(G)}\\ & \le h(Q, F, G)/(|\Aut(Q)||Q|^{d(Q)} ) \cdot  |C_Q(F)|^{d(G)}.
 \end{align*}

Next we consider the $M \in \NormSub(G, [S^G, G] \cap K)$ which be expressed as $[S^G, G] \cap K \cap \ker(\theta)$, for some $\theta \in \Epi(G, Q, \rho)$. Writing $N = \ker(\theta)$, 
from $\rho(S) = F$ we calculate
\[
\left| [F, Q]\right| =  \left| \left[\frac{S^GN}{N},\frac{G}{N}\right] \right| = \left| \frac{[S^G,G]N}{N} \right| \ge  \left| \frac{([S^G,G]\cap K)N}{N} \right| = \left| \frac{[S^G,G]\cap K}{M} \right|,\]
so $M$ has index at most $|[F, Q]|$ in $[S^G,G]\cap K$, and in particular $|\mathcal{N}(V, \rho, M)|$ is bounded as above.  Furthermore, since $[S^G, G] \cap K \leq G' \cap K$ 
we can bound $d_G(M) \le d_N$. 
Now Lemma~\ref{lem:tech2}, with $[S^G, G] \cap K$ in place of $L$, shows that there are at most $|[F,Q]|^{d_N}$ 
groups $M$.

Next, the 
number of normal subgroups $N$ that satisfy $N = \ker(\theta)$ for some $\theta \in \Epi(G, Q, \rho)$ 
is bounded by the product of $|[F,Q]|^{d_N}$ and the maximum of $|\mathcal{N}(V,\rho,H)|$ over all choices for 
$H = [S^G, G] \cap K \cap \ker(\theta)\}$.
Also, we can bound $\Epi(\langle V \rangle, Q)$ by $|Q|^{|V|} = |Q|^{d(Q)}$.
Thus, letting $\mathcal{N}(V)$ be the set of normal subgroups $N$ of $G$ such that $
\langle V\rangle N/N = G/N\cong Q$ we deduce that 
\begin{align*}\label{lab:rhoexp}
    |\mathcal{N}(V)| & \le 
    (h(Q, F, G)/|\Aut(Q)|) |C_Q(F)|^{d(G)}|[F,Q]|^{d_N}.
\end{align*}

Finally, Burnside's Basis Theorem shows that every $N$ such that $G/N \cong Q$ lies in $\mathcal{N}(V)$ for at least one $d(Q)$-element subset $V$ of $Z$. 
There are at most $ |Z|^{|V|} =d(G)^{d(Q)}$ such $V$, and for each such $N$ there are at most $|\Aut(Q)|$ epimorphisms $\theta: G \rightarrow Q$ with kernel $N$, so the main result follows.

The final claim is deduced by observing that we have absolutely bounded $d(Q)$, $|\Aut(Q)|$, $e$, and $\cl(G)$.
\end{proof}

\section{Normal subgroups of transitive $p$-groups}\label{sec:GenSec}

In this section we study normal subgroups of transitive $p$-subgroups of $\Sn_n$, and in particular derive bounds on their normal generator number. 

\begin{Definition}\label{def:excessive}
Let $G\le \Sn_n$ be a transitive $2$-group, and let $E\unlhd G$.  If there exists an $i > 0$ such that $d_G(E) = n/4 + i$ then for the largest such $i$ we say that $(G,E)$ has \emph{excess $i$} and is \emph{excessive}. 
Furthermore, if 
 $(G,E)$ has excess $i$ for some $E\unlhd G$, then for the largest such $i$ we say that
 $G$ has \emph{excess $i$} and is \emph{excessive}. 
\end{Definition}

We will start by studying excessive groups. In particular, in Proposition~\ref{prop:Gen2Groups} we shall prove that every excessive group has degree at most $32$, and that no group has excess more than $2$. 
Throughout this section, we will use the \Magma library \cite{CannonHolt32} of transitive groups of degree at most 32. The groups of degree $32$  have order up to $2^{31}$, so 
any statement involving all of their normal subgroups will be proved using a mixture of theoretical and computational techniques, to ensure that our computations run in a reasonable amount of time. For many of our computations it is much more efficient to work with polycyclic presentation of the transitive $2$-group, rather than the permutation group. If the degree $n$ is clear, we shall refer to \texttt{TransitiveGroup(n, i)} in the transitive groups library as the group with \emph{ID $i$}.

We first determine the excessive groups of degree 
$32$.  Recall that Lemma~\ref{lem:normal_gen_pgps}(i) bounds $d_G(E) \le \log |E/[G, E]|$. 

\begin{Lemma}\label{lem:GenCompRevised}
Let $G$ be a transitive $2$-group of degree 32.
\begin{enumerate}[\upshape(i)]
    \item If $(G, G)$ has excess $i$ then $i \le 2$. 
    \item If $(G, E)$ has excess $i$ for some proper normal subgroup $E$ of $G$ 
    then $i = 1$, $|G:E| = 2$,  and $(G,G)$ has excess $2$.
    Conversely, if $E$ is a proper normal subgroup of $G$ such that
    $|E/[G, E]| > 2^8$ then $(G, E)$ has excess $1$. In particular, $G$ is excessive if and only if $(G, G)$ is excessive. 
    \item Suppose that $G$ is excessive, and let $V$ be a submodule of the natural $\mathbb{F}_2[G]$-permutation module. Then $d_G(V) \le 7$, and if $d_G(V) = 7$ then $G$ is \texttt{TransitiveGroup(32, 145717)}. 
\end{enumerate}
\end{Lemma}

\begin{proof}
(i). This is proved computationally: we check using Lemma~\ref{lem:normal_gen_pgps}(i) that each transitive $2$-group $G$ of degree $32$ satisfies  $d(G) \le 10$.

\smallskip
\noindent (ii). By Lemma~\ref{lem:normal_gen_pgps}(i), if $|E/[G,E| \le 2^8$ then $d_G(E) \le 8$, so $(G, E)$ is not excessive. Thus 
%
we must prove that if $E$ is a proper normal subgroup of $G$ with $|E/[G,E]|>2^8$, then $|G:E| = 2$, with $d(G) = 10$ and $d_G(E) = 9$. 
So let $E$ be such a normal subgroup.  We start by computing $|\Gab|$ for each transitive $2$-group $G$ of degree $32$, and observe that in each case 
\begin{equation}\label{eq:gab}
    |\Gab| \le 2^{10}.\end{equation}

Suppose first that $E$ supplements an abelian subgroup $A$ of $G$. Then $[G,E]=[EA,E]=[E,E][A,E]= [E, E][A,E][E,A][A,A] =[EA,EA]=[G,G]$, so $ 2^9 \le |E/[G,E]|=|\Gab|/|G:E| \le 2^{10}/|G:E|$, so that $|G:E| = 2$ and $|\Gab| = 2^{10}$. We check computationally that if $|\Gab| = 2^{10}$ then $(G, G)$ has excess $2$, and $d_G(E) = 9$ for each $E \unlhd G$ with $|G:E| = 2$. 

Thus, from now on we may assume that $E$ supplements no abelian subgroup of $G$. Then $|G:E|\geq 4$, since in particular $E$ supplements no cyclic subgroup of $G$.
Furthermore, let $K$ be the kernel of the action of $G$ on a set $\Sigma$ of minimal blocks for $G$. Then $K$ is abelian so $EK\neq G$, and hence $G^{\Sigma} \neq E^{\Sigma}$. 
We shall show that no such $E$ exists. 

Let $U = E \cap K$, and let $V$ be the natural $\mathbb{F}_2[G^{\Sigma}]$-permutation module.
Then $K$ and $U$ are isomorphic to submodules of $V$, and $G$ acts trivially on $E/[G,E]$ and hence on  the module $(E\cap K)/([G,E]\cap K) = U/([G,E]\cap K)$. 
Let $d = d_{G^{\Sigma}}(U)$. Then
 \begin{equation}\label{eq:quotsize}
 2^9 \le \left|\frac{E}{[G,E]}\right|=\left|\frac{U}{[G,E]\cap K}\right| \left|\frac{E^{\Sigma}}{[G,E]^{\Sigma}}\right| \le 2^{d}\left|\frac{E^{\Sigma}}{[G,E]^{\Sigma}}\right|.
 \end{equation}

We use \Magma to consider each transitive $2$-group $H$ of degree $16$ as a possible $G^{\Sigma}$. If $H$ is not elementary abelian, then 
we compute the submodules of $\BBF_2[H]$ in \Magma and see that  $d\le 4$. If $H$ is elementary abelian (so has ID 3) then similarly we find that $d \le 6$.
Next, we use \Magma to calculate $|L/[H, L]|$ for  each proper normal subgroup $L$ of  $H$. In each case this is at most $2^5$ if $|H:L| =  2$, at most $2^4$ if $|H:L|>2$, and at most $2^4/[H:L]$ if $H$ is elementary abelian. Since $E^{\Sigma}\neq G^{\Sigma}$, it follows from \eqref{eq:quotsize} that $|G^{\Sigma}:E^{\Sigma}|=2$, and either $G^{\Sigma}$ is elementary abelian, so has ID 3 and $d=6$, or $|E^{\Sigma}|/|[G^{\Sigma}, E^{\Sigma}]| = 2^5$ and $d = 4$, which implies that 
$G^{\Sigma}$ has ID 197 or 448. 
In addition, notice that equality holds in \eqref{eq:quotsize}, so $|U/([G,E]\cap K)|=2^d$. 

Next, since $|G:EK|=2$, the subgroup is supplemented by the abelian group $\langle g \rangle$ for any $g \not\in EK$. Thus, reasoning as just after \eqref{eq:gab} shows that $[G, G] = [G, EK]$ and hence $[G, G] = [G,E][G,K]$. It follows that
\begin{equation}\label{eq:multiquot}
|[G,E]| =\left|\frac{[G,G]}{[G,K]}\right|\left|[G,E]\cap [G,K]\right|
\geq \left|\frac{[G,G]}{[G,K]}\right|\left|[E,K]\right| = \frac{|[G,G]|}{|\mathrm{Rad}_{G^{\Sigma}}(K)|}|\mathrm{Rad}_{E^{\Sigma}}(K)|.\end{equation}
 Therefore
 \begin{equation*}
 \left|\frac{E}{[G,E]}\right|\le \left|\frac{\mathrm{Rad}_{G^{\Sigma}}(K)}{\mathrm{Rad}_{E^{\Sigma}}(K)}\right| \frac{|E|}{|[G,G]|}.
 \end{equation*}
Furthermore, from Lemma \ref{lem:normal_gen_pgps}(i), $2^{d}=|U/[G,U]|$. Thus, from $|U/([G,E]\cap K)|=2^{d}$ we deduce that $[G,E]\cap K= [G, U] = \mathrm{Rad}_{G^{\Sigma}}(U)$. 
Since $[G,U]\le [G,E]\cap [G,K]\le [G,E]\cap K$, equality holds throughout, so from the initial equality in \eqref{eq:multiquot} we deduce
\begin{equation*}
    \left|\frac{E}{[G,E]}\right|= \left|\frac{\mathrm{Rad}_{G^{\Sigma}}(K)}{\mathrm{Rad}_{G^{\Sigma}}(U)}\right| \frac{|E|}{|[G,G]|}.
\end{equation*}
Comparing these two expressions for $|E/[G,E]|$ shows that $V$  has submodules $U \le K$, with $d_{G^{\Sigma}(U)}=d$, such that $\left|\mathrm{Rad}_{G^{\Sigma}}(U)\right|\ge \left|\mathrm{Rad}_{E^{\Sigma}}(K)\right|$ for some maximal subgroup $E^{\Sigma}$ of $G^{\Sigma}$.
We use \Magma to compute the maximal subgroups $E^{\Sigma}$ of our threee possible groups $G^{\Sigma}$, and to compute the submodules $K \ge U$ of the corresponding $V$, and find that this implies that
$K=U$. Since $U = E \cap K$ this implies that $K \le E$ so $|G:E| = |G^{\Sigma}:E^{\Sigma}| = 2$, a contradiction.

\smallskip

\noindent (iii). This is proved using  \Magman: for each excessive group $G$, we compute all submodules of the natural $\BBF_2[G]$ permutation module, and find the difference between its dimension and the dimension of its Jacobson radical.   
\end{proof}

Next we classify all normal subgroups requiring many generators. 

\begin{Proposition}\label{prop:Gen2Groups}
Let $p$ be prime, let $G$ be a transitive $p$-group of degree $n>2$, and let $E$ be a (not necessarily proper) normal subgroup of $G$. 
\begin{enumerate}[\upshape(i)]
\item Suppose that $p=2$. If $G$ is excessive then $G$ has  excess $i \in \{1, 2\}$ and $n \le 32$. If $i = 2$ then $(G, E)$ has excess $3 - [G:E]$, and $n \in \{8, 16, 32\}$. If $i = 1$ then the only $E$ such that $(G,E)$ is excessive is $G$ itself. 
\item Suppose that $p$ is odd. If $d_G(E)>n/(2p)$ then either $n=p$ and $E=G\cong C_p$; or $n=9$, $E=G$ and $d(G)=2$; or $n=27$, $E=G$ has ID $712$ in the \Magma database, and $d(G)=5$.
\item Suppose that $p$ is odd. If $d_G(E)> 2n/p^2$, then $n=p$ and $E=G\cong C_p$.
\end{enumerate}
\end{Proposition}
\begin{proof}(i). For $n \in \{4, 8, 16\}$ it is straightforward to check this in \Magman, using Lemma~\ref{lem:normal_gen_pgps}(i).
For $n = 32$ this is immediate from Lemma~\ref{lem:GenCompRevised}, so we may assume that $n \ge 64$: we must prove that $d_G(E)\le n/4$.

We fix some notation. Let $\Sigma$ be a minimal block system for $G$, let $\pi$ be the natural surjection to $S:=G^{\Sigma}$, and let $K = \ker(\pi)$.
 Note that $K$ is isomorphic, as an $S$-group, to a submodule of the natural $\mathbb{F}_2[S]$-permutation module.

Assume first that $n=64$, so that $S$ has degree $32$. 
If $S$ is elementary abelian, then \cite[Theorem 4.13]{Tracey18} shows that $d_S(E\cap K)\le \binom{5}{\lfloor \frac{1}{2}5\rfloor }=10$, so $d_G(E)\le 10+5<16=n/4$ by Lemma~\ref{lem:normal_gen_pgps}(ii). If $S$ is not elementary abelian, then $d_S(E\cap K)\le n/8=8$ by \cite[proof of Lemma 7.1(4)]{GMP}, so if $S$ is not excessive then $d_G(E)\le 8+8=16=n/4$, again by Lemma~\ref{lem:normal_gen_pgps}(ii). Finally, if $S$ is excessive, then  Lemma~\ref{lem:GenCompRevised}(iii) shows that $d_S(E \cap K) \le 6$, unless $S$ has ID 145717 in which case $d_S(E \cap K) \le 7$. Additionally, since the result holds for $n = 32$ we know that
$d_S(E\pi) \le 8 + 2 = 10$, and $d_S(E \pi) \le 9$ if $S$ has ID 145717, so again it follows from Lemma~\ref{lem:normal_gen_pgps}(ii) that $d_G(E) \le 16$. 

Thus we may assume  inductively that $n=2^k \ge 128$, and that no transitive $2$-group of degree less than $n$ and greater than $32$ is excessive. If $S$
is elementary abelian, then $d_{S}(E\pi)=d(E\pi)\le d(S) = k-1$. Further, $d_S(E\cap K)\le \binom{k-1}{\lfloor \frac{1}{2}(k-1)\rfloor }\le b2^{k-1}/\sqrt{k-1}$ by \cite[Lemma~4.1 and Theorem~4.13]{Tracey18}, where $b:=\sqrt{2/\pi}$. It follows from 
Lemma~\ref{lem:normal_gen_pgps}(ii) that $d_G(E)\le b2^{k-1}/\sqrt{k-1}+k-1$, which is less than or equal to $n/4$ since $k\geq 7$.
If instead $S$ is not elementary abelian,  then \cite[proof of Lemma 7.1(4)]{GMP} bounds $d_S(E\cap K)\le \frac{1}{4}\dim{K}=n/8$. Since $n/2>32$, the minimality of $n$ implies that $d_{S}(E\pi)\le n/8$, and so $d_G(E)\le n/4$ by Lemma~\ref{lem:normal_gen_pgps}(ii), as required.

\smallskip
\noindent (ii) and (iii). We prove these simultaneously by induction on $n$ as a power of $p$.  
The case $n = p$, so $G \cong C_p$, is immediate. 
If $n=p^2$, then $d_G(E) \le 2$, so Part (iii) holds, as does Part (ii) unless $n=9$. We check the transitive $3$-groups of degree $9$ and $27$ in \Magma to see that if $n = 9$ then $d_G(E)=2$ if and only if $G$ is not cyclic and $E=G$, and that the claims for $n = 27$ also hold. 

For the inductive step, assume that $n \ge p^3$, with $n>27$ if $p=3$. Then $G$ embeds in $C_p\wr \Sn_{n/p}$, and we let $\pi$  be the natural surjection, with transitive image $S$ and kernel $K$. 
If $n=81$ and $S$ is $\texttt{TransitiveGroup}(27,712)$, we can compute all submodules of the natural $\BBF_3[S]$-permutation module to check that  $d_S(K\cap E)\le 3$. It follows from Lemma~\ref{lem:normal_gen_pgps}(ii) that 
in this case $d_G(E)\le 3+5=8$. In all other cases, we use \cite[Lemma 3.1]{HRD} to bound $d_S(K\cap E)\le n/p^2$. Since $n>p^2$, and $n>27$ if $p=3$, the inductive hypotheses for parts (ii) and (iii) show 
$d_S(E\pi)\le n/(2p^2)$  and $d_S(E\pi)\le 2n/p^3$, respectively. We now apply Lemma~\ref{lem:normal_gen_pgps}(ii) again: since $n/p^2 + n/(2p^2) \le n/(2p)$ and $n/p^2+2n/p^3\le 2n/p^2$, both parts follow.
\end{proof}

\begin{Definition}\label{defn:k-projecting}
    Let $G$ be a subdirect product of $G_1 \times \cdots \times G_t$. We say that $G$ is \emph{$k$-projecting} if there exists a subset $I:= \{i_1, \ldots, i_k\}$  of $\{1, \ldots, t\}$ such that $G\pi_{I} = G_{i_1} \times \cdots \times G_{i_k}$, but no such subset of larger size. 
\end{Definition}

We are now ready to state and prove the two main results of this section. The first is a general upper bound on $d_G(N)$ for arbitrary normal subgroups $N$ of most $p$-groups $G$.

\begin{Theorem}\label{thm:Gen2pGroupsF}
Let $G_1,\hdots,G_t$ be transitive $p$-groups, of degrees $n_1,\hdots,n_t$ respectively, let $n = \sum_{i = 1}^n n_i$ and let $G$ be a subdirect product of $G_1\times\cdots\times G_t \le \Sn_n$. 
Assume that $G$ is $(t-r)$-projecting, and the $G_i$ are indexed so that $G\pi_{r+1, \ldots t} = G_{r+1} \times \cdots \times G_t$. 
Let $N$ be a normal subgroup of $G$, let $d_1 = \sum_{i \le r} n_i$ and $d_2 = \sum_{i \ge r+1} n_i$. 
\begin{enumerate}[\upshape(i)]
    \item If $p$ is odd and $n_i > p$ for all $i$, then $d_G(N) \leq 2n/p^2$. Moreover, if $p=3$ and  for each $i$ either $n_i > 27$ or $n_i = 27$ and $d_{G_i}((N\cap (G_1\times \cdots\times G_i))\pi_{i}) \neq 5$ then $d_G(N) \leq n/6 = n/(2p)$.
    \item Assume $p = 2$. 
    \begin{enumerate}[\upshape(a)] 
        \item $d_G(N)\le \frac{3}{8}d_1+\frac{1}{2}d_2$.
        \item If no $G_i$, for $i \le r$, has excess 2 and degree 8, then $d_G(N)\le \frac{5}{16}d_1+\frac{1}{2}d_2$.
        \item If no $G_i$, for $i \le r$, has excess 2, then $d_G(N)\le \frac{1}{4}d_1+\frac{1}{2}d_2$.
    \end{enumerate}
    \item Assume $p = 2$,  suppose that $N\le \GG\cap\ker(\pi_{r+1, \ldots, t})$, and let $J$ be the subset of $\{1, \ldots, r\}$ such that $G_i$ is non-abelian if and only if $i\in J$. Then $d_G(N)\le \frac{1}{4}(\sum_{j\in J}n_j)\le \frac{1}{4}d_1$.
\end{enumerate}
\end{Theorem}
\begin{proof}
For all $i$, let $N_i=[N\cap (G_1\times \cdots\times G_i)]\pi_{i}$. Then $N_i\unlhd G_i$ and $d_G(N)\le \sum_{i=1}^td_{G_i}(N_i)$. 

\smallskip

\noindent (i). This  is now immediate from Proposition~\ref{prop:Gen2Groups}(ii) and (iii).
\smallskip

\noindent (ii).  For any $i \le r$,  let  $X = (G \cap (G_1 \times \cdots \times G_i))\pi_{i, r+1, \ldots t}$. 
If $X \pi_i = G_i$, then $X \cap G_i = G_i$, and so $G \pi_{i, r+1, \ldots, t} = G_i \times G_{r+1} \times \cdots \times G_t$, contradicting the fact that $G$ is $(t-r)$-projecting. Therefore $N_i<G_i$ for all $i \le r$, so either $n_i \le 2$ and $d_{G_i}(N_i) = 0$, or Proposition~\ref{prop:Gen2Groups}(i) shows that $d_{G_i}(N_i) \le (n_i/4) + 1$. Hence in both cases $d_{G_i}(N_i)\le 3n_i/8$ for all $i \le r$.
Furthermore,  if no such $G_i$ has excess 2, then  $d_{G_i}(N_i)\le n_i/4$ for all $i\le r$; while if no  such $G_i$ has excess $2$ and degree $8$  then  $d_{G_i}(N_i)\le 5n_i/16$ for all  $i\le r$. Part (ii) now follows, since $d_{G_i}(N_i)\le n_i/2$ for $i\ge r+1$.    

    \smallskip
    
\noindent (iii).
Our assumptions imply that $N_i=1$ whenever 
$i\not\in J$. Furthermore, $N_i \le G'_i$ for each $i$, so $|G_i : N_i| > 2$ for all $i \in J$. It then follows from Proposition~\ref{prop:Gen2Groups}(i) that $d_{G_i}(N_i) \le n_i/4$ for all $i \in J$, and Part (iii) now follows.
\end{proof}

To prove our second main theorem we need a technical lemma.

\begin{Lemma}\label{lem:precomp}
Let $G$ be a finite $2$-group such that $\Gab$ is elementary abelian. 
Let $N$ be a normal subgroup of $G$, and let $Q = G/N$. Then
\begin{enumerate}[\upshape(i)]
    \item $d(Q)=\dim(G/(N\GG))$.
    \item Suppose that $N\le \GG$. Let $K$ be an elementary abelian normal subgroup of $G$ such that $d(G)=d_G(K)+d(G/K)$. Let $d$ be  an upper bound on $d_G(U)$, over every $G$-submodule $U$ of $K$. 
    Then 
     $$|Q'||Z(Q)|\le  2^{d-d_G(K)} |KN/N||(G/KN)'||Z(G/KN)|.$$
 \end{enumerate}
\end{Lemma}
\begin{proof}
(i). Here $Q' \cong G'/(N \cap G')$, so $Q^{\mathrm{ab}} = G/(NG^\prime)$. The result now follows from Lemma~\ref{lem:normal_gen_pgps}(i), since   $Q^{\mathrm{ab}}$ is elementary abelian.  

\smallskip

\noindent (ii). 
 Let $a = d(Q)$ and let $W = KN/N$. 
From $N\le \GG$ and Part (i) we deduce that
$$a=d(G)=d_G(K)+d(G/K)\geq d_G(W)+d(Q/W)\geq a,$$
so equality holds throughout. 
Now applying Part~(i) with $Q/W$ in place of $G/N$ gives \[d_G(W) + d(Q/W) =d_G(W) +  \dim{(Q/WQ')} = \dim{(W/[Q,W])}+\dim{(Q/WQ')} =a.\]  
Conversely, since $Q^{\mathrm{ab}}$ is elementary abelian, $\dim{(W/(W\cap Q'))} + \dim{(Q/WQ')} = \dim(Q/Q') = a$, so  from $W\cap Q'\geq [Q,W]$ we deduce that $W\cap Q'=[Q,W]$.  Since $(Q/W)' = Q'/(W \cap Q')$, rearranging yields 
$$|Q'|=|W \cap Q'||(Q/W)'|=|[Q,W]||(G/KN)'|.$$

Temporarily setting $N = 1$ shows that $K \cap G' = [K, G]$, so $K \cap N \le [G, K] \le K$ and we deduce that
$$\frac{W}{[Q,W]}\cong \frac{K[G,K]N}{[G,K]N}\cong \frac{K}{K\cap [G,K]N}\cong \frac{K}{(K\cap N) [G,K]} =  \frac{K}{[G,K]}.$$ 
Since $K$ is elementary abelian, so is $K/[G, K]$, whence $|W/[Q,W]| = |K/[G, K]| = 2^{d_G(K)}$.  Hence  $$|Q'|=|[Q,W]||(G/KN)'| = 2^{-d_G(K)} |W| |(G/KN)'| = 2^{-d_G(K)} |KN/N| |(G/KN)'|.$$

 Now, $G$ acts trivially on the elementary abelian section $Z(Q) \cap 
(KN/N)$ of $K$, so from the definition of $d$ we deduce that $|Z(Q) \cap (KN/N)| \le 2^d$. Hence 
$|Z(Q)|\le |Z(Q) \cap W||Z(Q/W)| \le 2^d|Z(G/KN)|$, and the result follows. 
\end{proof}

The motivation behind our second main result of this section is Theorem~\ref{thm:HomTheorem}, where the invariants $|[F, Q]|$ and $|C_Q(F)|$ for $p$-groups $F\unlhd Q$ play an important role. 

\begin{Theorem}\label{thm:comp}
Let $G$ be an excessive group, and let $N$ be a normal subgroup of $G$ such that $Q:=G/N$ is non-abelian. If $G$ has excess $1$ then assume also that $d(Q)=d(G)$. Then there exists a normal subgroup $F$ of $Q$ such that one of the following holds.
\begin{enumerate}[\upshape(i)]
    \item The group $G$ has excess 2 and degree 8, $N=1$, $F=Q=G$, and
    $|[F, Q]||C_{Q}(F)|^{3}= 2^4 = 2^{n/2}$ with $|C_{Q}(F)|= 2 = 2^{n/8}$. 
    \item The group $G$ has excess $1$ and $|[F, Q]||C_{Q}(F)|\le 2^{n/2}$ with $|C_{Q}(F)|\le 2^{n/4}$.
    \item The group $G$ has excess $2$  and $|[F, Q]||C_{Q}(F)|^{5/4}\le 2^{n/2}$ with $|C_{Q}(F)|\le 2^{n/4}$.
\end{enumerate}
\end{Theorem}
\begin{proof}
If $n\le 16$, then we can use \Magma to construct each $G$ and $N$, and at least one suitable $F$ whenever $Q=G/N$ satisfies the hypotheses. So for the remainder of the proof, assume that $n = 32$, let $G$ have ID $m$, and notice that  $d(G) > 8$.  We shall show that the result holds for some $F$ that is a term in the upper or lower central series of $Q$. We may use Lemma~\ref{lem:GenCompRevised}(ii) to construct all excessive groups $G$.


\noindent (ii). In this case $d(Q)=d(G) = 9$ and we must find an $F$ such that $|[F, Q]||C_Q(F)| \le 2^{16}$ and $|C_Q(F)| \le 2^8$.  We start by bounding $|Z(Q)|$, so let $Z$ be the full preimage of $Z(Q)$ in $G$.
From $|G:Z| = |Q:Z(Q)|$ and $Q$ non-abelian we deduce that $|G:Z|>2$, so that by Lemma~\ref{lem:GenCompRevised}(ii) we can bound 
$|Z/[G,Z]| \le 2^8$, i.e. 
\begin{align}\label{lab:Zbound}
 |Z(Q)|\le 2^{8} = 2^{n/4}.   
\end{align} 
 
We check the groups of excess $1$ and degree $32$ and compute that in each case  $\Gab$ is elementary abelian,
so it follows from Lemma~\ref{lem:precomp}(i) that $N\le \GG$. 
Write $|\GG|=2^{\ell}$: we can check using \Magma that $\ell\le 16$.
If $|N| = |N\cap \GG|\geq 2^{\ell-8}$, then since $G'/N = G'N/N=Q'$, the result follows from \eqref{lab:Zbound} by taking $F=Q$. So we may assume that 
\begin{equation}\label{eq:nsmall}
|N|< 2^{\ell-8},
\end{equation}
and in particular that $\ell \ge 9$. 

We can check using \Magma that all $21$ of the groups $G$ of excess $1$ with $\ell \ge 9$ satisfy $|Z(G)|=2$ and $|Z_2(G)|\le 16$ (here $Z_2(G)$ is the second term in the upper central series).  
Thus we may compute the set of $G$-normal subgroups $N$ of $Z_2(G)$ and verify that for each $Q = G/N$ there is at least one suitable $F$ in the upper or lower central series. 
 Any normal subgroup of $G$ of order at most $2^k$ lies in $Z_k(G)$, so in particular this completes the arguments when $|N| \le 4$. 

We may therefore assume that $|N| \ge 8$, so $\ell\geq 12$ by \eqref{eq:nsmall}. There are only seven such groups $G$, 
and for each of them $|Z_3(G)| \le 256$. We compute all $G$-normal subgroups $N$ of $Z_3(G)$ and check that at least one term in the upper central series of the quotient $G/N$ is a suitable group $F$. In particular this completes the arguments for $|N| = 8$. 
%

We may now assume further that $N\not\le Z_3(G)$,  since $|N| \ge 16$, and that $\ell \geq 13$ by \eqref{eq:nsmall}. There are only four such groups $G$, and we shall show that for each $G$ the result holds with $F = Q$. Let $\Sigma$ be a set of minimal blocks for $G$, and let $K$ be the kernel of the action of $G$ on $\Sigma$. We compute that in each case $d_G(K) = 3$ and $d(G/K) = 6$, so 
$d(G)=d_G(K)+d(G/K)$ and Lemma~\ref{lem:precomp}(ii) applies to $G$, $K$ and $N$.

We next compute the maximum $d$  of $d_{G^{\Sigma}}(U)$ over all submodules $U$ of the natural $\mathbb{F}_2[G^{\Sigma}]$-permutation module (which includes all submodules of $K$).
If $N\not\le K$, then all $N^{\Sigma}$-orbits have length at least $2$, so since $N$ acts trivially on $KN/N$, we deduce that $|KN/N|\le 2^{d_N(K)}\le 2^{n/4}=2^8$ by \cite[Lemma 1.7]{GMP}. Furthermore, $|(G/KN)'||Z(G/KN)|$ is bounded by the maximum over all normal subgroups $M$ of $G^{\Sigma}$ of
$|(G^{\Sigma}/M)'||Z(G^{\Sigma}/M)|$, and we may compute this maximum and denote it $2^s$.
Applying Lemma~\ref{lem:precomp}(ii) now yields $|Q'||Z(Q)|\le 2^{d-3+b+s}$, which is at most $2^{16}$ in each case.

If instead $N\le K$, then from $|N|\geq 16$ we deduce that $|KN/N|\le b:= |K|/16$. 
Furthermore, in this case we can simply let $2^s = |(G/KN)'||Z(G/KN)| = |(G^{\Sigma})'||Z(G^{\Sigma})|$.
Again, applying Lemma~\ref{lem:precomp}(ii) yields $|Q'||Z(Q)|\le 2^{d-3+b+s}$, which is at most $2^{16}$, as required.

\smallskip
\noindent  (iii). For each $G$ of excess 2, we use \Magma to check that 
$|\GG|\le 2^9$, so
 we can compute all the possibilities for $L:= N \cap \GG$. The group $Q$ is a quotient of $\overline{G}:=G/L$, and  
letting $\overline{N} = N/L$ we deduce that $\overline{\GG}\cap \overline{N}=1$. Thus $[\ol{G}, \ol{N}] = 1$, i.e. $\ol{N}\le Z(\ol{G})$.  We use \Magma to compute each such $\overline{N}$, and to check that each $Q:= \overline{G}/\overline{N}$ has a suitable term $F$ in its upper central series.
\end{proof}

\section{The reduction to bounded orbit lengths}\label{sec:Bounded}

In this section we first establish tight bounds on the number of generators of groups with no short orbits, and then prove the surprising result that to count the subgroups of $\Sn_n$ satisfying a given 
property $\mathcal{P}$, it often suffices to count only those subgroups with all orbits of length less than $C$, for some $C$ depending on $\mathcal{P}$.

 The next three results are based on ideas from \cite{Tracey18}. 
For brevity, we prove them under the assumption that the group is soluble: it is not difficult to prove 
them (with the same constants) for insoluble groups. 

\begin{Lemma}\label{lem:NormalGens}
There exists an absolute constant $\zeta_0$ such that for all $n>1$, all soluble transitive $G\le \Sn_n$,  and all normal subgroups $N$ of $G$, 
    \[d_G(N)\le 
    \frac{\zeta_0 n}{\sqrt{\log{n}}}.\]
\end{Lemma}
\begin{proof}
In the proof below, we shall encounter absolute constants $\kappa$ and $\eta$, both from \cite{Tracey18}. We shall prove that the result holds with $\zeta_0$ minimal subject to $\zeta_0 \ge 1$ and  \begin{align*}\label{eq:zetadef}
\frac{(\kappa\eta \log{r}+\zeta_0)s}{\sqrt{\log{s}}}\le \frac{\zeta_0 rs}{\sqrt{\log{rs}}}
\end{align*}
 for all $r,s \in \mathbb{Z}_{\geq 2}$. 

We induct on $n$. If $G$ is primitive then 
$d(N)\le 1+\log{n} \le \zeta_0 n/\sqrt{\log{n}}$ for every 
normal subgroup $N$ of $G$, by \cite[Theorem 1.1]{HRD}. 
Hence the result follows in this case, and  for all $n \le 3$.

For the inductive step, assume that $G$ is imprimitive, with a minimal block system of $s$ blocks of size $r$ say. Then there exist a primitive soluble group $R\le \Sn_r$, and a transitive soluble group $T\le \Sn_{s}$, such that $G\le R\wr T$ and $G$ projects onto $T$. Let $K = G \cap R^s$, and 
write $\mathrm{cl}(R)$ for the composition length of $R$.

By \cite[Lemma 5.8]{Tracey18}, since $G$ is soluble  there is a $G$-normal series
\begin{equation*}
1= N_0\le N_1\le \cdots\le N_{\ell} =N\cap K
\end{equation*}
such that for each $i$ the quotient $M_i = N_i/N_{i-1}$ is 
isomorphic to an $\mathbb{F}_{p_i}[G]$-submodule  of an induced module $U_i\uparrow^G_{H}$, with $p_i$ prime, $\sum_{i=1}^{\ell} \dim_{\mathbb{F}_{p_i}}{U_i}\le \mathrm{cl}(R)$, and $|G:H| = s$. Furthermore, by \cite[Lemma 2.17, Corollary 4.26]{Tracey18} there exists an absolute constant $\kappa$ such that each such  $M_i$ satisfies 
$d_G(M_i) \leq \kappa s\dim_{\BBF_{p_i}}{U_i}/\sqrt{\log{s}}$. The group $R$ is primitive of degree $r$, so $\cl(R) \le \eta \log r$ for some absolute constant $\eta$, by \cite[Theorem~2.10]{Pyber}.
Therefore 
\begin{equation*}
d_G(N \cap K) \le \sum_{i = 1}^{\ell} d_G(M_i)\le \sum_{i = 1}^{\ell} \frac{\kappa s\dim_{\BBF_{p_i}}{U_i}}{\sqrt{\log{s}}}
\le \frac{\kappa s}{\sqrt{\log{s}}} \sum_{i = 1}^{\ell} \dim_{\BBF_{p_i}}{U_i} \le 
\frac{\kappa s \mathrm{cl}(R)}{\sqrt{\log{s}}} \le \frac{\kappa \eta s\log{r}}{\sqrt{\log{s}}}.
\end{equation*}

Now, $G/K$ embeds transitively in $\Sn_s$, so by induction $d_{G/K}(NK/K) \le \zeta_0 s/\sqrt{\log s}$. 
Hence
\begin{align*}
d_G(N) & \leq  d_G(N \cap K) + d_{G/K}(NK/K) 
 \le \frac{\kappa \eta s \log r}{\sqrt{\log s}} + \frac{\zeta_0 s}{\sqrt{\log s}} = \frac{(\kappa \eta \log r + \zeta_0)s}{\sqrt{\log s}}.\end{align*} 
Hence the result follows. 
\end{proof}

Here is our main result on generating groups with large orbits: the assumption of solubility can be removed.

\begin{Proposition}\label{prop:NormalGens}
There exists an absolute constant $\zeta$ such that the following holds. For all $n \ge 1$ and all non-trivial soluble subgroups  $G$ of $\Sn_n$, if all non-trivial orbits of $G$ have length at least $m$, then $d(G)\le \zeta n/\sqrt{\log{m}}$.  In particular, for any $\lambda > 0$,  
letting  $C = 
2^{(\zeta/\lambda)^2}$
ensures that each soluble 
$G \le \Sn_n$ with all non-trivial orbit lengths at least $C$ satisfies $d(G) \le \lambda n$.
\end{Proposition}

\begin{proof} We shall show that the result holds with $\zeta = \zeta_0$, the constant from Lemma~\ref{lem:NormalGens}.

Let $G\le \Sn_n$ be a counterexample  with $n$ minimal, and let $m$ be a lower bound on the length of the non-trivial orbits of $G$. By Lemma~\ref{lem:NormalGens}, 
$G$ has more than one non-trivial orbit, so $G$ is a subdirect product of $G_1 \times G_2$ for some soluble groups $G_1 \le \Sn_{n_1}$ and $G_2 \le \Sn_{n - n_1}$, such that $G_1$ is transitive and $n_1 \ge m$.

 Let $N =G\cap G_1\unlhd G_1$, and identify $N$ with the corresponding subgroup of $\Sn_{n_1}$. Then
 $d_G(N) = d_{G_1}(N)$, so
 by Lemma~\ref{lem:NormalGens} 
\begin{align*}
  d(G)& 
  \le d(G/N) + d_G(N) = d(G_2) + d_{G_1}(N) 
  \le d(G_2)+ \zeta n_1/\sqrt{\log{n_1}}.
\end{align*}
Now  $d(G_2) \le \zeta (n-n_1)/\sqrt{\log{m}}$, since $G_2$ has all nontrivial orbits of length at least $m$, and $G$ is a  minimal counterexample. The result follows from $n_1 \ge m$. 
\end{proof}

We also record a second easy corollary of Lemma~\ref{lem:NormalGens}.
\begin{Corollary}\label{cor:NormalGens}
Let $G\le \Sn_n$ be soluble, with non-trivial orbit lengths $n_1,\hdots,n_t$, and let $N\unlhd G$. Let $\zeta$ be the constant from Proposition~\ref{prop:NormalGens}. Then $d_G(N)\le \zeta \sum_{i=1}^tn_i/\sqrt{\log{n_i}}$.   
\end{Corollary}
\begin{proof}
We induct on $t$: the case $t=1$ is Lemma \ref{lem:NormalGens} (recalling from the proof of Proposition~\ref{prop:NormalGens} that $\zeta = \zeta_0$). For $t>1$, the group $G$ embeds is a subdirect product of $G_1\times G_2 \le \Sn_{n_1} \times \Sn_{n-n_1}$, where $G_1$ is transitive
and $G_2$ has non-trivial orbit lengths $n_2,\hdots,n_t$. 
Let $\rho$ denote the projection to $G_2$, then $d_G(N)\le d_G(N\cap G_1)+d_G(N\rho)$. As in the previous proof,  $d_G(N \cap G_1) = d_{G_1}(N \cap G_1) \le \zeta n_1/\sqrt{\log{n_1}}$, whilst $d_G(N \rho) = d_{G_2}(N \rho)$. The result follows by induction.
\end{proof}

 Here is the main result of this section.  
 
\begin{Theorem}\label{thm:BoundedOrbits}
Fix a real number $\epsilon> 0$ and a group theoretic property $\mathcal{P}$ such that:
\begin{enumerate}[\upshape(i)]
    \item $\mathcal{P}$ is closed under taking subgroups and quotients;
    \item there exist constants $\delta = \delta_{\mathcal{P}}$ and $\gamma = \gamma_{\mathcal{P}}$ such that for all $n\in\mathbb{Z}_{>0}$, there exists a (possibly empty) set $\mathcal{C}_n$  of transitive subgroups of $\Sn_n$, each of order at most $2^{\delta n}$, and lying in at most $2^{\gamma n}$ conjugacy classes,  such that each transitive $\mathcal{P}$-group is a subgroup of at least one group in $\mathcal{C}_n$;
    \item there exists a constant $C$ such that for all $n \ge 1$, if $G \le \Sn_n$ is a $\mathcal{P}$-group with all non-trivial orbits of length at least $C$ then $d(G) \le \epsilon n/\delta$.
\end{enumerate}
Let $f: \mathbb{Z}_{\ge 0} \rightarrow \mathbb{R}_{\ge 0}$ be a  non-decreasing function such that for all partitions $n = n_1  + \cdots + n_t$ with $n_i < C$ for each $i$,  and all groups $H_i \in \mathcal{C}_{n_i}$, we can bound $|\Sub_{\mathcal{P}}(H_1 \times \cdots \times H_t)| \le 2^{\epsilon n^2 + f(n)}$.  
 Then 
 \[|\Sub_{\mathrm{\mathcal{P}}}(\Sn_n)|\le 2^{\epsilon n^2+f(n)+n\log{n}+\gamma n+ 4 \sqrt{n}} \leq 2^{\epsilon n^2 + f(n) + (5 + \gamma) n \log n}.\] 
\end{Theorem} 

By Proposition~\ref{prop:NormalGens}, if $\mathcal{P}$ is a subclass of the soluble groups then a suitable value of $C$ always exists. (In fact, since Proposition~\ref{prop:NormalGens} holds without the assumption of solubility, such a $C$ is guaranteed to exist for any property $\mathcal{P}$.) 
Theorem~\ref{thm:BoundedOrbits} 
tells us that, as long as the property $\mathcal{P}$ 
satisfies Conditions (i)--(iii), then to prove that $|\Sub_{\mathcal{P}}(\Sn_n)| \le 2^{\epsilon n^2+o(n^2)}$, it suffices to prove that such a bound holds for the number of $\mathcal{P}$-subgroups of $\Sn_n$ with bounded orbit lengths, i.e. to find a suitable function $f(n) \in o(n^2)$.  
Moreover, Conditions (i)--(iii) hold for a number of natural properties $\mathcal{P}$:
for example, if $\mathcal{P}$ is the property of being a $p$-group, then we can let $\mathcal{C}_n$ be the set of Sylow $p$-subgroups if $n$ is a $p$-power, and be empty otherwise.

\begin{proof}[Proof of Theorem~\ref{thm:BoundedOrbits}]
For a partition $n = n_1 + \cdots + n_t$ with $n_1 \le n_2 \le \cdots \le n_t$, let $\mathcal{S}(n_1,\hdots,n_t)$ be the set of $\mathcal{P}$-subgroups  of $\Sn_n$ with orbit lengths $n_1,\hdots,n_t$.
The number $\mathrm{Part}(n)$ of integer partitions of  $n$ is less than $2^{4 \sqrt{n}}$ (see \cite[p172]{Ayoub}), so
\[|\Sub_{\mathcal{P}}(S_n)| \leq \sum_{n = n_1 + \cdots + n_t}|\mathcal{S}(n_1,\hdots,n_t)| \le 2^{4\sqrt n} \max_{n = n_1 + \cdots + n_t} \{|\mathcal{S}(n_1,\hdots,n_t)|\}.\] 

Fix $n_1, \ldots, n_t$ attaining this maximum value, and let
 $\mathcal{D}$ be the set of subgroups of $\Sn_n$ of the form $L_1\times\cdots\times L_t$, where $L_i \in \mathcal{C}_{n_i}$. By Condition (ii), the set
 $\mathcal{D}$ contains at most $2^{\gamma n}$ conjugacy classes of groups, so 
 $|\mathcal{D}|\le 2^{\gamma n+n\log{n}}$.
 For $L_i \in \mathcal{C}_{n_i}$,  let $\mathcal{T}(L_1, \ldots, L_t)$ be the set of $\mathcal{P}$-subgroups $G$ of $L_1\times\cdots\times L_t$ with the property that the coordinate projection $G\pi_i$ is transitive for all $i$. 
Then each $G \in \mathcal{S}(n_1,\hdots,n_t)$  lies in $\mathcal{T}(L_1, \ldots, L_t)$ for some $L_1 \times \cdots \times L_t \in \mathcal{D}$.  
Therefore
\begin{align*}
\max_{n = n_1 + \cdots + n_t} \{|\mathcal{S}(n_1,\hdots,n_t)|\}
& \leq 2^{\gamma n + n \log n} \max_{L_1 \times \cdots \times L_t \in \mathcal{D}}\{|\mathcal{T}(L_1, \ldots, L_t)|\}.  
\end{align*}

   Fix one such partition of $n$,  let $r \in \{0, \ldots, t\}$ be such that  $n_i <  C$ if and only if  $i\le r$,
   and fix $L_i \in \mathcal{C}_{n_i}$.
Let $d_1 =\sum_{i\le r}n_i$ and $d_2 = n-d_1$. 
Let $D_1=L_1\times\cdots\times L_r$ and $D_2=L_{r+1}\times\cdots\times L_t$, with $D_i$ trivial if $d_i = 0$. 
Let $H$ be a $\mathcal{P}$-subgroup of $D_1$, and let $T$ be a $\mathcal{P}$-subgroup of $D_2$ such that $T\pi_i$ is transitive for all $i \in \{r+1, \ldots, t\}$ (the group $T$ is trivial if $d_2 = 0$).  
Notice that Condition (ii) bounds $|N_{D_1}(H)/H|\le |D_1| \le 2^{\delta d_1}$  and $|D_2| \le 2^{\delta d_2}$, while   $d(T)\le \epsilon d_2/\delta$ by Condition (iii). Thus
\begin{align*}
   |\Hom(T,N_{D_1}(H)/H)|\le (2^{\delta d_1})^{\epsilon d_2/\delta }=2^{\epsilon d_1d_2}. 
\end{align*}
There are at most $2^{\epsilon d_1^2+f(d_1)}$ groups $H$ by definition of $f$, and there are at most $|D_2|^{\epsilon d_2/\delta} = 2^{\delta d_2 \epsilon d_2 /\delta} = 2^{\epsilon d_2^2}$ choices for $T$. 
%
Hence by 
Lemma~\ref{lem:Goursat}(ii) 
\begin{align*}
|\Sub_{\mathcal{P}}(S_n)| & \le 2^{4 \sqrt{n}  + \gamma n + n \log n}|\mathcal{T}(L_1, \ldots, L_t)|\\
& \le 2^{4 \sqrt{n}  + \gamma n + n \log n + \epsilon d_1^2+f(d_1)+\epsilon d_2^2 + \epsilon d_1d_2} 
 \le 2^{\epsilon n^2 + f(n)+n \log n + \gamma n + 4 \sqrt{n} }
\end{align*}and the proof is complete.
%
%
%
\end{proof}

\section{The proofs of Theorems~\ref{thm:Pyberp} and \ref{thm:Pybernilp}}\label{sec:PyberProof}

 In this section we shall 
 prove Theorems~\ref{thm:Pyberp} and \ref{thm:Pybernilp}. We will first prove the lower bounds and then the upper bounds. 

We start by dealing with some small cases. 

\begin{Lemma}\label{lem:helper}
Let $R$ be the set of triples $(p, m, n)$ of positive integers, where $p$ is prime, $m \ge 2$, $mp \le n < (m+1)p$, and $1 - 4mp^{-\lfloor m/2 \rfloor} \le 1/2$. 
Then  there exists an $\alpha_R > 0$ such that  
$|\Sub_p(\Sn_n)| \ge  p^{n^2/4p^2 + \alpha_R n \log n}$ for all $(p, m, n) \in R$.
In particular, this bound holds for all $n < 24$ and $p \le n/2$. 
\end{Lemma}

\begin{proof}
 Let $(p, m)$ be as above.  Then $1 - 4mp^{-\lfloor m/2 \rfloor} \le 1/2$
    if and only if  $m = 2$ and $p \le 13$, or $m = 3$ and $p \le 23$, or $m \in \{4, 5\}$ and $p \le 5$, or $m \in \{6, 7\}$ and $p \in \{2, 3\}$ or $m \leq 13$ and $p = 2$. Thus the set $R$ contains these values of $p$ and $m$ together with each $n$ such that $m = \lfloor n/p \rfloor$. 
    
    For each such $(p, m)$ we start by using \Magma to construct the elementary abelian group $G:= C_p^m \le \Sn_{mp}$. Each subdirect product $H$ of $G$ acts as $C_p$ on each orbit, so any two such subdirect products $H$ and $H_1$ that are $\Sn_{mp}$-conjugate are also $N_{\Sn_{mp}}(G)$-conjugate, while 
    groups with different orbit projections are not conjugate in $\Sn_{mp}$. Thus, the number of 
    $p$-subgroups of $\Sn_n$ is at least the product of the number of subdirect products of $G$ and $|\Sn_{n} : \AGL_{1}(p) \wr \Sn_m|$.

    For $mp \le 10$  we simply count the subdirect products of $G$ and see that for each $(p, m, n) \in R$ this yields strictly more than $p^{n^2/4p^2}$ $p$-groups.
    For larger values of $mp$ we approximate the number of subdirect products of $G$ by observing that each subgroup that is \emph{not} subdirect projects trivially to at least one of its orbits of length $p$. Thus we count the $\lfloor m/2 \rfloor$-spaces of $\BBF_{p}^m$, and subtract from this the product of $m$ and the number of $\lfloor m/2 \rfloor$-subspaces of $\BBF_{p}^{m-1}$. 
    
    If $n < 24$ and $p$ is any prime at most $n/2$ then $(p, \lfloor n/p \rfloor, n) \in R$, so the final claim follows.
\end{proof}

\begin{Lemma}\label{lem:helper2}
There exists an absolute constant $\alpha_0 < 1$ such that the following holds. Let $n\geq 24$, let $p\le n/2$ be prime, and let  $m=\lfloor n/p\rfloor$. Then
$n^{\alpha_0(n-m)} \ge (p-1)^{m+p-1/2}p^{(m+1)/2} e^{n-p-m+3}$.   
\end{Lemma}
\begin{proof}
We shall prove that this holds with $\alpha_0 = 448/453$.
Write $\ln$ for natural logarithms, and notice first that since $n/p - 1 < m \le n/p$, 
\begin{align*}
\ln \left((p-1)^{m+p-\frac{1}{2}}p^{\frac{m+1}{2}} e^{n-p-m+3}\right) & = \left(m+p-1/2\right)\ln(p-1) + \left(\frac{m}{2} + \frac{1}{2}\right) \ln p + (n-p-m+3)\\
& \le  \left(\frac{n}{p}+p-\frac{1}{2}\right)\ln(p-1)+\left(\frac{n}{2p}+\frac{1}{2}\right)\ln p+(n-p-\frac{n}{p}+4)
\end{align*}
and we let $a(n, p)$ denote the final right hand side above. We bound
$\ln(n^{n-m}) = (n-m)\ln n \ge b(n, p):=  (n - n/p + 1) \ln n$.


For $p \ge 7$ we shall use the fact that $n\geq 24$ to bound $3 \le \ln{n}$, giving 
\begin{align*}
a(n, p) &
\le \left(\left(\frac{n}{p}+p-\frac{1}{2}\right)+\left(\frac{n}{2p}+\frac{1}{2}\right)+\left(\frac{n}{3}-\frac{p}{3}-\frac{n}{3p}+\frac{4}{3}\right)\right)\ln{n}
\le \left(\frac{7n}{6p} + \frac{n}{3} + \frac{2p}{3} +  \frac{4}{3}\right)\ln n.
\end{align*}
Then using  $7 \le p \le n/2$ and $n \ge 24$  we deduce that 
\begin{align*}
    \frac{a(n,p)}{b(n,p)} & \le  \frac{(7/(6p))n + n/3 + 2p/3 +4/3}{n - n/p + 1} 
    \le \frac{(7/(6p))n + n/3 + n/3 + 4/3}{(6/7)n + 1} & \le \frac{(5/6)n + 4/3}{(6/7)n + 1} \le \frac{448}{453}.
    \end{align*}

For $p=2$ our ratio $a(n,2)/b(n, 2)$ simplifies to
$$\frac{(n/4+1/2)\ln 2 +(n/2+2)}{(n/2+1)\ln n}=\frac{\ln 2}{2\ln n}+\frac{1}{\ln n}+\frac{1}{(n/2+1)\ln n} < 0.44$$
where the upper bound comes from setting $n = 24$. Similarly, $a(n,3)/b(n,3) < 0.51$ and $a(n,5)/b(n,5) < 0.58$, so the result follows.
\end{proof}

 
\begin{Proposition}\label{prop:H}
Let $p$ be prime. There exists a constant $\alpha_p > 0$  such that if $n \geq p$, then $|\Sub_p(\Sn_n)|\geq p^{n^2/4p^2+\alpha_p n\log{n}}$. In particular, there exists an absolute constant $\alpha > 0$ such that $|\Sub(\Sn_n)|\geq 2^{n^2/16+\alpha n\log{n}}$ whenever $n>1$.
\end{Proposition}
\begin{proof}
Let $\alpha_R$ be as in Lemma~\ref{lem:helper} and let $\alpha_0$ be as in Lemma \ref{lem:helper2}.
We shall show that the result holds for 
$$\alpha_p:=\min\left\{0.08, \frac{(1-(2p-1)^2/(4p^2))}{2p\log{(2p)}},\alpha_R,(1-\alpha_0)(1-1/p)\right\},$$
noting that these are all positive.

Suppose first that $n < 2p$. For $n \le 5$ we compute the $p$-subgroups of $\Sn_n$ and check that there are strictly more than $p^{n^2/4p^2 + 0.08n\log n}$, so assume that $p \ge 5$. 
Then  $(p-2)! > p$ so $\Sn_n$ has $n!/(p(p-1)(n-p)!) = \binom{n}{n-p} \cdot (p-2)!\geq p$  nontrivial $p$-subgroups. Now $n/2p < 1$, so
\begin{align*}
\frac{1-(2p-1)^2/4p^2}{2p\log{(2p)}} & \le  \dfrac{1-\frac{n^2}{4p^2}}{2p\log{(2p)}}= \dfrac{\log_p{(pp^{-n^2/4p^2})}} {2p\log{(2p)}} 
\le\dfrac{\log_p{(\binom{n}{n-p} \cdot (p-2)!p^{-n^2/4p^2})}}{n\log{n}}\\ & < \dfrac{\log_p{(|\Sub_p(\Sn_n)|p^{-n^2/4p^2})}}{n\log{n}},
\end{align*}
as required.


Now, suppose that $m: = \lfloor n/p \rfloor \ge 2$ and $(p, m, n) \not\in R$, so that in particular $n\ge 24$.
Let $G = G_1\times\cdots\times G_{m}\le\Sn_n$ be  generated by $m$ disjoint $p$-cycles, so that that $G \cong C_p^{m}$. We now reason similarly to the proof of Lemma~\ref{lem:helper}.
By Lemma~\ref{lem:AnerCount}(i), for all $d \in \{1, \ldots,  m\}$ the number of subgroups of $G$ of order $p^d$ is at most  $4 p^{(m-d)d}$ and  at least $p^{(m-d)d}$. It follows that 
the number of subdirect products of $G$ of order $p^d$ is at least $p^{(m-d)d}-4 m p^{(m-d-1)d} = p^{(m-d)d}(1 - 4m p^{-d}) 
\ge \frac{1}{2}p^{(m^2 - 1)/4}$, by Lemma~\ref{lem:helper}.
From $n/p < m+1$ we deduce that $\frac{1}{2} p^{(m^2-1)/4} \ge \frac{1}{2} p^{(n^2/p^2 - 2n/p)/4} \ge  p^{n^2/4p^2} \cdot \frac{1}{2p^{(m+1)/2}}$.  
As in the proof of Lemma~\ref{lem:helper}, the sets $\{A\text{ : }A\le G\text{ subdirect of order }p^{\lfloor m/2 \rfloor}\}^g$, as $g$ runs over  a set of coset representatives for  $N_{\Sn_n}(G)$  in $\Sn_n$, are pairwise disjoint. Thus, the number of $p$-subgroups of $\Sn_n$ is at least $p^{n^2/4p^2} \cdot (|\Sn_n : N_{\Sn_n}(G)|\frac{1}{2p^{(m+1)/2}})$. Next, by Stirling's approximation
\begin{equation*}
\sqrt {2\pi d}\ \left({\frac {d}{e}}\right)^{d} 
<d!< 1.1 {\sqrt {2\pi d}}\ \left({\frac {d}{e}}\right)^{d}
\end{equation*}
(valid for all integers $d \ge 1$),
and the bounds $p \leq n/2$ and $n-mp \le p-1$, we see that 
\begin{align*}
|\Sn_n: N_{\Sn_n}(G)|\frac{1}{2p^{(m+1)/2}} & =\frac{n!}{(p(p-1))^{m} m! (n-mp)!} \frac{1}{2p^{(m+1)/2}}\\
& \geq \frac{n!}{(p(p-1))^{m} m! (p-1)!} \frac{1}{2p^{(m+1)/2}}\\
& \geq \frac{n^n }{2 (1.1)^2 \sqrt{2 \pi} (p-1)^{m+p -1/2} p^{m} m^m e^{n - m - p + 1} } \frac{1}{p^{(m+1)/2}}\\
& \geq \frac{n^{n-m}}{(p-1)^{m+p-1/2}p^{(m+1)/2} e^{n-m-p+3}} \ge \frac{n^{n-m}}{n^{\alpha_0(n-m)}}\\
& \geq n^{(1-\alpha_0)(n-m)}\geq n^{(1-\alpha_0)(n-n/p)} = n^{n(1-\alpha_0)(1-1/p)} \ge n^{\alpha_pn},
\end{align*}
where the final line follows 
by definition of $\alpha_0$ (see Lemma \ref{lem:helper2}). This completes the proof. 
\end{proof} 

%
%

We shall prove our upper bounds via a series of propositions, each of which will bound $|\Sub(G_1\times\cdots\times G_t)|$ for various groups $G_i$. Recall Definition~\ref{def:tnst} of a lower triangular normal subgroup tableau, 
of a lower coordinate tableau, 
and of the set $\mathcal{L}(G_1,\hdots,G_t)$. Recall also Definition~\ref{defn:k-projecting} of a $k$-projecting subdirect product.

We start with the hardest case, namely when the groups $G_i$ are all $2$-groups. 

\begin{Proposition}\label{prop:Pyber2} Let $C$ be any integer greater than two. Then there exists a constant $\tau = \tau(2, C)$ such that the following holds. Let $n = n_1 + \cdots +  n_t$ be a partition of $n$ into 2-power parts with $2 \le n_i < C$ for all $i$. 
Then for all transitive $2$-groups $G_i \le \Sn_{n_i}$,  \[|\Subdir(G_1\times\cdots\times G_t)|\le 2^{n^2/16 +\tau n\log{n}}.\]
\end{Proposition}

\begin{proof}
For normal subgroups $N_i$ of $G_i$, let
$\mathcal{S}(N_1,\hdots,N_t)$ be the set of subdirect products $G$ of $G_1\times\cdots\times G_t$ such that 
$G \cap G_i = N_i$, and let $Q_i = G_i/N_i$. 
The largest task of the proof will be 
to show that there exists a constant $\cEight = \cEight(C)$ 
such that for all normal subgroups $N_i$ of $G_i$ 
\begin{align}\label{lab:preLClaim}
    |\mathcal{S}(N_1,\hdots,N_t)|\le 2^{n^2/16+\cEight n\log{n}}|\mathcal{L}(Q_2,\hdots,Q_t)|.
\end{align}
We shall prove that this holds with $\cEight = 1 + \cSeven + \log \cSevenA+ \cSix$, where $\cSeven = \cSeven( 2, C)$ is from Theorem~\ref{thm:InitialSubdirect}, and both $\cSevenA = \cSevenA(2, C-1)$ and $\cSix = \cSix(2, C-1)$ are from Theorem~\ref{thm:HomTheorem}. 

\medskip

\noindent\textbf{Case I.} 
Suppose first for each $i \in \{1, \ldots,  t\}$, if $G_i$ has a normal subgroup $M_i$ containing $N_i$ such that $d_{G_i}(M_i/N_i)> n_i/4$ then the quotient $Q_i$ is abelian.

Index the $G_i$ so that $G_i$ has such an $M_i$ if and only if $i \le r$, 
and let $A = Q_1 \times \cdots \times Q_r$  and 
 $d_1:=\sum_{i\le r}n_i$, so that $A$ is trivial and $d_1 = 0$ if $r = 0$.
 
 Lemma~\ref{lem:tableaux} bounds $|\mathcal{S}(N_1,\hdots,N_r)| \le |\Subdir(A)| \le |\Sub(A)|$. Additionally, 
$|A|\le 2^{d_1/2}$ by Theorem~\ref{thm:KovPrae},  so if $d_1 \neq 0$ then
 $|\Sub(A)|\le 2^{d_1^2/16+\log{(d_1/2)}+2} \leq  2^{d_1^2/16 + d_1\log d_1}$, by Lemma~\ref{lem:AnerCount}(i). 
If $r=t$, then $d_1=n$ and (\ref{lab:preLClaim}) follows, so assume that $0 \le r<t$.  

We apply the ``In particular" part of 
Theorem~\ref{thm:InitialSubdirect} to $Q_{r+1}\times\cdots\times Q_t$ with $d = 1/4$  to see that 
\begin{equation*}
|\Subdir(Q_{r+1} \times \cdots \times  Q_t)|\le  2^{(n-d_1)^2/16 + \cSeven (t-r) \log (t-r)}.\end{equation*}
By Lemma~\ref{lem:tableaux}, $|\mathcal{S}(N_{r+1},\hdots,N_t)|\le |\Subdir(Q_{r+1} \times \cdots \times  Q_t)|$, so in particular the result now follows if $r = 0$: assume from now on that $0 < r < t$.
Our indexing of the groups means that if $L$ is a subdirect product of $Q_{r+1}\times\cdots\times Q_t$ then  $d(L)\le (n-d_1)/4$ by Lemma~\ref{lem:generate_subdirect}, whilst each section $Y$ of $A$ satisfies $|Y|\le |A| \le 2^{d_1/2}$. We therefore deduce that $|\Hom(L,Y)|\le 2^{d_1(n-d_1)/8}$ for all such $L$ and $Y$. 

Each subdirect product of $Q_1 \times \cdots \times Q_t$ projects to a subdirect product of $Q_{r+1} \times \cdots \times Q_t$, so we now apply 
Goursat's Lemma~\ref{lem:Goursat}(\ref{goursat_property}) with $\mathcal{P}$ being the property of being subdirect to count 
\begin{align*}
 |\mathcal{S}(N_1,\hdots,N_t)| &\le |\Subdir(Q_1 \times \cdots \times Q_r \times Q_{r+1} \times \cdots\times Q_t))| \quad \quad \mbox{ by Lemma~\ref{lem:tableaux}}\\
 & \le |\Sub(A)|\cdot |\Subdir(Q_{r+1}\times\cdots\times Q_t)| \\
 & \quad  \cdot \max\{|\Hom(L, Y)| \text{ : } Y \mbox{ a section of } A, \  L \in \Subdir(Q_{r+1} \times \cdots \times Q_t)\}\\
 &\le 2^{d_1^2/16+d_1\log{d_1}+(n-d_1)^2/16+\cSeven (t-r)\log{(t-r)}+d_1(n-d_1)/8}\le 2^{n^2/16+ (1 + \cSeven)n\log{n}}.
\end{align*}
Claim (\ref{lab:preLClaim}) now follows. 

\medskip

\noindent\textbf{Case  II.} 
Now, without loss of generality,  $G_1$  has a normal subgroup $M_1$ containing $N_1$ such that $d_G(M_1/N_1)> n_1/4$ and $Q_1 = G_1/N_1$ is non-abelian. We shall induct on the number of such factors $G_i$: the base case of none follows from Case I.  The pair $(G_1, M_1)$ is excessive, as in Definition~\ref{def:excessive}. We index the $G_i$ so that precisely one of the following holds, and define related variables $e$ and $f$. 
\begin{enumerate}
    \item[Case (IIa)] $G_1$ has excess 2 and degree 8, so that $N_1=1$ (to ensure that $Q_1$ is non-abelian). We set $e = 3$ and $f = 3/8$. 
    \item[Case (IIb)] $G_1$ has excess 2, and if there exists an $i$ such that $G_i$ has excess 2 and degree 8 then 
     $Q_i$ is abelian. We set $e = 5/4$ and $f = 5/16$. 
    \item[Case (IIc)]  $G_1$ has excess 1,  and if there exists an $i$ such that  $G_i$ has excess 2 then $Q_i$ is abelian. We set $e = 1$ and $f = 1/4$. 
\end{enumerate}
Next, notice that 
\[\mathcal{S}(N_2, \ldots, N_t) \subseteq \mathcal{L} := \{L \in \Subdir(G_2 \times \cdots \times G_t) \text{ : } L \cap G_i \ge N_i\},\] and temporarily fix an $L \in \mathcal{L}$. 
Suppose that $L$ is $(t-r)$-projecting, and 
re-index the $G_i$ for $i > 1$ so that $L \pi_{\{r+1, \ldots, t\}} = G_{r+1} \times \cdots \times G_t$. 

Our next task is to use  the ``in particular" statement of Theorem~\ref{thm:HomTheorem} to count the epimorphisms from $A:= L/(N_2 \times \cdots \times N_t)$ to $Q:= Q_1$. We shall put the projection of $A$ to 
$Q_2 \times \cdots \times Q_r$ in place of $R$, and to $Q_{r+1} \times \cdots \times Q_t$ in place of $P$.
Each $Q_i$ has order at most $2^{C-2}$, since $n_i \le C$, so we may set $k = C-2$.

We have assumed that $Q$ is non-abelian, and that there exists an $M_1$ such that $d_G(M_1) \geq d_G(M_1/N_1) > n_1/4$. If $G_1$ has excess 1 then  by Proposition~\ref{prop:Gen2Groups} $d(G_1) = n_1/4+1$, and the only normal subgroup $K$ of $G_1$  with $d_{G_1}(K) > n_1/4$ is $G_1$ itself, so $M_1 = G_1$, and assumption becomes $d_{G_1}(Q) = n_1/4 + 1 = d(G_1)$.  Hence Theorem~\ref{thm:comp} applies to $G_1$ in all three cases, and our choices of $e$ ensure that there exists a normal subgroup $F$ of $Q$ 
such that $|[F, Q]| |C_{Q}(F)|^e \le 2^{n_1/2}$ and $|C_{Q}(F)|\le 2^{n_1/4}$. Then taking logs in Theorem~\ref{thm:HomTheorem} gives  
\[\begin{array}{rl}
 \log{|\Epi(A, Q)|} & \le \log \cSevenA + \cSix \log{d(A)}+d(A)\log{|C_Q(F)|}+d_N\log{|[F, Q]|}.\\ 
 \end{array} \]
We now bound $d(A)$ and $d_N$.

Let $d_3 := \sum_{i=2}^r n_i$ and $d_4:=\sum_{i=r+1}^tn_i$. Then Theorem~\ref{thm:Gen2pGroupsF}(ii) implies that $d(L)\le fd_3+d_4/2$ which, since $f \le e/4$ in each case, is less than $ed_3/4 + d_4/2$; furthermore by Theorem~\ref{thm:KovPrae} $d(L) \le (n-n_1)/2$, so $d(A)\le d(L) \le \min\{ed_3/4 + d_4/2, (n-n_1)/2\}$. 
Next, by definition $d_N$ is any upper bound on $\{d_A(H) \text{ : } H \in \NormSub(A, A'\cap \ker(\pi_{r+1, \ldots, t}))\}$. Now Theorem~\ref{thm:Gen2pGroupsF}(iii) shows that every 
$H \in \NormSub(L, L'\cap \ker(\pi_{r+1,\ldots,t}))$ can be generated as an $L$-group by $d_3/4$ elements, so the same holds for $A$ and we  may set $d_N = d_3/4$. It now follows  that
\begin{align*}
 \log{|\Epi(A, Q)|}   &\le \log \cSevenA +  \cSix \log{\frac{n-n_1}{2}}+\left(\frac{ed_3}{4}+\frac{d_4}{2}\right)\log{|C_Q(F)|}+ \frac{d_3}{4}\log{|[F, Q]|}\\
 &\le \log \cSevenA +  \cSix \log{(n-n_1)}+ \frac{d_3}{4}\left(e\log{|C_Q(F)|} + \log{|[F, Q]|}\right)+\frac{d_4}{2}\log{|C_Q(F)|}\\
 &\le \cEight\log{(n-n_1)}+\frac{n_1d_3}{8}+\frac{n_1d_4}{8}\\
 & \le  \cEight\log{(n-n_1)} +\frac{n_1(n-n_1)}{8}.
\end{align*}

Next, define $\mathcal{S}(N_1,\hdots,N_{t}, L)$ to be the set of groups $G$ in $\mathcal{S}(N_1,\hdots,N_t)$ such that $G\pi_{\{2,\hdots,t\}}=L$.
Then $|\mathcal{S}(N_1,\hdots,N_t, L)|$ is at most the number of subdirect products of $Q_1\times (L/(N_2\times\cdots\times N_t))$ whose intersection with $Q_1$ is trivial. By Goursat's Lemma~\ref{lem:Goursat}(\ref{goursat_kernel}), this is precisely $|\Epi(L/(N_2\times\cdots\times N_t), Q_1)|$.  
Hence $\log{|\mathcal{S}(N_1,\hdots,N_t, L)|} \leq \cEight\log{(n-n_1)} + n_1(n-n_1)/8$. 


Next we bound $|\mathcal{L}|$. Let $\mathcal{M} = \{(M_2, \ldots, M_t) \text{ : } N_i \le M_i \unlhd G_i\}$.
By induction,
\begin{align*}
    |\mathcal{L}| & = \sum_{(M_2, \ldots, M_t) \in \mathcal{M}} |\{L \in \mathcal{L} \text{ : } L \cap G_i = M_i\}|
     = \sum_{(M_2, \ldots, M_t) \in \mathcal{M}}|\mathcal{S}(M_2, \ldots, M_t)|\\
    & \le \sum_{(M_2, \ldots, M_t) \in \mathcal{M}} 2^{\frac{(n-n_1)^2}{16}+ \cEight (n-n_1)\log{(n-n_1)}}|\mathcal{L}(G/M_3, \ldots, G/M_t)|\\
    & \le 2^{\frac{(n-n_1)^2}{16}+ \cEight (n-n_1)\log{(n-n_1)}}\sum_{(M_2, \ldots, M_t) \in \mathcal{M}} |\mathcal{L}(G/M_3, \ldots, G/M_t)|\\
    & = 2^{\frac{(n-n_1)^2}{16}+ \cEight (n-n_1)\log{(n-n_1)}}|\mathcal{L}(Q_2, \ldots, Q_t)|.
\end{align*}

Finally we are able to deduce
\begin{align*}
    |\mathcal{S}(N_1,\hdots,N_t)| & \le |\mathcal{L}| \cdot \max\{|\mathcal{S}(N_1, \ldots, N_t, L)| \text{ : } L \in \mathcal{L}\} \\
     & \le 2^{\frac{(n-n_1)^2}{16}+ \cEight (n-n_1)\log{(n-n_1)}+ \cEight{\log(n - n_1)}+n_1(n-n_1)/8}|\mathcal{L}(Q_2, \ldots, Q_t)|\\
   &  \le 2^{(n-n_1)^2/16+n_1(n-n_1)/8+\cEight n\log{n}}|\mathcal{L}(Q_2,\hdots,Q_t)|,
\end{align*}
and so \eqref{lab:preLClaim} follows in Case II. 

\medskip

We now complete the proof of the theorem. Firstly, by Lemma~\ref{lem:HowManyTabs} 
there exists a constant $\cNine = \cNine(C)$ such that 
$|\mathcal{L}(G_1, \ldots, G_t)| \le 2^{\cNine t \log t}$ . Let $\tau = \cNine + \cEight$,  and let $\mathcal{N} = \{(N_1, \ldots, N_t)  \text{ : } N_i \unlhd G_i\}$. Then \eqref{lab:preLClaim} gives
\begin{align*}
    |\Subdir(G_1\times\cdots\times G_t)| & =\sum_{(N_1, \ldots, N_t) \in \mathcal{N}}|\mathcal{S}(N_1,\hdots,N_t)|\\
    & \leq  2^{n^2/16+\cEight n\log{n}}\sum_{(N_1,\hdots,N_t)\in \mathcal{N}}|\mathcal{L}(Q_2,\hdots,Q_t)|\\
    & \leq  2^{n^2/16+\cEight n\log{n}}|\mathcal{L}(G_1,\hdots,G_t)|
     \leq 2^{n^2/16 + \tau n \log n},
\end{align*}
as required. 
%
\end{proof}

Next we record an easy result  that will be useful for all $p \ge 3$.
Recall the constant  $\cSeven = \cSeven(p, C)$ from Theorem~\ref{thm:InitialSubdirect}.
\begin{Lemma}\label{lem:easy_p}
    Let $p \ge 3$ be a prime, and let $C$ and $n$ be positive integers. Then the following holds for all partitions $n = n_1 + \cdots + n_t$ and transitive $p$-groups $G_i \le \Sn_{n_i}$. If $n_i \le p$ for all $i$ then $|\Sub(G_1 \times \cdots \times G_t)| < p^{n^2/4p^2} \cdot 2^{\log (n/p) +2}$.  If instead $C > n_i \ge p^2$ for all $i$, then $|\Subdir(G_1 \times \cdots \times G_t)| \le p^{n^2/4p^2} \cdot 2^{\cSeven n \log n}$. 
\end{Lemma}

\begin{proof}
Let $D = G_1 \times \cdots \times G_t$. 
    If $n_i \le p$  for all $i$ then  each $G_i$ is either trivial or isomorphic to $C_p$, so  $|D| \le p^{n/p}$. Lemma~\ref{lem:AnerCount}(i)  bounds $|\Sub(D)| \le (n/p)4 p^{n^2/4p^2} < p^{n^2/4p^2} \cdot 2^{\log (n/p) + 2}$.

    Assume instead that $C > n_i \ge p^2$ for all $i$. 
 Then for all normal subgroups $N_i$ of $G_i$, Proposition~\ref{prop:Gen2Groups}(iii) bounds $d_{G_i}(N_i)\le 2n_i/p^2$. Now $|G_i|\le p^{(n_i-1)/(p-1)}$ , so by Theorem~\ref{thm:InitialSubdirect} with $c = 1/(p-1)$, $d = 2/p^2$, and $\lambda$ and $k$ absolutely bounded (since $n_i < C$) 
\[
|\Subdir(G_1\times\cdots\times G_t)|\le 2^{\cSeven  t\log{t}}p^{2n^2/(4(p-1)p^2)}\le p^{n^2/4p^2} \cdot 2^{\cSeven n\log{n}},
\]
as required.
\end{proof}

We now consider the case $p = 3$. 

\begin{Proposition}\label{prop:Pyber3}
Let $C$ be any integer greater than three. 
Then there exists an constant $\tau = \tau(3, C)$ such that the following holds for all positive integers $n\ge 3$. 
Let $n = n_1 + \cdots + n_t$ be a partition of $n$ into $3$-power
parts such that $3 \le n_i < C$ for all $i$. Fix transitive $3$-groups $G_i\le \Sn_{n_i}$, for $1\le i\le t$.  Then $|\Subdir(G_1\times\cdots\times G_t)|\le 3^{n^2/36}2^{\tau n\log{n}}$.
\end{Proposition}
\begin{proof}
Fix a lower triangular normal subgroup tableau $T = (G_{ij})$ for $(G_1,\hdots,G_t)$ and  let $N_i =  G_{i1}$ and  $Q_i = G_i/N_i$.   Let $\Subdir(T)$ be the set of subdirect products $G$ of $G_1\times\cdots\times G_t$ with lower coordinate tableau $T$, so that in particular $N_i =  G \cap G_i$ for all $i$. Lemma~\ref{lem:tableaux} bounds $|\Subdir(T)|
\leq |\Subdir(Q_1\times\cdots\times Q_t)|$. Let $\cSeven=\cSeven(3,C)$,  $\cSix=\cSix(3,\lfloor (C-1)/2 \rfloor)$ and $\cSevenA=\cSeven(3,\lfloor (C-1)/2 \rfloor)$ be as in Theorems~\ref{thm:InitialSubdirect} and \ref{thm:HomTheorem}. Set $\cTen = \cSeven + \cSix + \log \cSevenA$.
We will first prove that
\begin{align}\label{eq:SM3}
    |\Subdir(T)|\le 3^{n^2/36} 2^{\cTen n\log{n}},
\end{align}
and then use this to bound $|\Subdir(G_1\times\cdots\times G_t)|$.

\noindent \textbf{Case I:} Suppose first that $Q_i$ is abelian for all $i$. Then $\prod_{i=1}^t|Q_i|\le 3^{n/3}$ by Theorem~\ref{thm:KovPrae}, so \eqref{eq:SM3} follows easily from Lemma~\ref{lem:AnerCount}(i). 

\noindent \textbf{Case II:} Suppose next that some $Q_i$ is nonabelian, but whenever $n_i\le 27$, either the group $Q_i$ is abelian or  $n_i = 27$ and $d(G_i) < 5$. By re-indexing the $G_i$, we may fix an $r < t$ such that
 $i \le r$ if and only if $n_i \le 27$ and $Q_i$ is abelian.

Let $A = Q_1 \times \cdots \times Q_r$ and $d_1 =\sum_{i=1}^rn_i$.
Since $A$ is abelian, it has order at most $3^{d_1/3}$ by Theorem~\ref{thm:KovPrae}, so 
$|\Sub(A)|< 4\frac{d_1}{3} 3^{d_1^2/36}$ by Lemma~\ref{lem:AnerCount}(i).

Now let $D =G_{r+1}\times\cdots\times G_t$ and $d_2 = n-d_1$. 
By Lemma~\ref{lem:easy_p} the number of subdirect products of $D$ is at most $3^{d_2^2/36}2^{\cSeven d_2\log{d_2}}$, and by Theorem~\ref{thm:Gen2pGroupsF}(i) (with $N = G$)  every subdirect product $L$ of $D$ can be generated by $d_2/6$ elements.   It follows that $|\Epi(L,X)|\le 3^{d_1d_2/18}$ for any such $L$ and any section $X$ of $A$. 
Then, by Goursat's Lemma~\ref{lem:Goursat}(\ref{goursat_property}), with $\mathcal{P}$ being  subdirectness, 
\begin{align*}
|\Subdir(T)| & \leq |\Sub_{\mathcal{P}}((Q_1 \times \cdots \times Q_r) \times (Q_{r+1} \times \cdots \times Q_t))| \\
& \le 4\frac{d_1}{3} 3^{d_1^2/36}3^{d_1d_2/18}3^{d_2^2/36}2^{\cSeven d_2\log{d_2}} \le 3^{n^2/36}2^{\cSeven n\log{n}}.
\end{align*}

\noindent \textbf{Case III:} Suppose finally that there exists an $i$ such that $Q_i$ is non-abelian, and either $n_i = 9$  or $n_i = 27$ and $d(G_i) = 5$ (so $G_i$ is {\sc TransitiveGroup(27, 712)}, by Proposition~\ref{prop:Gen2Groups}). We shall induct on the number of such factors, noting that if there are none then the result follows from the previous two cases. 

We may assume $n_1 \in \{9, 27\}$ with $Q_1$ non-abelian, and that $G_1$ has ID 712 if $n_1 = 27$. Let $G\pi_{2 \ldots, t}$ be $(t-r)$-projecting, and re-index the $G_i$ so that $G\pi_{r+1,\hdots,t}=G_{r+1}\times\cdots\times G_t$. Let $V$ be the result of deleting the first row and column from $T$, and let $\Subdir(V)$ be the set of subdirect products of $G_2\times\cdots\times G_t$ with lower coordinate tableau  $V$. Thus, if $G\in\Subdir(T)$, then $G\pi_{2,\hdots,t}\in \Subdir(V)$. Fix $L\in \Subdir(V)$, and let $\mathcal{S}(T,L):=\{G\in \Subdir(T)\text{ : }G\pi_{2,\hdots,t}=L\}$. Let $d_1:=n_1$, let $d_2:=\sum_{2\le i\le r}n_i$, and let $d_3:=\sum_{i>r}n_i$. 

We shall now use Theorem~\ref{thm:HomTheorem} to  count $|\Epi(L, Q_1)|$. One can check the transitive $3$-groups of degree $9$  and {\sc TransitiveGroup(27, 712)} to see that each possible $Q_1$ has a normal subgroup $F$ such that $|C_{Q_1}(F)| \le  3^{n_1/9}$ and $|[Q_1, F]| \le 3^{2n_1/9}$.
Let $R$ be the projection of $L$ to $G_2 \times \cdots \cdots \times G_r$ and $P$ be the projection to $G_{r+1} \times \cdots \times G_t$.  
If $H \in \NormSub(L, L'\cap \ker(\pi_{r+1,\ldots,t}))$ then $H$ projects to a proper subgroup of $G_i$ for each $i \ge 2$, so we may set
$d_N = d_2/6$, and similarly
$d(L)\le \min\{d_2/6+d_3/3, (n-d_1)/3\}$.  Finally, each factor has bounded degree, and hence bounded order. 
 Therefore by Theorem~\ref{thm:HomTheorem}
 \begin{align*}
     |\Epi(L,Q_1)|& \leq \cSevenA d(L)^{\cSix}|C_{Q_1}(F)|^{d(L)} |[Q_1, F]|^{d_N}\\
     & \le \cSevenA ((n-d_1)/3)^{\cSix}3^{(d_1/9)(d_2/6+d_3/3)+(2d_1/9)(d_2/6)}\\
     & \le n^{\cSix + \log \cSevenA}3^{d_1d_2/18+n_1d_3/27}.
 \end{align*}

By induction, the number of groups $L$ is at most $|\Subdir(V)|\le 3^{(n-d_1)^2/36}2^{\cTen}(n-d_1)\log{(n-d_1)}$. Thus, by Goursat's Lemma~\ref{lem:Goursat}\ref{goursat_property} (with $\cal{P}$ trivial)
\[|\Subdir(T)|\le 3^{(n-d_1)^2/36}3^{d_1d_2/18+d_1d_3/27}2^{ \cTen}(n-d_1)\log{(n-d_1)}n^{\cSix + \log \cSevenA}\le 3^{n^2/36}2^{ \cTen n\log{n}},\] as required. 

\smallskip

To complete the proof,  Lemma~\ref{lem:HowManyTabs} bounds $|\mathcal{L}(G_1, \ldots, G_t)| \le 2^{\cNine t \log t}$. 
The result follows by setting $\tau =  \cTen + \cNine$.  
\end{proof}

The cases $p\in\{5,7\}$ are considerably easier. Let $\cSeven = \cSeven(p, C)$ be as in Theorem~\ref{thm:InitialSubdirect}. 

\begin{Proposition}\label{prop:Pyber57}
Fix $p\in\{5,7\}$, and let $C$ be any fixed positive integer greater than $p$. Let $n = n_1 + \cdots + n_t$ be a partition of $n$ into $p$-power parts of size  greater than $1$ and less than $C$, and fix transitive $p$-groups $G_i\le \Sn_{n_i}$ for $1\le i\le t$.
Then $|\Subdir(G_1\times\cdots\times G_t)|\le p^{n^2/4p^2} \cdot 2^{\cSeven n\log{n}}$.
\end{Proposition}
\begin{proof}
If $n_i = p$ for all $i$, or if
$n_i \ge p^2$ for all $i$, then the result follows from Lemma~\ref{lem:easy_p}, so
index the $n_i$ so that  $n_i=  p$ if and only if $i \le r$. Let $D_1:=G_1\times\cdots \times G_r$, $D_2:=G_{r+1}\times\cdots\times G_t$, $d_1=\sum_{i=1}^rn_i$ and $d_2=n-d_1$. 

Since each factor of $D_2$ has degree at least $p^2$, by Lemma~\ref{lem:easy_p} the number of subdirect products of $D_2$ is at most $p^{d_2^2/4p^2} 2^{\cSeven d_2 \log d_2}$; whilst $D_1 \cong C_p^{d_1/p}$  so  $|\Sub(D_1)| < p^{d_1^2/4p^2}2^{\log (d_1/p) + 2}$. By Theorem~\ref{thm:Gen2pGroupsF}(i) we  can bound $d(L)\le 2d_2/p^2$ for all subdirect products $L$ of $D_2$. Since $|D_1|= p^{d_1/p}$, there are are most $|D_1|^{2d_2/p^2} = p^{2d_1d_2/p^3}$ homomorphisms from such an $L$ to a subgroup of $D_1$, so by Goursat's Lemma~\ref{lem:Goursat}(\ref{goursat_property}) with $\mathcal{P}$ being subdirectness, 
\begin{align*}
|\Subdir(D_1 \times D_2)| & \le |\Sub(D_1)| |\Subdir(D_2)| \mathrm{max}\{|\Hom(L, H)| text{ : } L \in \Subdir(D_2), H \in \Sub(D_1)\}|\\
& \le p^{d_1^2/4p^2+d_2^2/4p^2+2d_1d_2/p^3} 2^{\log (d_1/p) + 2 + \cSeven d_2 \log d_2}\\
& \le p^{n^2/4p^2} 2^{\cSeven n \log n}
\end{align*}
 as required.  
\end{proof}

Finally, we quickly count the $p$-subgroups of $\Sn_n$ for each prime $p\geq 11$.

\begin{Proposition}\label{prop:Pyber11}
Fix a prime $p\geq 11$. Then $|\Sub_{p}(\Sn_n)|\le p^{n^2/4p^2}2^{\frac{3}{2} n \log n}$.
\end{Proposition}
\begin{proof}
This is a straightforward application of Theorem~\ref{thm:BoundedOrbits}.
Let $\mathcal{P}$ be the property of being a $p$-group and let $\epsilon =(\log{p})/4p^2$. A Sylow $p$-subgroup of $\Sn_n$ has order at most $p^{(n-1)/(p-1)}$, so we may let $\delta=(\log{p})/(p-1)$ and $\gamma=0$.
We now let $C=p^2$, and note that all $\mathcal{P}$-subgroups $G\le \Sn_n$ with all orbits of length at least $C$ satisfy $d(G)\le 2n/p^2$ by Theorem~\ref{thm:Gen2pGroupsF}(i). This is less than $\epsilon n/\delta=n(p-1)/4p^2$ since $p\geq 11$. 

For all $n\in\mathbb{N}$, all partitions $n=n_1+\cdots+n_t$ with $n_i<C$  and all transitive $p$-subgroups $H_i\le \Sn_{n_i}$, we deduce that $n_i \le  p$ so Lemma~\ref{lem:easy_p} bounds 
$\Sub(H_1\times\cdots\times H_t)|\le 2^{\epsilon n^2 + \log (n/p) +2}$.
We may  therefore let $f:\mathbb{N}\rightarrow\mathbb{R}_+$ be  $f(n)=\log (n/p) + 2$.
The result is now immediate from Theorem~\ref{thm:BoundedOrbits}, since $4 \sqrt{n} + \log(n/p) + 2 \le \frac{1}{2} n \log n$ for $n \ge p \ge 11$. 
\end{proof}

We are now ready to prove Theorems~\ref{thm:Pyberp} and \ref{thm:Pybernilp}.

\begin{proof}[Proof of Theorem~\ref{thm:Pyberp}]
The lower bound on $\Sub_p(\Sn_n)$ is Proposition~\ref{prop:H}. For the upper bound, if $p\geq 11$ then $n \ge 11$ and the result is immediate from
Proposition~\ref{prop:Pyber11}, 
so  assume that $p\le 7$.  Let $\epsilon = (\log{p})/4p^2$, let 
$\delta =(\log{p})/(p-1)$, and let $C = 2^{(\zeta\delta/\epsilon)^2}$, where $\zeta$ is the absolute constant from Proposition~\ref{prop:NormalGens}. 
Finally, let  $n = n_1 + \cdots + n_t$ be a partition into $p$-power parts of size less than $C$, choose a Sylow $p$-subgroup $H_i \le \Sn_{n_i}$, and let $D = H_1 \times \cdots \times H_t$.

Since $n_i < C$, there exists a constant $\cTwenty$ such that $|\Sub(H_i)| \le 2^{\cTwenty}$ for all $i$. For any choice of $G_i \le H_i$, each orbit of $G_i$ has length less than $C$. Furthermore, if $G_i$ has any trivial orbits, then these can be deleted without changing the number of subdirect products of $G_1 \times \cdots \times G_t$. Hence by Propositions~\ref{prop:Pyber2}, \ref{prop:Pyber3} and \ref{prop:Pyber57}, 
there exists a constant $\tau = \tau(p, C)$ such that
\begin{equation}\label{eq:final_p_count}
\begin{array}{rl}
|\Sub(D)| & \le \prod_{i=1}^t|\Sub(H_i)| \cdot  \max\{|\Subdir(G_1\times\cdots\times G_t)|\text{ : }G_i\le H_i\} \\
& \le 2^{\cTwenty t} p^{n^2/4p^2} 2^{\tau n \log n} \le 2^{((\log p)/4p^2)n^2} 2^{(\cTwenty + \tau)n \log n}.
\end{array}\end{equation}

We now apply Theorem~\ref{thm:BoundedOrbits}. Let $\mathcal{P}$ be the property of being a $p$-group, let $\mathcal{C}_m$ be $\mathrm{Syl}_p(\Sn_m)$ if $m$ is a power of $p$, and be empty otherwise. Then the groups in $\mathcal{C}_m$ are all conjugate and have order less than $2^{(n \log p)/(p-1)}$, so with our current values of $\epsilon$ and $\delta$, we may in addition let 
$\gamma =0$. Finally, for all $n \in \mathbb{N}$, and all $p$-subgroups $G$ of $\Sn_n$, if all orbits of $G$ have length at least $C$ then $d(G) \le \epsilon n/\delta$. Finally, by \eqref{eq:final_p_count} we may set $f(n) = 2^{(\cTwenty + \tau)n \log n}$. 
Hence by Theorem~\ref{thm:BoundedOrbits}
\[    \Sub_p(\Sn_n) \leq 2^{\epsilon n^2 + f(n) + (5 + \gamma)n \log n} 
  \leq p^{n^2/4p^2} 2^{(\cTwenty + \tau  + 5)n \log n},
\]
as required. 
\end{proof}

\begin{proof}[Proof of Theorem~\ref{thm:Pybernilp}]
This is similar to the previous proof. First, let  $\epsilon = (\log p)/4p^2$, let $\delta =1$, and let $C =2^{(\zeta\delta/\epsilon)^2}$, where $\zeta$ is as in Proposition~\ref{prop:NormalGens}.
Let $n = n_1 + \cdots + n_t$ be a partition into parts of size less than $C$, such that each $n_i$ is divisible by no primes less than $p$.  Let $p_1<\cdots<p_s$ be the distinct prime divisors of the $n_i$, for $j \in \{1, \ldots, s\}$ let $m_j$ be the sum of the non-trivial $p_j$-power parts of the $n_i$, and fix $P_j \in \mathrm{Syl}_{p_j}(\Sn_{m_j})$.   By Theorem~\ref{thm:Pyberp}, $|\Sub(P_j)| \leq p_j^{m_j^2/4p_j^2}2^{\beta_{p_j} m_j \log m_j}$. Since the function $(\log r)/4r^2$ is decreasing as $r$ increases, letting $\kappa_p  = \max\{3/2, \beta_q \text{ : } q \geq p\}$ we see that
$|\Sub(P_j)| \leq p^{m_j^2/4p^2}2^{\kappa_p m_j \log m_j}$. By Lemma~\ref{lem:nilp_degree}(iii), if $H_i$ is a maximal transitive nilpotent subgroup of $\Sn_{n_i}$ for each $i$ then $H_1 \times \cdots \times H_t$ is isomorphic to a subgroup of $D:= P_1 \times \cdots \times P_s$.
Let $m = m_1 + \cdots + m_s$. Then  
\begin{equation}\label{eq:pybernilp}
\begin{array}{rl}
|\Sub(H_1 \times \cdots \times H_t)| = |\Sub(D)| & \leq |\Sub(P_1)|\cdots |\Sub(P_s)| 
\leq p^{m^2/4p^2}2^{\beta_p m\log{m}} \\ &  \leq 2^{((\log p)/4p^2)n^2 + \beta_p n \log n}.\end{array}
\end{equation}

We now apply Theorem~\ref{thm:BoundedOrbits}. Let $\mathcal{P} = \mathcal{N}_p$.
For each $n$, by Lemma~\ref{lem:nilp_degree}(iii) there is a unique class of maximal transitive nilpotent subgroups of $\Sn_n$: we let $\mathcal{C}_n$ be this class if $n$ has no prime divisor less than $p$, and be empty otherwise. Then the groups in $\mathcal{C}_n$ are conjugate and have order less than $2^n$, so with our current values of $\epsilon$ and $\delta$ we let $\gamma=0$. 
Every $\mathcal{N}_p$-subgroup of $\Sn_n$ with all orbits of length at least $C$ satisfies $d(G) \leq \epsilon n/\delta$. Finally, by \eqref{eq:pybernilp} we may set $f(n) = \kappa_p n \log n$. Then by Theorem~\ref{thm:BoundedOrbits}, 
\[ \Sub_{\mathcal{N}_p}(\Sn_n) \leq 2^{\epsilon n^2 + f(n) + (5 + \gamma) n \log n} \leq p^{n^2/4p^2} 2^{(\kappa_p+5)n \log n}.\] The final claim follows from noting that $\kappa_p + 5$ is non-increasing with $p$. 
\end{proof}

\section{The proofs of Theorem~\ref{thm:NilpotentRedTheorem} and Theorem~\ref{thm:PyberConj}}\label{sec:NilpotentRed}

We start with three technical lemmas. Recall from the introduction 
the definition of the socle length $\dl{G}$. 

\begin{Lemma}\label{lem:slemma}
Let $G$ be a finite soluble group, and let $N$ be a normal subgroup of $G$. Then
\begin{enumerate}[\upshape(i)]
    \item $\dl(G/N)\le \dl(G)$; and 
    \item $\dl(N)\le \dl(G)$.
\end{enumerate}
\end{Lemma}
\begin{proof}
We will prove  both parts together by induction on $|G|$. Clearly the result holds when  $|G| \le 3$, so assume inductively that  
both claims hold for groups of order less than $|G|$. We may also assume that $|N|>1$. 
Then $\Soc(G)N/N$ is contained in $E/N := \Soc(G/N)$, so writing bars to denote reduction modulo $\Soc(G)N$ and applying induction gives $\dl(\ol{G}/\ol{E})\le \dl(\ol{G})$. Since $\ol{G}/\ol{E}\cong G/E\cong (G/N)/(E/N) = (G/N)/\Soc(G/N)$, it follows that $\dl((G/N)/\Soc(G/N))\leq \dl(G/\Soc(G)N)$. 
Now
$$\dl(G/N)=\dl((G/N)/\Soc(G/N))+1\le \dl(G/\Soc(G)N)+1\le \dl(G/\Soc(G))+1=\dl(G),$$
where the last inequality follows from the inductive hypothesis for (i) applied to $G/\Soc(G)$. This proves (i).

For (ii), we may assume that $N<G$. Let $J$ be a maximal normal subgroup of $G$ containing $N$. Then $\dl(N)\le \dl(J)$ by the inductive hypothesis for (ii) applied to $J$. Thus, we may assume that $N$ is a maximal normal subgroup of $G$.  Suppose first that $N\Soc(G) = G$. Then $\Soc(G)$ is not contained in $N$, so there exists a minimal normal subgroup $M$ of $G$ which is not contained in $N$. Then $N\cap M=1$, and the maximality of $N$ implies that 
$G=N\times M$.
Hence $\Soc(G)=\Soc(N)\times M$, so $G/\Soc(G)\cong N/\Soc(N)$. Thus 
$$\dl(N)=\dl(N/\Soc(N))+1=\dl(G/\Soc(G))+1=\dl(G),$$
and the result follows in this case.

Hence we may assume that $N\Soc(G)$ is proper in $G$, and hence by maximality $N\Soc(G) = N$. 
By Clifford theory and solubility, $N$ acts completely reducibly on each minimal normal subgroup of $G$, and so $\Soc(G) \le \Soc(N\Soc(G)) = \Soc(N)$. By applying (i) to $N/\Soc(G)$ we see that $\dl(N/\Soc(N))\le \dl(N/\Soc(G))$. 
Thus inductively applying (ii) to $G/\Soc(G)$ shows that
 $$\dl(N)=1+\dl(N/\Soc(N))\le 1+\dl(N/\Soc(G))\le 1+\dl(G/\Soc(G))=\dl(G),$$
as required.
\end{proof}

For a $G$-module $V$ and $v_1, \ldots, v_k \in V$, we write $\langle v_1, \ldots, v_k\rangle^G$  for the $G$-submodule of $G$ generated by $v_1, \ldots, v_k$.
\begin{Lemma}\label{lem:GenMod}
  Let $G$ be a finite group, and let $V$ be a completely reducible $G$-module over an arbitrary field $\mathbb{F}$. Write $V=W_1^{\delta_1}\oplus\cdots\oplus W_e^{\delta_e}$, where the $W_i$ are pairwise non-isomorphic irreducible $\mathbb{F}[G]$-modules. Then $d_G(V)\le \delta:= \max\{\delta_i\text{ : }1\le i\le e\}$.
\end{Lemma}

\begin{proof}
For $i \in \{1, \ldots, e\}$ and $j \in \{1, \ldots, \delta_i\}$, write $W_{i, j}$ for the $j$th copy of $W_i$. By irreducibility there exists 
$w_{i,j} \in W_{i,j}$ such that $\langle w_{i,j} \rangle^G = W_{i,j}$, so that 
$\langle w_{i, 1}, \ldots, w_{i,\delta_j}\rangle^G = W_i^{\delta_j}$. 
Let $v_1 = w_{1, 1} + w_{2, 1} + \cdots + w_{e, 1}$, for $j \in \{2, \ldots, \delta\}$ let $v_j$ be the sum of those $w_{i, j}$ that are defined, and let $U = \langle v_1, \ldots, v_\delta\rangle^G$. Then $U$ projects surjectively onto each $W_i^{\delta_i}$,
so $W_i^{\delta_i}/(U\cap W_i^{\delta_i})$ and $W_j^{\delta_j}/(U\cap W_j^{\delta_j})$ are $G$-isomorphic modules, and hence trivial. 
Therefore $U = V$. 
\end{proof}

  We now present some information about the structure of finite soluble groups with trivial Frattini: this result is standard, but we have been unable to find a reference, so sketch a proof. We write $A \cn B$ for the semidirect product of groups $A$ and $B$.

\begin{Lemma}\label{lem:SolTrivFrattini}
  Let $G$ be a finite soluble group with trivial Frattini subgroup. Then there exist an integer $e \ge 1$,  elementary abelian groups $W_i$ and irreducible groups $L_i  \le \GL(W_i)$ for $i \in \{1, \ldots, e\}$ such that  $G$ embeds as a subdirect product of
  $H_1 \times \cdots \times H_e$, where
  \[H_i:=\{(x_1,\hdots,x_{\delta_i})\in (W_i \cn L_i)^{\delta_i}\text{ : }x_k\equiv x_{\ell}\text{ (mod }W_i\text{) for each }k,\ell\} \le (W_i \cn L_i)^{\delta_i}.\]
  Furthermore, $\Soc(G) \cong W_1^{\delta_1} \times \cdots \times W_e^{\delta_e}$.
  \end{Lemma}

  \begin{proof}
 Let $V$ be the Fitting 
 subgroup of $G$. The assumptions that $G$ is soluble and $\Phi(G) = 1$ together imply that $V = \Soc(G)$, and that $V$ is complemented and self-centralising. 
 Therefore $G = V \cn L$, for some subgroup $L$ which acts faithfully on $V$ by conjugation.
 
 The group $V$ is a direct product of elementary abelian groups, each of which is an irreducible $G$-module. Let $W_1, \ldots, W_e$ be representatives of the isomorphism types of irreducible $G$-modules in $V$, so that
 $V = W_1^{\delta_1} \times \cdots \times W_e^{\delta_e}$.
 
 For $i \in \{1, \ldots, e\}$, identify $G/C_G(W_i) \cong L/C_L(W_i)  = L/C_L(W_i^{\delta_i})$ 
 with an irreducible subgroup $L_i$ of $\GL(W_i)$. Since $V$ is equal to $C_G(V)$, the intersection $\bigcap_{i = 1}^e C_L(W_i)$ is trivial, so
 $L$ embeds as a subdirect product of $(L/C_L(W_1)) \times \cdots \times (L/C_L(W_e)) \cong L_1 \times \cdots \times L_e$, 
 and the result follows. 
 \end{proof}

We are now ready to generalise the following famous result due to Aschbacher and Guralnick. 

\begin{Theorem}[\cite{AschGur}]\label{thm:AschGur}
Let $H$ be a finite group. Then $H$ can be generated by a soluble subgroup together with one other element.
\end{Theorem}

\begin{proof}[Proof of Theorem~\ref{thm:NilpotentRedTheorem}]
The claim about arbitrary groups $H$ follows from the claim for soluble groups and Theorem~\ref{thm:AschGur}, so it suffices to prove the result for soluble groups. 
Let $G$ be a  counterexample of minimum  order.

If the Frattini subgroup $\Phi(G)\neq 1$, then by induction $G/\Phi(G)$ can be generated by a nilpotent subgroup $K/\Phi(G)$, together with a set $\{\Phi(G)x \text{ : } x \in X \} \subset G/\Phi(G)$ of size  at most $4\dl(G/\Phi(G))\sqrt{\ma(G/\Phi(G))}$.  
Choose $K_0\le K$ minimal subject to $\Phi(G)K_0=K$, and let $M$ be a maximal subgroup of $K_0$. Then $(K_0\cap \Phi(G))M$ is equal to $M$ or $K_0$. If $(K_0 \cap \Phi(G))M = K_0$ then $K = \Phi(G)(K_0 \cap \Phi(G))M = \Phi(G)M$, contradicting the minimality of $K_0$. Thus
$K_0 \cap \Phi(G) \leq M$, and since  this holds for every such $M$, we deduce that $K_0\cap\Phi(G)\le \Phi(K_0)$.
A finite group $N$ is nilpotent if and only if $N/\Phi(N)$ is nilpotent. Since $K_0/(K_0\cap\Phi(G))\cong K/\Phi(G)$ is nilpotent, and $K_0/\Phi(K_0)$ is a quotient of $K_0/(K_0\cap \Phi(G))$, we deduce that $K_0$ is nilpotent. 
We proved that $G = \langle K_0, \Phi(G), X\rangle = \langle K_0, X\rangle$. Since $|X| \le \dl(G/\Phi(G))\sqrt{\ma(G/\Phi(G))}\le \dl(G)\sqrt{\ma(G)}$ by Lemma~\ref{lem:slemma}(i), the result follows in this case. 

Hence we may assume from now that $\Phi(G)=1$, so that Lemma~\ref{lem:SolTrivFrattini} applies to $G$: we shall use the notation from that lemma. Let $V= \Soc(G)$, 
$a =\ma(G)$,  write $p_i^{m_i}=|W_i|$ for some prime $p_i$, and let $m:=\sum_{i=1}^em_i$.
Re-index the $W_i$ so that there exists an $r \in \{0, \ldots, e\}$ such that  $\delta_i < \sqrt{a}$ for $i\le r$, and $\delta_i \ge \sqrt{a}$ for $i>r$.
Let $A$ and $B$ be the projections of $G$ onto $H_1 \times \cdots \times H_r$ and $H_{r+1} \times \cdots \times H_e$, respectively, so that $G$ is a subdirect product of $A\times B$ (and $A$ or $B$ is trivial when $r = 0$ or $e$, respectively). Let $N_1$ and $N_2$ be the intersections of $G$ with $H_1 \times \cdots \times H_r$ and $H_{r+1} \times \cdots \times H_e$, respectively, so that $G = N_1$ if $r = e$ and $G = N_2$ if $r = 0$. Then $N_1$ contains $V_1:=W_1^{\delta_1}\times \cdots\times W_r^{\delta_r}$ and $N_2$ contains $V_2:=W_{r+1}^{\delta_{r+1}}\times \cdots\times W_e^{\delta_e}$. Let $K_1 = N_1 \cap (L_1 \times \cdots \times L_r)$, so that $N_1 = V_1 \cn  K_1$, and similarly write $N_2 = V_2 \cn K_2$. 
We note that the $W_j$ are completely reducible as $N_i$-modules, and that pairwise isomorphic $G$-modules are pairwise isomorphic as $N_i$-modules.  

Let $d:=\max\{\delta_i\text{ : }1\le i\le r\} < \sqrt{a}$. Then by Lemma~\ref{lem:GenMod} there exist $v_1,\hdots,v_d \in V_1$ which generate $V_1$ as a $G$-module.
 Now $N_1\unlhd A$ and $A \cong G/N_2$, so  $$K_1\unlhd A\cap (L_{1}\times\cdots\times L_r)\cong  A/V_1\cong  (G/N_2)/(V N_2/N_2) \cong G/(V N_2).$$ 
Since $V=\Soc(G)$, it follows from Lemma~\ref{lem:slemma} that $\dl(K_1)\le \dl(G/V)  = \dl(G)-1$. Let $\ell:=4\sqrt{a}(\dl(G)-1)$. Then the minimality of $G$ as a counterexample shows that there exist $x_1, \ldots, x_\ell \in K_1$ and a nilpotent subgroup $J$ of $K_1$ such that $K_1 = \langle J, x_1, \ldots, x_{\ell} \rangle$.
If $r = e$ then $V_1 = V$, so $G$ can be generated by $J$ and $\ell + d 
\le 4 \sqrt{a}\dl(G) - 3\sqrt{a}$ elements, and the result follows. Thus we may assume from now on that $r < e$.

The group  $N_2/V_2$ is isomorphic to a subdirect product of $L_{r+1}\times \cdots\times L_e$, and each $L_i$ is an irreducible subgroup of $\GL_{m_i}(p_i)$. Therefore each $L_i$-normal subgroup $R_i$ of $L_i$ is completely reducible, so $d_{L_i}(R_i) \leq \frac{3}{2} m_i$ by a theorem of Kov\'{a}cs and Robinson \cite{KovRob}. Thus by Lemma~\ref{lem:generate_subdirect}, letting $m' = \sum_{i=r+1}^em_i$ gives
\begin{equation}\label{eq:r=e}
d(N_2/V_2)\le \frac{3}{2}m'. 
\end{equation}
Let $\delta:=\min\{\delta_i\text{ : }{r+1}\le i\le e\}\geq  \sqrt{a}$. Then since $\ma(G)=a$ and $V_2$ is an abelian section of $G$ with $\Omega(|V_2|)\geq m' \delta\geq m'\sqrt{a}$, we see $m' \leq \sqrt{a}$. Hence, letting $k = \frac{3}{2} \sqrt{a}$,
there exist $g_1, \ldots, g_{k}  \in N_2$ such that 
    $N_2 = \langle V_2, g_1, \ldots, g_{k} \rangle$.
 
In particular, if $r = 0$ then $N_2 = G_2$, and the result follows, so we may now assume that $0 < r < e$. 

Now  $G/(N_1\times N_2) \cong B/ N_2$ by Lemma~\ref{lem:subdir}. Thus 
we may again argue as in \eqref{eq:r=e} to see that $d(G/(N_1\times N_2)) = d(B/N_2) \le d(B/V_2) \le \frac{3}{2} m' \leq k$. 
Let $z_1,\hdots,z_k \in G$ generate $G$ together with $N_1 \times N_2$.
Finally, the group $F:= \langle V_1,J,x_1,\hdots,x_{\ell},z_1,\hdots,z_k\rangle=V_1\langle J,x_1,\hdots,x_{\ell},z_1,\hdots,z_k\rangle$ projects onto $A$ by construction, so 
the action of $F$ on $V_1$ is the same as that of $G$. 
Therefore
$$G=\langle  V_2 \times J, x_1,\hdots,x_{\ell},y_1,\hdots,y_{k},z_1,\hdots,z_k,v_1,\hdots,v_d\rangle.$$ Since $\ell  =  4\sqrt{a}(\dl(G)-1)$, $2k=3\sqrt{a}$ and $d\le\sqrt{a}$, we conclude that $$\ell+2k+d\le 4\sqrt{a}+4\sqrt{a}(\dl(G)-1)=4\sqrt{a}\dl(G).$$
Since $V_2 \times J$ is nilpotent, this completes the proof. 
\end{proof}

We shall apply Theorem~\ref{thm:NilpotentRedTheorem} to subdirect products. 

\begin{Lemma}\label{lem:usefuldl}
Let $G$ be a subdirect product of a direct product $G_1\times\cdots\times G_t$ of finite soluble groups. Then $\dl(G)\le \max\{\dl(G_i)\text{ : }1\le i\le t\}$.
\end{Lemma}
\begin{proof}
We shall induct on $|G_1\times\cdots\times G_t|$: if $|G_1 \times \cdots \times G_t| = 2$ then the result is trivial, so assume the result holds for all smaller direct products. 
 For $1\le i\le t$, let $S_i:=\Soc(G_i)$ and $K_i=\ker(\pi_i)$, and let $S=G\cap (S_1\times\cdots\times S_t)$. We start by showing that $\Soc(G) = S$. 
  
  Let $M$ be a minimal normal subgroup of $G$. Then $M\cap K_i$ is equal to $1$ or $M$, so $M\pi_i=MK_i/K_i$ is either trivial or isomorphic to $M$.  
  Suppose that $M\pi_i\cong M$, so that $M\cap K_i=1$. Let $N$ be a normal subgroup  $N$  of $G$ such that $K_i \le N$ and  $N\pi_i  \cong N/K_i \unlhd  MK_i/K_i  \cong M\pi_i$. Then $N\le MK_i$, so $N=N\cap MK_i=(N\cap M)K_i$. Since $N\cap M$ is $1$ or $M$ by minimality of $M$, we see that $N/K_i$ is trivial or $MK_i/K_i$. Thus $M\pi_i$ is a minimal normal subgroup of $G_i$, and so $M\pi_i\le S_i$. If instead $M \pi_i = 1$ then trivially $M\pi_i \leq S_i$. Thus each minimal normal subgroup of $G$ is contained in $S$ so  $\Soc(G)\le S$.

Conversely,  let $p_1,\hdots,p_k$ be the distinct prime divisors of $|S_1\times\cdots\times S_t|$. By solubility,  each $S_i$ is abelian, so  $S=S(p_1)\times\cdots\times S(p_k)$, where $S(p_i)$ is the intersection of $G$ with the $p_i$-part of $S_1\times\cdots\times S_t$. 
Let $p = p_j$ for some $j$, then the $p$-part of $S_1\times \cdots\times S_t$ is the direct product $(S_1)_p\times\cdots\times (S_t)_p$, where $(S_i)_p$ is the $p$-part of $S_i$. If $(S_i)_p \neq 1$ then $(S_i)_p$ is a direct product of minimal normal subgroups of $G_i$. It follows that $(S_1)_p\times\cdots\times (S_t)_p$ is a completely reducible $\mathbb{F}_p[G]$-module, 
and hence $S(p)$ is also completely reducible as an $\mathbb{F}_p[G]$-module, and so is
a 
product of minimal normal subgroups of $G$. Hence, $S(p)\le \Soc(G)$ for all $p$. We conclude that
 $S=\Soc(G)$.

If $G_i = S_i$ for all $i$, then $\Soc(G) = S = G$ and the result follows. Otherwise,  
the natural homomorphism from $G/\Soc(G)$ to $(G_1/S_1) \times \cdots \times (G_t/S_t)$ 
embeds $G/\Soc(G)$ as
a subdirect product of $(G_1/S_1)\times\cdots\times (G_t/S_t)$. By induction, $\dl(G/S)\le \max\{\dl(G_i/S_i) \text{ : } 1 \le i \le t\}$. 
Hence, $\dl(G)=\dl(G/S)+1\le \max\{\dl(G_i/S_i)\text{ : }1\le i\le t\} + 1
=\max\{\dl(G_i)\text{ : }1\le i\le t\}$. This completes the proof.
\end{proof}

We now combine Theorem~\ref{thm:NilpotentRedTheorem} with Theorem~\ref{thm:BoundedOrbits} to get a technical result from which both Theorem~\ref{thm:PyberConj} and Theorem~\ref{thm:randSyl2} follow easily. 
We shall make essential use of the following result: part (i) is due to Pyber \cite[Lemma 3.2]{PybAnnals}, and Part (ii) is by Dixon \cite{Dixon}. 
\begin{Theorem}\label{thm:Dixon}
\begin{enumerate}[\upshape(i)]
    \item There are at most $2^{3 \log^3 n + 2 \log n} < 2^{13 n}$ classes of maximal transitive soluble subgroups of $\Sn_n.$ 
\item Let $G$ be a soluble subgroup of $\Sn_n$. Then $|G| \le 24^{n/3}$.
\end{enumerate}
\end{Theorem}

Recall the constant $\zeta$ from Proposition~\ref{prop:NormalGens}. The function $g(n)$ in the following statement must exist, by Pyber's \cite{PybAnnals} bound on the number of subgroups of $\Sn_n$. 

\begin{Proposition}\label{prop:count_subgroups}
 Let $\epsilon$ be a positive real number, and let $\mathcal{Q}$ be a subgroup- and quotient-closed property. Then there exists a  constant $\cTwelve = \cTwelve(\epsilon)$ such that the following holds. 
Let $C = \lceil 2^{(\frac{\zeta \log 24}{3 \epsilon})^2}\rceil$, and let $g: \mathbb{N} \rightarrow \mathbb{R}$ be a non-decreasing function such that for all partitions $n = n_1  + \cdots + n_t$ into parts of size less than $C$,  and all transitive subgroups $G_i\le \Sn_{n_i}$, we can bound $|\Sub_{\mathrm{nilp}, \mathcal{Q}}(G_1 \times \cdots \times G_t)| \le 2^{\epsilon n^2 + g(n)}$. Then
$$|\Sub_{\mathcal{Q}}(\Sn_n)| \leq 2^{\epsilon n^2 + g(n) + \cTwelve n^{3/2}}.$$
\end{Proposition}

\begin{proof}
    We shall show that the result holds with $\cTwelve = 19 + 4C (\log 24)^2 \sqrt{\log 3}/(9 \sqrt{3})$. 
    Let $n_1+\cdots+n_t=n$ be such a  partition of $n$, and  for $i \in \{1, \ldots, t\}$ let $H_i$ be a soluble transitive subgroup of $\Sn_{n_i}$. 
       We start by counting the $\mathcal{Q}$-subgroups $G$ of $D:= H_1 \times \cdots \times H_t$.
       
   An abelian section of $\Sn_{n}$ has order at most $3^{n/3}$ by Theorem~\ref{thm:KovPrae}, so $\ma(G) \leq \frac{n}{3}
\log 3$.     By Theorem~\ref{thm:Dixon}(ii), each $H_i$ has order less than $24^{C/3}$, so 
$\dl(G\pi_i) <  \frac{C}{3}\log{24}$. Since $G$ is  subdirect in $G\pi_1 \times \cdots \times G\pi_t$,  Lemma~\ref{lem:usefuldl} yields $\dl(G)\le \max \{\dl(G\pi_i) \text{ : } 1 \le i \le t\} \le 
\frac{C}{3}\log{24}$. 

Theorem~\ref{thm:NilpotentRedTheorem} proves that each $\mathcal{Q}$-subgroup $G$ of $D$ can be generated by a nilpotent subgroup, together with at most $4 \dl(G) \sqrt{\ma(G)} \le 4\frac{C}{3} \log 24 \sqrt{\frac{n}{3} \log 3} = \frac{3(\cTwelve - 19)}{\log 24}n^{1/2}$ other elements. 
Thus, again applying Theorem~\ref{thm:Dixon}(ii) to bound $|D|$ we see that
\begin{equation}\label{eq:d_count1}
    \begin{array}{rl}
    |\Sub_{\mathcal{Q}}(D)| & \le |D|^{\frac{3(\cTwelve - 19)}{\log 24}n^{1/2}}|\Sub_{\mathrm{nilp}, \mathcal{Q}}(D)|
    \le 24^{\frac{(\cTwelve - 19)}{\log 24}n^{3/2}}2^{\epsilon n^2 + g(n)}\\
& \le 2^{\epsilon n^2 + g(n) + (\cTwelve - 19)n^{3/2}}.
\end{array} 
\end{equation}

We now use Theorem~\ref{thm:BoundedOrbits} to determine the number of soluble $\mathcal{Q}$-subgroups of $\Sn_n$, with our current $\epsilon$ and $C$. We let $\mathcal{C}_n$ be the maximal soluble transitive subgroups of $\Sn_n$, so that by Theorem~\ref{thm:Dixon}(i) we may take $\gamma = 13$ and $\delta = \frac{1}{3} \log 24$. Then by Proposition~\ref{prop:NormalGens} each soluble subgroup $G$ of $\Sn_n$ with all orbit lengths at least $C$ satisfies $d(G) \le \epsilon n /\delta$. 
By \eqref{eq:d_count1} we may 
let $f(m)= g(m) + (\cTwelve - 19)n^{3/2}$. Then by Theorem~\ref{thm:BoundedOrbits} 
\begin{align*}
    |\Sub_{\mathrm{sol, \mathcal{Q}}}(\Sn_n)| & \le 2^{\epsilon n^2 +f(n)+(5 + \gamma)n\log{n}} 
     \le 2^{\epsilon n^2 + g(n) + (\cTwelve - 19)n^{3/2}+ 18n \log n}  & \le 2^{\epsilon n^2 + g(n) + (\cTwelve - 1)n^{3/2}}.
\end{align*}
The result now follows from Theorem~\ref{thm:AschGur}.
\end{proof}

It is now a straightforward matter to prove Theorem~\ref{thm:PyberConj}.

\begin{proof}[Proof of Theorem~\ref{thm:PyberConj}]
The lower bound is Proposition~\ref{prop:H}. We shall prove the upper bound with $\beta = \gamma(2) + \cTwelve(1/16)$.
Let $\mathcal{Q}$ be tautological, let $\epsilon  = 1/16$ and let $C$ be as in Proposition~\ref{prop:count_subgroups}. By Theorem~\ref{thm:Pybernilp}, there are at most $2^{n^2/16 + \gamma(2) n \log n}$ nilpotent subgroups of $\Sn_n$, so this certainly bounds the number of nilpotent subgroups of $\Sn_n$ with all orbits of length less than $C$. We let $g(n) = \gamma(2) n \log n$. The result is now immediate from Proposition~\ref{prop:count_subgroups}. 
\end{proof}

The proof of Theorem~\ref{thm:randSyl2} is also now an easy exercise. We split part of the proof into a separate lemma for use in the next section. Throughout this and the next section, we shall write $d_2(G)$ for the number of generators of a Sylow $2$-subgroup of $G$.

\begin{Lemma}\label{lem:mu_count}
    Let $\mu \ge (\log 3)/36$ be a real number. Then the number of nilpotent subgroups of $\Sn_n$ with a Sylow $2$-subgroup requiring at most $\mu n$ generators is at most $2^{\mu n^2 + (\gamma(3) +5)n \log n}$.

    Hence for all $\mu < 1/16$, a random nilpotent subgroup of $\Sn_n$ has a Sylow $2$-subgroup that requires at least $\mu n$ generators.
\end{Lemma}

\begin{proof}
Fix a partition $n = n_1 + \cdots + n_t$ of $n$, and let $H_i$ be maximal transitive nilpotent subgroup of $\Sn_{n_i}$ for each $i \in \{1, \ldots, t\}$. Let $d_1$ be the sum of the $2$-parts of $n_1, n_2, \ldots, n_t$ which are greater than $1$, and let $d_2$ be the sum of the $2'$-parts which are greater than $1$. By applying Lemma~\ref{lem:nilp_degree}(iii) to each orbit, we see that $H_1\times\hdots\times H_t$ is isomorphic to a direct product of a $2$-subgroup of $\Sn_{d_1}$ and an odd order nilpotent (i.e. $\mathcal{N}_3-$) subgroup of $\Sn_{d_2}$. If $S$ is a Sylow $2$-subgroup of $\Sn_{d_1}$ then $|\Sub_{d \le \mu n}(S)| \le  |S|^{\mu n} \le 2^{\mu d_1 n}$. By Theorem~\ref{thm:Pybernilp} the number of $\mathcal{N}_3$-subgroups of $\Sn_{d_2}$ is at most $2^{d_2^2 \log 3/36 + \gamma(3) d_2 \log d_2}$. Since $d_1 + d_2 \le n$, 
\begin{align*}
 |\Sub_{\mathrm{nilp}, d_2 \le \mu n}(H_1 \times \cdots \times H_t)| & \le 2^{\mu d_1 n + (n-d_1)^2 (\log 3)/36 + \gamma(3) d_2 \log d_2}\le 2^{\mu (d_1 n + (n-d_1)^2) + \gamma(3) d_2 \log d_2} \\&  \le 2^{\mu n^2  + \gamma(3) n \log n}.   
\end{align*}

By Lemma~\ref{lem:nilp_degree}(iii) the groups $H_i$ are determined up to $\Sn_{n_i}$-conjugacy by the value of $n_i$, so $H_1 \times \cdots \times H_t$ is determined up to $\Sn_n$-conjugacy by the partition $n = n_1 + \cdots + n_t$. There are at most $2^{4n}$ partitions of the integer $n$ (see \cite[p172]{Ayoub}). Hence
\[|\Sub_{\mathrm{nilp}, d_2 \le \mu n}(\Sn_n)| \leq 2^{\mu n^2  + \gamma(3) n \log n + n \log n + 4n}.\]

 For the final claim, there are at least as many nilpotent groups as $2$-groups, so at least $2^{n^2/16}$ by Theorem \ref{thm:Pyberp}. Therefore $|\Sub_{\text{nilp},d_2<\mu n}(\Sn_n)|/|\Sub_{\text{nilp}}(\Sn_n)|\rightarrow 0$ as $n\rightarrow \infty$, as required.
\end{proof}

\begin{proof}[Proof of Theorem~\ref{thm:randSyl2}]
We first prove that a random subgroup of $\Sn_n$ has an elementary abelian $2$-section of order at least $2^{\mu n}$. Assume for convenience that $\mu \ge (\log 3)/36$, since if the result holds for this  $\mu$ then it holds for smaller $\mu$. 
Let $\mathcal{Q}$ be the property of having \emph{no} elementary abelian $2$-section this large:  $\mathcal{Q}$ is closed under taking subgroups and quotients. 

Each group $G \in \mathcal{Q}$ satisfies $d_2(G) \le \mu n$, so by Lemma~\ref{lem:mu_count} there are at most $2^{\mu n^2 + (\gamma(3) +5) n \log n}$ nilpotent $\mathcal{Q}$-subgroups of $\Sn_n$.  
Let $\epsilon = \mu$ and $g(n) =  (\gamma(3) + 5 )n \log n$. Then by Proposition~\ref{prop:count_subgroups},  there are at most $2^{\mu n^2 + o(n^2)}$ $\mathcal{Q}$-subgroups of $\Sn_n$. Since $\mu < 1/16$ the result follows from the lower bound in Theorem~\ref{thm:PyberConj}. 

The argument for the groups with a Sylow $2$-subgroup of order at most $2^{\nu n}$ is similar: we let $\mathcal{R}$ be the property of having a \emph{smaller} Sylow $2$-subgroup, and use Lemma~\ref{lem:AnerCount}(i) to see that the number of $\mathcal{R}$-subgroups of a Sylow $2$-subgroup  of $\Sn_n$ is at most $4 \cdot 2^{\nu n (n-\nu n)}$. Hence a bound similar to Lemma~\ref{lem:mu_count} on the number of nilpotent $\mathcal{R}$-subgroups follows easily. Since $\nu (1 - \nu)  <1/16$, the result follows in the same way as the previous case. 
\end{proof}

\section{The proof of Theorems~\ref{thm:not_Kantor}  and \ref{thm:rand_nilp}}

Now that we have Theorems~\ref{thm:Pyberp} and \ref{thm:NilpotentRedTheorem} at our disposal,  we can prove our remaining results. We begin with four lemmas, then prove Theorem~\ref{thm:rand_nilp} and finally Theorem~\ref{thm:not_Kantor}. 
 Let $\mathcal{P}$ be the property of being nilpotent and not a $2$-group. Recall Definition~\ref{def:gridded} of a pair of gridded partitions: we now generalise this to intransitive groups. 
 We will use the following notation throughout the section.

\begin{Definition}\label{def:calPgridded}
A \emph{$\mathcal{P}$-gridded partition of $\Omega$} is an ordered pair $\Sigma = (\Sigma_2, \Sigma_{2'})$ of partitions, both refinements of some partition $\Lambda$ of some subset $\Xi (=\Xi(\Sigma))$ of $\Omega$, such that all parts of $\Sigma_{2'}$ have size greater than $1$, and each restriction $\Sigma|_{\Lambda_i} := (\{C \in \Sigma_2 \text{ : } C \subseteq \Lambda_i\}, \{\Delta \in \Sigma_{2'} \text{ : } \Delta \subseteq \Lambda_i\})$  
is a pair of $2$-gridded partitions.

Since $\Sigma|_{\Lambda_i}$ is $2$-gridded for each $i$, the sets $\Delta$ in $\{\Delta \in \Sigma_2 \text{ : } \Delta \subseteq \Lambda_i\}$ have a common size $a_i$, and the sets $\{\Delta \in \Sigma_{2'} \text{ : }\Delta \subseteq \Lambda_i\}$ have a common size $b_i$ (assumed to be greater than one). If $\Lambda$ has $s$ parts, then we shall index the $\Lambda_i$ so that $a_1, \ldots, a_t > 1$ for some $t \le s$, whilst $a_i = 1$ for $i > t$, and we let $a(\Sigma) = \sum_{i = 1}^t a_i$ and $b(\Sigma)= \sum_{i = 1}^s b_i$.
We also let $\Theta =\{1,\hdots,n\}\setminus \Xi$, and let $e = |\Theta|$ and $f= |\Xi|$. 
\end{Definition}

The following lemma, the proof of which is left as an exercise, shows where $\mathcal{P}$-gridded partitions come into our arguments.
 
\begin{Lemma}\label{lem:NewMixedGroupProps}
\begin{enumerate}[\upshape(i)]
\item Let $G$ be a $\mathcal{P}$-subgroup of $\Sn_n$, let $\Xi$ be the union of the $G$-orbits of size divisible by at least one odd prime, let $Q = G|_{\Xi}$, and let $\Sigma_2$ and $\Sigma_{2'}$ be the orbit partitions of $Q_2$ and $Q_{2'}$ respectively. Then $\Sigma(G):= (\Sigma_2,\Sigma_{2'})$ is a $\mathcal{P}$-gridded partition of $\{1, \ldots, n\}$. 
\item[\upshape(ii)] Let $(\Sigma_2, \Sigma_2')$ be a $\mathcal{P}$-gridded partition of a set $\Omega$. Then  $$|\Xi| = f = \sum_{i = 1}^s a_i b_i, \ a \leq \frac{1}{3}f, \ a+b \le f, \ \mbox{ and } (a+1)b \ge f.$$
\end{enumerate}
\end{Lemma}

For a $\mathcal{P}$-subgroup $G$ of $\Sn_n$,  we shall slightly abuse notation and write $\Xi=\Xi(G)$, $a = a(\Sigma(G))=a(G)$, etc. 
The  following correspondence is clear from Definition \ref{def:calPgridded}. 
\begin{Lemma}\label{lem:nilp_degree2}
Let $\Sigma$ be a $\mathcal{P}$-gridded partition of $\{1, \ldots, n\}$. There is a natural 
injection $\Psi$ from the set of $\mathcal{P}$-subgroups  $G$ of $\Sn_n$ with $\Sigma = \Sigma(G)$ to the set of nilpotent subgroups of $\Sym(\Theta)\times \Sn_a\times\Sn_b$ that project onto a $2$-subgroup of $\Sym(\Theta)\times \Sn_a$ and an odd order subgroup of $\Sn_b$.   
\end{Lemma}

We shall use the following notation throughout this section. 
For $\Xi \subseteq \{1, \ldots, n\}$, and  $a,b \in \mathbb{Z}_{\ge 0}$, write $\Sub_{\text{mix},a,b}(\Xi)$ for the set of those nilpotent $Q \le \Sym(\Xi)$ having no orbits of $2$-power size, and satisfying  $a(Q)=a$ and $b(Q)=b$.  Notice that if $\Xi \neq \emptyset$ then each such $Q$ is a $\mathcal{P}$-group. 

\begin{Lemma}\label{lem:fw}
\begin{enumerate}
    \item[\upshape(i)]  There exists an absolute constant $\tau_0$ such that the following holds. Fix $\Xi \subseteq \{1, \ldots, n\}$, and  $a,b \in \mathbb{Z}_{\ge 0}$, and write $f=|\Xi|$. Then $|\Sub_{\text{mix},a,b}(\Xi)| \leq 2^{a^2/16 + b^2/36 + \tau_0f \log f}$. 
\item[\upshape(ii)] There exists an absolute 
constant $\tau>0$  such that for all $n$ and all subsets $\Gamma$ of $\{(e, a, b) \in \mathbb{Z}_{\ge 0}^3 \text{ : } e+a+b \leq n\}$,  the number of $\mathcal{P}$-subgroups $G$ of $\Sn_n$ with $(e(G), a(G), b(G)) \in \Gamma$ is at most
  \[
    \max_{(e,a,b) \in \Gamma}2^{\frac{(e+a)^2}{16}+\frac{b^2}{36}+\tau n\log{n}}.\] 
\end{enumerate}
\end{Lemma}

\begin{proof}
Let $\mathcal{R}$ be the set of all  
$\mathcal{P}$-gridded partitions of $\{1, \ldots, n\}$. Then $|\mathcal{R}|$ is less than the square of the number of partitions of $\{1, \ldots, n\}$, which is less than
$2^{2n \log n}$. 

\noindent (i). For each $\mathcal{P}$-gridded partition $\Sigma$ with 
$\Xi(\Sigma) = \Xi$, $a(\Sigma) = a$ and $b(\Sigma) = b$,  Lemma~\ref{lem:nilp_degree2} shows that there is an injection from the set of  $\mathcal{P}$-subgroups of $\Sn_n$ with $\mathcal{P}$-gridded partition $\Sigma$ to the set of subgroups of $\Sn_{a} \times \Sn_{b}$ that project onto a $2$-group  in $\Sn_{a}$, and an odd order nilpotent group (a group in $\mathcal{N}_3$) in $\Sn_b$. Conversely, by  Lemma~\ref{lem:NewMixedGroupProps}(i), each  such group $G$ determines a unique $\Sigma \in \mathcal{R
}$. We apply Theorem~\ref{thm:Pyberp} to the $2$-subgroups of $\Sn_{a}$ and Theorem~\ref{thm:Pybernilp} to the $\mathcal{N}_3$-subgroups of $\Sn_{b}$ to see that the result follows with $\tau_0= \max\{\beta_2, \gamma(3)\}> 0$.

\smallskip 

\noindent (ii). We apply an almost identical argument, except that we include an $e$ term and adjust the constant to account for the choice of the subset $\Theta$ of size $e$. 
\end{proof}

For groups $P$ and $X$ and an integer $d$, we write $\Subdir_{\ge d}(P \times X)$ for the set of subdirect products of $P \times X$ that require at least $d$ generators.  Our final preliminary result is a technical lemma counting certain useful sets of $2$-groups. 

\begin{Lemma}\label{lem:constants2}
Let $P$, $X$ and $E$ be finite $2$-groups, with $E$ elementary abelian, and let $d$ be a positive integer. Let $\mathcal{S}(P,X,E, d)$ be the set of subdirect products of $P\times X\times E$ with the property that the projection to $P\times X$ requires at least $d$ generators and the projection to $X\times E$ is surjective. If $d-d(X)>\max\{d(E),1\}$ then
$$|\mathcal{S}(P,X,E,d)|\geq 2^{(d-d(X))d(E)/2}|\Subdir_{\ge d}(P\times X)|.$$
\end{Lemma}

\begin{proof}
Let $H$ be a subdirect product of $P \times X$ with $d(H) \ge d$. Choose a subgroup $L$ of $H$ that is minimal subject to $(H\cap P)L=H$. Writing bars to denote reduction modulo $\Phi(H)$, there exists a subgroup $A$ of $H$ such that the Frattini subgroup $\Phi(H)\le A$ and $\ol{A}$ is a complement for $\ol{L}$ in the elementary abelian group $\ol{H}$.

 Next,
notice that $d(\ol{A})+d(\ol{L}) = d(\ol{H}) \ge d$, so $d(\ol{A}) \ge  d-d(\ol{L})$. The minimality of $L$ implies that $H\cap P\cap L\le \Phi(L)$, so that $P\cap L\le \Phi(L)$,  and hence $d(L)=d(L/P\cap L)=d(X)$. Thus, $d(\ol{L})\le d(L)=d(X)$, so  $d(\ol{A})\geq d-d(\ol{L})\geq d-d(X)$.
Identifying $\ol{A}$ with $\BBF_2^{d(\ol{A})}$ and $E$ with $\BBF_2^{d(E)}$, since $d(\ol{A}) > d(E)$ we  see that $|\Epi(\ol{A},E)|$ is at least the product of $|\GL_{d(E)}(2)|$ and the number of subspaces of $\BBF_2^{d(\ol{A})}$ of dimension $d(E)$. We deduce from Lemma~\ref{lem:AnerCount}(ii)-(iii), and the assumption that $d-d(X) \ge 2$, first that $|\GL_{d(E)}(2)|\geq 2^{d(E)^2-d(E)}$ and then that
\[|\Epi(\ol{A},E)|\geq 2^{d(E)^2-d(E)}2^{d(E)((d-d(X))-d(E))}= 2^{(d-d(X)-1)d(E)}\ge 2^{(d-d(X))d(E)/2}.\]

Let $\alpha: \ol{A} \rightarrow E$ be  one such epimomorphism. We define a homomorphism $\theta_{\alpha}:H\rightarrow E$  by $\theta_{\alpha}(a \ell)=\alpha(\ol{a})$ for $a \in A$ and $\ell \in  L$, which is well-defined since $A \cap L \le \Phi(H) \le \ker(\theta_{\alpha})$.  Notice that $L\le \ker(\theta_{\alpha})$, so $\ker(\theta_{\alpha})$ projects onto $X$, and that $\theta_{\alpha}=\theta_{\beta}$ if and only if $\alpha=\beta$. Since $\theta_{\alpha}$ is surjective, this defines an injection  $f:\Epi(\ol{A}, E)\rightarrow \Epi(H, E)$, $\alpha\rightarrow \theta_{\alpha}$.

For each $\theta\in \mathrm{Im}(f)$, 
define 
$G_{\theta}:=\{(p, x, (p,x)\theta)\text{ : }(p,x)\in H\} \leq P \times X \times E$. 
It is clear that $G_{\theta}=G_{\tau}$ if and only if $\theta=\tau$. By definition $G_{\theta}$ projects onto $H$, and since $\ker(\theta)$ projects onto $X$ the group $G_{\theta}$ projects onto $X \times E$.  Hence $G_{\theta} 
\in \mathcal{S}(P, X, E, d)$ so the number of subgroups in $\mathcal{S}(P, X, E, d)$ that project onto $H$ is at least the number of choices of $\theta$, which is at least $|\mathrm{Im}(f)|=|\Epi(\ol{A}, E)| \ge 2^{(d-d(X))d(E)/2}$.

The result now follows from the fact that there are $|\Subdir_{\ge d}(P \times X)|$ choices for $H$, and distinct choices of $H$ yield distinct subgroups.
\end{proof}

We are now ready to prove Theorem \ref{thm:rand_nilp}. 

\begin{proof}[Proof of Theorem~\ref{thm:rand_nilp}]
Each nilpotent subgroup of $\Sn_n$ is either a $2$-group or a $\mathcal{P}$-group. We shall show that for sufficiently large $n$ there are exponentially more $2$-subgroups of $\Sn_n$ than $\mathcal{P}$-groups, so that the probability that a random nilpotent subgroup of $\Sn_n$ is a $2$-group tends to $1$.

\medskip

Let $\tau>0$ be as in Lemma~\ref{lem:fw}.
First we count the $\mathcal{P}$-groups $G$ such that $a:=a(G)>n/10^6$ or $b:=b(G)>9\tau\log{n}$. 
Suppose first that $a>n/10^6$. Then $|\Xi|\geq 3a$ by Lemma~\ref{lem:NewMixedGroupProps}(ii),
so  $e+a \leq (n - 3a) + a < n -2n/10^6$, which is less than $n - 9 \tau \log n$ for $n$ sufficiently large. Similarly, if $b> 9\tau\log{n}$, then $e+a \le n-b <  n-9\tau \log{n}$. Thus we may assume in both cases that $e+a <  n-9\tau \log{n}$.

Applying Lemma~\ref{lem:fw}(ii) to the set $\Gamma$ of all such $(e, a, b)$ with $e + a\le n-9\tau\log{n}$ yields  a maximum at 
$e + a = n - 9\tau \log{n}$, so for $n$ sufficiently large the number of $\mathcal{P}$-groups $G$ for which $a >  n/(10^6)$ or $b > 9 \tau \log n$ is at most
$$2^{(n-9\tau\log{n})^2/16+(9\tau\log{n})^2/36+\tau n\log{n}}\le 2^{n^2/16-(\tau/8)n \log{n} + (117\tau^2/16)(\log{n})^2},$$ 
for $n$ sufficiently large this is at most $2^{n^2/16-(\tau/9) n\log{n}}$.  By Theorem~\ref{thm:Pyberp} the number of $2$-subgroups of $\Sn_n$ is greater than  $2^{n^2/16}$, so a random nilpotent subgroup of $\Sn_n$ is either a $2$-group or a $\mathcal{P}$-group $G$ such that $a(G)\le n/10^6$ and  $b(G) \le 9\tau\log{n}$. 

\medskip

 Next, by Lemma~\ref{lem:mu_count}, a random nilpotent subgroup of $\Sn_n$ has a Sylow $2$-subgroup requiring more than $n/17$ generators. Hence such a subgroup is either a $2$-group, or lies in the set
\[\mathcal{S} = \{ G \le \Sn_n \text{ : } G \mbox{ a $\mathcal{P}$-group,} \  d_2(G) \ge n/17, \ 
a(G)\le n/(10^6) \text{ and } b(G)  \leq 9\tau\log{n} 
 \}\]  
 The result will follow if we can show that the ratio $|\mathcal{S}|/|\Sub_2(\Sn_n)|$ tends to zero as $n \rightarrow \infty$.

\medskip

For $\Xi \subseteq \{1,\hdots,n\}$, and integers $a \in \{0, \ldots, 
 \lfloor  n/(10^6) \rfloor\}$ 
and $b \in  \{3, \ldots, \lfloor 9 \tau \log n \rfloor\}$,
let \[\mathcal{S}(\Xi,a,b) =  \{ G\in \mathcal{S} \text{ : } \Xi(G) =\Xi, \ a(G) =a \text{ and }b(G) =b\},\] 
so that $\mathcal{S}$ is the disjoint union of the sets $\mathcal{S}(\Xi, a, b)$.
The main work of the proof will be to bound  $|\mathcal{S}(\Xi,a,b)|$.

\bigskip

If $G \in \mathcal{S}(\Xi,a,b)$, then the restriction $Q$ of $G$ to $\Xi$ lies in $\Sub_{\text{mix},a,b}(\Xi)$. 
By construction, the Sylow $2$-subgroup $Q_2$ of
each $Q\in \Sub_{\text{mix},a,b}(\Xi)$ embeds in $\Sn_a$. Thus, every subgroup of $Q_2$ requires at most $a/2$ generators by Theorem \ref{thm:KovPrae}, and in particular
$d(Q_2)\le a/2$.  
It follows that  each $2$-subgroup $P$ of $\Sym(\Theta)$ such that $P\times Q_2$ has a subdirect product requiring at least $n/17$ generators satisfies $d(P)\geq m:= n/17-a/2$.
Each $G \in \mathcal{S}(\Xi,a,b)$ satisfies $d_2(G) \ge n/17$, so we may bound
\[\begin{array}{rl}
  |\mathcal{S}(\Xi, a, b)|& \le \sum_{P\in \Sub_{2, d \ge m}(\Sym(\Theta)), Q\in\Sub_{\text{mix},a,b}(\Xi)}|\Subdir_{d_2 \ge n/17}(P\times Q)|\\
  & =\sum_{P\in\Sub_{2, d \ge m}(\Sym(\Theta)), Q\in\Sub_{\text{mix},a,b}(\Xi)}  |\Subdir_{d_2\geq n/17}(P\times Q_2)|
\end{array}\]
where $\Subdir_{d_2 \ge n/17}(P\times Q)$ is the set of subdirect products of $P\times Q$ in which the Sylow $2$-subgroup requires at least $n/17$ generators. 

To bound the size of these sets of subdirect products, first let $T(\Xi,a,b)$ be the set of $2$-subgroups $X$ of $\Sym(\Xi)$ that occur as Sylow $2$-subgroups of groups  $Q \in \Sub_{\mathrm{mix}, a, b}(\Xi)$, and for each $X \in T(\Xi, a, b)$ let  $\mathcal{Q}(X,a,b)$ be the (non-empty) set of groups $Q\in\Sub_{\mathrm{mix},a,b}(\Xi)$ with $Q_2 = X$. Since each such $Q$ is a nilpotent group with no orbits of $2$-power length,
$\mathcal{Q}(X,a,b)$ consists of all groups $X\times L$ such that $L$ 
is a fixed-point-free $\mathcal{N}_3$-subgroup of $C:= C_{\Sym(\Xi)}(X)$. 
For some indexing of $a_1, \ldots, a_t$, the centraliser
 $C$ factorises as $W_1\times\cdots\times W_k \times \Sn_{c}$, where each $W_i$ is permutation isomorphic to $R_i\wr \Sn_{c_i}$ for some semiregular $2$-group $R_i$ of degree $a_i > 1$,  and some positive integers $c_i\geq 3$ and $c\geq  0$ with $c+\sum_{i=1}^kc_i=b$ and $c+\sum_{i=1}^ka_ic_i=f=|\Xi|$. (If the Sylow $2$-subgroup of $G$ acts equivalently on more than one $G$-orbit then the sum $\sum_{i=1}^ka_i$ will be less than $a$.) Let  $U \le C$ be the direct product of the base groups of the $W_i$. Then $|U|=|R_1^{c_1}\times\cdots\times R_{k}^{c_k}|\le \prod_{i=1}^k a_i^{c_i} \le  2^{\sum c_i a_i} \le 2^{f}$. 
 Fix a complement $B$ of $U$ in $C$ so that 
$B\cong \Sn_{c_1}\times\cdots\times \Sn_{c_k} \times \Sn_{c} \le \Sn_b$. Since $U$ is a $2$-group, by the Schur-Zassenhaus theorem each $\mathcal{N}_3$-subgroup  of $C = U \cn B$ is $U$-conjugate
to a subgroup of $B$, and so $|\Sub_{\mathcal{N}_3}(C)| \leq 2^{f} |\Sub_{\mathcal{N}_3}(B)|  \le 2^f |\Sub_{\mathcal{N}_3}(\Sn_b)|$.
%
 Hence
Theorem \ref{thm:Pybernilp}  yields
\begin{align*}\label{lab:OX}
|\mathcal{Q}(X,a,b)|\le |\Sub_{\mathcal{N}_3}(C)|  \le 2^{f} |\Sub_{\mathcal{N}_3}(\Sn_b)|  \le 2^{b^2/36+\gamma(3)b\log{b}+f}.   
\end{align*}
Since $9 \tau \log n \ge b$ there exists an absolute constant $\lambda$ such that
$b^2/36+\gamma(3)b\log{b}\le \lambda (\log{n})^2$. 
Substituting, we deduce that 
\begin{align*}
     |\mathcal{S}(\Xi, a, b)| &  \le \sum_{P\in\Sub_{2, d \ge m}(\Sym(\Theta)), Q\in\Sub_{\text{mix},a,b}(\Xi)}  |\Subdir_{d_2\geq n/17}(P\times Q_2)|\\
     & \le \sum_{P\in\Sub_{2, d \ge m}(\Sym(\Theta)),X\in T(\Xi,a,b)}  |\mathcal{Q}(X,a,b)||\Subdir_{d_2\geq n/17}(P\times X)|\\
& \leq  2^{\lambda (\log{n})^2+f}\sum_{P,X} |\Subdir_{d_2\geq n/17}(P\times X)|.
\end{align*}

For each $X \in T(\Xi, a, b)$,  each $c_i$ is at least $3$, so the corresponding $B \le \Sym(\Xi)$ has at least one elementary abelian $2$-subgroup  of order exactly $2^{\lfloor b/3\rfloor}$:  fix one and denote it $E_X$. Since $E_X \leq B \leq C_{\Sym(\Xi)}(X)$ and $X\cap  C_{\Sym(\Xi)}(X)\le U$, the intersection $E_X \cap X = 1$, so 
$XE_X \cong X\times E_X$ is a $2$-subgroup of $\Sym(\Xi)$.  Also, $d(X) \le a/2 \le n/(2 \cdot 10^6)$ so
$d(E_X)=\lfloor b/3\rfloor \le 3 \tau \log n < n/17-a/2\le n/17-d(X)$ for large enough $n$.
Since the group $X$ uniquely determines $E_X$, for the rest of this proof we shall write $\mathcal{S}(P, X)$ instead of $\mathcal{S}(P, X, E_X, n/17)$ for the set of subdirect products of $P \times X \times E_X$ whose projection to $P \times X$ requires at least $n/17$ generators and whose projection to $X \times E_X$ is surjective. 
With this notation,  for sufficiently large $n$  Lemma~\ref{lem:constants2}  bounds 
\begin{align*}
 |\mathcal{S}(\Xi, a, b)| & \leq  2^{\lambda (\log{n})^2+f}
 \sum_{P, X} 2^{-((n/17) -d(X))d(E_X)/2}|\mathcal{S}(P,X)|\\
&\leq  2^{\lambda (\log{n})^2+f+(a/4-n/34)\lfloor b/3\rfloor} \sum_{P,X} |\mathcal{S}(P,X)|.
\end{align*}
We remark that $a/4-n/34 <-0.029n$.

For distinct $P$ and $X$, the sets $\mathcal{S}(P, X)$ may not be disjoint. We define an equivalence relation on the sets $\mathcal{S}(P, X)$ by $\mathcal{S}(P, X) \sim \mathcal{S}(R, Y)$ if and only if $P = R$ and $X \times E_X = Y \times E_Y$. 
We now bound the size  of the equivalence class $[\mathcal{S}(P, X)]$,  
so let $\mathcal{S}(R,Y)\in[\mathcal{S}(P,X)]$. Then 
$E_{Y}\le Z(X\times E_{X})=Z(X)\times E_{X}$. Since $X$ is isomorphic to a subgroup of $\Sn_a$, Theorem~\ref{thm:KovPrae}  bounds $|Z(X)|\le 2^{a/2}$.  By assumption $|E_{X}|=|E_{Y}| = 2^{\lfloor b/3\rfloor}$, so $|Z(X) \times E_{X}| \le 2^{a/2 + \lfloor b/3 \rfloor}$. Hence by Lemma~\ref{lem:AnerCount}(i) there are at most $4 \cdot 2^{\lfloor b/3 \rfloor (a/2)}$  groups $E_{Y}$. 
Now let 
$$f(E_Y)=\{\mathcal{S}(P,Y')\in[\mathcal{S}(P,X)]\text{ : } E_{Y'}=E_Y\} \subseteq [\mathcal{S}(P, X)].$$
From $d(E_Y)=d(E_{X})$ and the fact that each group in $\mathcal{S}(P, X)$ projects to $X \times E_X \le \Sym(\Xi)$, we deduce that $d(Y)=d(X)=d(Y') \le a/2$ for all $\mathcal{S}(P,Y')\in f(E_Y)$. 
Each complement $Y'$ to $E_Y$ in $Y\times E_Y$ has order $|Y|$ and projects to $Y$, so has the form $\{( y,y\theta)\text{ : } y\in Y\}$ for some homomorphism $\theta:Y\rightarrow E_Y$. Hence $|f(E_Y)|\le |\Hom(Y,E_Y)|\le |E_Y|^{d(Y)}\le 2^{\lfloor b/3\rfloor a/2}$. We have therefore shown that $|[\mathcal{S}(P, X)]|\le  4 \cdot 2^{\lfloor b/3\rfloor a/2}2^{\lfloor b/3\rfloor a/2}=2^{a\lfloor b/3\rfloor + 2}$.

An  easy exercise shows that for a fixed $\Xi$ the intersection  $\mathcal{S}(P,X)\cap \mathcal{S}(R,Y)$ is non-empty if and only if  $\mathcal{S}(P,X)\sim \mathcal{S}(R,Y)$.
To bound 
$\sum_{P, X}|\mathcal{S}(P, X)|$,
let $\mathcal{C}_1:= \mathcal{S}(P_1, X_1)$, $\ldots$, $\mathcal{C}_l:= \mathcal{S}(P_\ell, X_{\ell})$ be a complete set of $\sim$-class representatives, chosen such that each is of  maximal size in its $\sim$-class. Then the sets $\mathcal{C}_i$ are pairwise disjoint, so 
\begin{align*}
 \sum_{P\in\Sub_{2, d \ge m}(\Sym(\Theta)),X\in T(\Xi,a,b)}|\mathcal{S}(P,X)| 
 & \le \sum_{i=1}^\ell|[\mathcal{C}_i]||\mathcal{C}_i|\\
& \le \max_{i \in \{1, \ldots, \ell\}}|[\mathcal{C}_i]|\sum_{i=1}^{\ell}|\mathcal{C}_i| 
  \le \max_i|[\mathcal{C}_i]||\bigcup_{P, X}\mathcal{S}(P,X)| \\
  & \le 2^{a\lfloor b/3\rfloor + 2}|\bigcup_{P, X} \mathcal{S}(P, X)|
  \le 2^{a\lfloor b/3\rfloor + 2}|\Sub_2(\Sn_n)|.
\end{align*}
Since 
$a \le n/10^6$ we can bound $5a/4 - n/34 \le -0.029n$. Hence
\begin{align*}
|\mathcal{S}(\Xi, a, b)| 
& \leq  2^{\lambda (\log{n})^2+f + (5a/4-n/34)\lfloor b/3\rfloor + 2}|{\Sub}_2(\Sn_n)|
 \le 2^{\lambda (\log{n})^2 +f - 0.029n \lfloor b/3\rfloor + 2}|{\Sub}_2(\Sn_n)|.
 \end{align*}

Let $\mathcal{S}'$ denote the subset of $\mathcal{S}$ for which $f > n/1000$.
Now  $(a+1)b \ge f$ by Lemma~\ref{lem:NewMixedGroupProps}(ii), so $a \le n/10^6$ implies that $b > 500$ for $n$ sufficiently large, and hence 
$-0.029n \cdot \lfloor b/3 \rfloor +2 \le -0.029 n \cdot 166 + 2 \le  -4.8n$. Furthermore, $f \le n$ so 
\begin{align*}
  |\mathcal{S}'|& \le \sum_{ \sum_{f > n/1000}, |\Xi| = f, }a+b \le f|\mathcal{S}(\Xi, a, b)|
    \le \sum_{f > n/1000, |\Xi| = f, a+b\le f} 2^{-3.8n+\lambda (\log{n})^2}|\Sub_2(\Sn_n)|\\
   & \le 2^nn^3 2^{-3.8n+\lambda (\log{n})^2}|\Sub_2(\Sn_n)|.
\end{align*}

Now we count the subset of $\mathcal{S}$ for which $f \le  n/1000$. It is a standard fact (see for example \cite[Lemma~16.19, page 427]{FG}), that the total number of subsets $\Xi$ of $\{1, \ldots, n\}$ of size at most $n/1000$ is at most
$$\sum_{i = 0}^{\lfloor n/1000\rfloor} \binom{n}{i} \le 2^{n(\frac{1}{1000}\log{1000}-\frac{999}{1000}\log{\frac{999}{1000}})}<2^{0.012n}.$$
From $b \ge 3$ we deduce that 
$-0.029n \cdot \lfloor b/3 \rfloor + 2 \le -0.028 n$ for $n$ sufficiently large. 
Hence 
\begin{align*} |\mathcal{S} \setminus \mathcal{S}'| & \le \sum_{f \le n/1000,|\Xi| = f, a+b \le f}|\mathcal{S}(\Xi, a, b)|
    \le \sum_{f \le n/1000,|\Xi| = f, a+b \le f} 2^{-0.028n + 0.001n}2^{\lambda (\log{n})^2}|\Sub_2(\Sn_n)|\\
 & \le  2^{0.012n} \cdot  n^3 \cdot 2^{-0.027n + \lambda (\log{n})^2}|\Sub_2(\Sn_n)|.
 \end{align*}
Thus, $|\mathcal{S}|/|\Sub_2(\Sn_n)|$ tends to $0$ as $n$ tends to $\infty$,  and the result follows. 
\end{proof}

\begin{proof}[Proof of Theorem~\ref{thm:not_Kantor}]
Let $n$ be congruent to $3 \mod 4$. We shall show that there is a map $f$ from the set of $2$-subgroups of $\Sn_n$ to the set of non-nilpotent groups of $\Sn_n$, such that $f$ maps at most four $2$-subgroups to the same non-nilpotent subgroup. Hence, the number of non-nilpotent subgroups is at least $1/4$ of the number of $2$-groups. The result will then follow from Theorem~\ref{thm:rand_nilp}.

Let $G$ be a $2$-subgroup of $\Sn_n$. Since $n \equiv 3 \mod 4$, the group $G$ either has at least $3$ fixed points, or a single fixed point and at least one orbit of length two. In the first case, let $i, j, k$ be the smallest fixed points (with respect to the natural ordering on $[n]$), and in the latter let $i$ be the fixed point and $j, k$ be the two points in the first orbit of length $2$. We then define $f(G)=\langle G,(i,j,k),(i,j)\rangle$ in the first case, and $f(G)=\langle G,(i,j,k)\rangle$ in the second case. Then $H:=f(G)$ induces $\Sn_3$ on $\Delta:= \{i, j, k\}$; the pointwise stabiliser  $H_{(\Delta)}$  is equal to $G_{(\Delta)}$; and the projection of $H$ to $\{1, \ldots, n\} \setminus \Delta$ is the same the projection of $G$ to this union of orbits.

To see that each group in $\mathrm{Im}(f)$ corresponds to at most four groups $G$, let $H$ be a group in $\mathrm{Im}(f)$. Then $H$ has a unique  orbit $\Delta = \{i, j, k\}$ on which $H$ acts as $\Sn_3$. Let $\Gamma = \{1, \ldots, n\} \setminus \Delta$.  If $H$ is a direct product of $\Sym(\Delta)$ with $H|_{\Gamma}$, then $H = f(G)$ where $G$ is a direct product of $G|_{\Delta}$ and $G|_{\Gamma}$. The group $G|_{\Gamma}$ is either trivial or $\Sn_2$, and there are at most three choices for the group $\Sn_2$, so at most four such $G$. Otherwise, $H_{(\Delta)}|_{\Gamma} = G_{(\Delta)}|_{\Gamma}$ is an index $2$ subgroup of $G|_{\Gamma}$,  and $H|_{\Gamma} = G|_{\Gamma}$ is also completely determined. There exactly three choices for $G|_{\Delta}$, and so exactly three possible groups $G$. 
\end{proof}

 \paragraph{Acknowledgements}
Both authors would like
to thank the Isaac Newton Institute for Mathematical Sciences, Cambridge,
for support and hospitality during the programme “Groups, Representations
and Applications: New perspectives”, where work on this paper was under-
taken. This work was supported by EPSRC grant no EP/R014604/1, and also
partially supported by a grant from the Simons Foundation. The second author also acknowledges EPSRC grant no EP/T017619/1.
Finally, we would like to thank Yiftach Barnea, Derek Holt, Bill Kantor, L\'{a}zl\'{o} Pyber and Jan--Christoph Schlage--Puchta.

\noindent {\small {\sc Colva M. Roney-Dougal, School of Mathematics and Statistics, St Andrews, Fife KY16 9QQ, UK}, \texttt{colva.roney-dougal@st-andrews.ac.uk}}

\noindent {\small {\sc Gareth Tracey, Mathematics Institute, University of Warwick, Coventry CV4 7AL, UK}, \texttt{G.M.Tracey@warwick.ac.uk}}

\end{document}